\documentclass{article}
\usepackage[latin1]{inputenc}
\usepackage{amsfonts}
\usepackage{amsmath, latexsym}
\usepackage{amssymb,verbatim}
\usepackage{mathtools}
\usepackage{enumitem}

\usepackage[nameinlink,capitalize]{cleveref}

\usepackage{graphicx, psfrag}
\usepackage{tikz}
\usetikzlibrary{arrows,3d,patterns}
\usepackage{color}

\usepackage{euscript}

\usepackage{amsthm}
\theoremstyle{plain}
\newtheorem{theorem}{Theorem}[section]
\newtheorem{lemma}[theorem]{Lemma}
\newtheorem*{lemma*}{Lemma}
\newtheorem{claim}[theorem]{Claim}
\newtheorem{proposition}[theorem]{Proposition}
\newtheorem{corollary}[theorem]{Corollary}

\theoremstyle{definition}
\newtheorem{definition}[theorem]{Definition}
\newtheorem{remark}[theorem]{Remark}
\newtheorem{example}[theorem]{Example}

\theoremstyle{remark}

\numberwithin{equation}{section}

\def\R{\mathbb{R}}

\def\Z{\mathbb{Z}}
\def\N{\mathbb{N}}
\def\C{\mathbb{C}}
\def\T{\mathbb{T}}
\def\S{\mathbb{S}}

\newcommand{\diff}{\mathop{}\mathopen{}\mathrm{d}}

\newcommand\eps{\varepsilon}
\newcommand{\be}{\begin{equation}}
\newcommand{\ee}{\end{equation}}
\newcommand{\nd}{\noindent}
\newcommand{\f}{\varphi}

\title{A DeGiorgi type conjecture for minimal solutions to a nonlinear Stokes equation}

\author{
{\Large Radu Ignat}
\footnote{Institut de Math\'ematiques de Toulouse \& Institut Universitaire de France, UMR 5219, Universit\'e de Toulouse, CNRS, UPS
IMT, F-31062 Toulouse Cedex 9, France. Email: Radu.Ignat@math.univ-toulouse.fr} 
\and {\Large Antonin Monteil }\footnote{Institut de Recherche en Math\'ematique et Physique, Universit\'e catholique de Louvain, \'Ecole de Math\'ematique, Chemin du Cyclotron 2, bte L7.01.02, 1348 Louvain-la-Neuve, Belgium. Email: Antonin.Monteil@uclouvain.be
}
}

\begin{document}
\maketitle
\begin{abstract}
We study the one-dimensional symmetry of solutions to the nonlinear Stokes equation
\[
\begin{cases}
-\Delta u+\nabla W(u)=\nabla p&\text{in }\R^d,\\
\nabla\cdot u=0&\text{in }\R^d,\\
\end{cases}
\]
which are periodic in the $d-1$ last variables (living on the torus $\T^{d-1}$) and globally minimize the corresponding energy in $\Omega=\R\times \T^{d-1}$, i.e.,  
$$E(u)=\int_{\Omega} \frac12 |\nabla u|^2+W(u)\, dx, \quad \nabla\cdot u=0.$$
Namely, we determine a class of nonlinear potentials $W\geq 0$ 
such that any global minimizer $u$ of $E$ connecting two zeros of $W$ as $x_1\to\pm\infty$ is one-dimensional, i.e., $u$ depends only on the $x_1$ variable.
In particular, this class includes in dimension $d=2$  the nonlinearities $W=w^2$ with $w$ being an harmonic function or a solution to the wave equation, while in dimension $d\geq 3$, this class contains a perturbation of the Ginzburg-Landau potential as well as potentials $W$ having $d+1$ wells with prescribed transition cost between the wells. For that, we develop a theory of calibrations relying on the notion of entropy (coming from scalar conservation laws). We also study the problem of the existence of global minimizers of $E$ for general potentials $W$ providing in particular compactness results for uniformly finite energy maps $u$ in $\Omega$ connecting two wells of $W$ as $x_1\to\pm\infty$.
\end{abstract}
\tableofcontents

\section{Introduction}

Let $d\geq 2$ and $\Omega=\R\times \T^{d-1}$ be an infinite cylinder, where $\T=\R/\Z$ is the flat torus.  
The purpose of this paper is to investigate 
the one-dimensional symmetry of divergence-free periodic solutions $u:\Omega\to\R^d$ to the nonlinear Stokes problem
\begin{equation}
\label{stokes}
\begin{cases}
-\Delta u+\nabla W(u)=\nabla p&\text{on }\Omega,\\
\nabla\cdot u=0&\text{on }\Omega,\\
\end{cases}
\end{equation}
where $W:\R^d\to\R_+$ is a nonnegative potential and $p:\Omega\to \R$ is a pressure. 
A solution $u$ of \eqref{stokes} is a critical point of the functional 
\begin{equation}\label{EE}
E(u)= \int_\Omega \frac 12|\nabla u|^2 + W(u) \diff x, \quad u\in\dot{H}_{div}^1(\Omega,\R^d),
\end{equation}
where $|\cdot|$ is the euclidean norm 
and 
\begin{equation*}
\dot{H}_{div}^1(\Omega,\R^d)=\left\{u\in H^1_{loc}(\Omega,\R^d)\;:\;\nabla u\in L^2(\Omega,\R^{d\times d})\text{ and }\nabla\cdot u=0\text{ in }\Omega\right\}.
\end{equation*}
As boundary condition at $x_1=\pm \infty$, we impose that our configurations $u$ connect two wells $u^\pm\in\R^d$ of $W$ where $u^+_1=u^-_1=a\in \R$ (due to the divergence constraint on $u$). Namely, we impose that the $x'$-average of $u$ is a function in $x_1\in \R$ having the following limit at infinity:
\begin{equation}\label{BC}
\lim\limits_{x_1\to \pm \infty} \int_{\T^{d-1}} u(x_1, x')\, \diff x'=u^\pm, \quad W(u^\pm)=0,
\end{equation}
where $x'=(x_2,\dots,x_d)$ denotes the $d-1$ last variables in $\T^{d-1}$.

\medskip

Our aim is to analyze the following De Giorgi type problem for global minimizers of $E$ on $\dot{H}_{div}^1(\Omega,\R^d)$ under the boundary condition \eqref{BC}:

\medskip
\nd \textbf{Question 1 (one-dimensional symmetry of global minimizers):} {\it under which conditions on the potential $W$, is it true that every global minimizer $u$ of $E$ over the set of divergence-free maps satisfying the boundary condition \eqref{BC} is one-dimensional, i.e., $u=u(x_1)$?}

\medskip

We will also discuss the more general existence problem:

\medskip
\nd \textbf{Question 2 (existence of global minimizers):} {\it under which conditions on $W$, does there exist a global minimizer of $E$ on $\dot{H}_{div}^1(\Omega,\R^d)$ 
under the boundary condition \eqref{BC}?}
\medskip

Our study of the one-dimensional symmetry of minimizers uses a kind of calibration method which proves on the one hand that any optimal one-dimensional transition layer $u$ 
connecting $u^\pm$ at $\pm\infty$ is a global minimizer of $E$, and on the other hand that any global minimizer of $E$ on $\dot{H}_{div}^1(\Omega,\R^d)$ under \eqref{BC} is one-dimensional. Thus, this method solves both Question 1 and Question~2. However, the one-dimensional symmetry of global minimizers requires strong assumptions on $W$ whereas the general existence problem in Question 2 only requires generic assumptions on $W$. For this reason, Question 1 and Question 2 are studied independently.

\subsection{Motivation}
The above questions arise naturally in the study of certain phase transition models, in particular, in the theory of liquid crystals or micromagnetics. 

\bigskip

i) In dimension $d=2$, for the Ginzburg-Landau potential $W(u)=\frac14(1-|u|^2)^2$, the system \eqref{stokes} is well known as the Aviles-Giga model (see \cite{Aviles:1987}) that can be seen as a ``baby" model for the stable states in smectic liquid crystals. In the last twenty years, there has been an intensive research on the asymptotic behavior of the rescaled energy  
$$u_\eps\in\dot{H}_{div}^1(\Omega,\R^2) \mapsto \int_\Omega \frac \eps2|\nabla u_\eps|^2 + \frac1 \eps W(u_\eps) \diff x,\quad \Omega\subset \R^2,$$
in the limit $\eps\downarrow 0$. More precisely, it is expected that the limit configurations write $u=(-\partial_2 \f, \partial_1 \f)$ where $\f$ solves the eikonal equation, while the limit energy concentrates on the jump singularities of $u$ and is proportional to the cubic cost $|u^+-u^-|^3$ of each jump of $u$.
This study was carried within the framework of the $\Gamma$-convergence in a series of papers proving compactness results (see \cite{Ambrosio:1999, DKMOcomp, Jabin_perthame}), lower bounds (see \cite{Aviles:1999, Jin:2000}) and upper bounds for $BV$ limit configurations (see \cite{ContiDelellis, Polia}). However, the question of proving the upper bound for general limit configurations $u$ of  finite energy (not necessarily in $BV$) is still open. A challenging problem is the understanding of the structure properties of the limit configurations (see \cite{DeLellis_Otto}) as well as the lower semicontinuity of the limit energy functional for general potentials $W$ (see \cite{Ambrosio:1999,Aviles:1999, Bochard:2015, IgnatMerlet:2012}).

In this theory, a key point relies on the one-dimensional symmetry of minimizing transition layers $u_\eps$ connecting two limit states $u^-, u^+\in \mathbb{S}^1$ (i.e., $W(u^\pm)=0$). A partial result was proved by Jin-Kohn in \cite{Jin:2000}: the optimal one-dimensional transition layer is a global minimizer of the $2D$ energy $E$ on $\dot{H}_{div}^1(\Omega,\R^2)$ under the boundary constraint \eqref{BC}. In this paper, we will prove the complete result, namely, every global $2D$ minimizer of $E$ over $\dot{H}_{div}^1(\Omega,\R^2)$ within \eqref{BC} 
is one-dimensional, meaning that no $2D$ microstructure is expected to nucleate between $u^-$ and $u^+$.  

\bigskip
 
ii) In dimension $d=3$, the system \eqref{stokes} can be seen as a reduced model in micromagnetics in the regime where the so-called stray-field energy is strongly penalized favoring the divergence constraint $\nabla \cdot u=0$ of the magnetization $u$ (the unit-length constraint on $u$ being relaxed in our system \eqref{stokes}). One of the main issues consist in understanding the behavior of the magnetization $u$ describing the stable states of the energy $E$, the potential $W$ playing the role of the anisotropy favoring certain directions of $u$ (i.e., the zeros of $W$ are the ``easy axes" of the magnetization).  A rich pattern formation is expected for the magnetization, the generic state consisting in large uniformly magnetized $3D$ regions (called magnetic
domains) separated by narrow transition layers (called domain walls) where the magnetization
varies very rapidly between two directions $u^-$ and $u^+$ (see e.g. \cite{DKMOoverview, Hubert:1998} for more details). In this theory, a challenging question concerns the symmetry of domain walls. Indeed, much effort has been devoted lately to identifying on the one hand, the domain walls that have one-dimensional symmetry, such as the so-called symmetric N\'eel and symmetric Bloch walls (see e.g. \cite{DKO, IO, Ignat:2011}), and on the other hand, the domain walls involving microstructures, such as the so-called cross-tie walls (see e.g., \cite{Alouges:2002,Riviere:2001}), the zigzag walls (see e.g., \cite{Ignat:2012, Moser_zigzag}) or the asymmetric N\'eel / Bloch walls 
(see e.g. \cite{Doring:2013, DI}). Our paper aims to give a general approach in identifying the anisotropy potentials $W$ for which the domain walls are one-dimensional in the system \eqref{stokes}.

\subsection{Relation with the famous De Giorgi conjecture}
If one removes the divergence constraint in our model \eqref{stokes}, the problem reduces to the study of one-dimensional symmetry of solutions $u:\R^d\to \R^N$ of the equation
\be
\label{nodiv}
-\Delta u+\nabla W(u)=0 \quad \text{ in }\R^d,
\ee
for a nonnegative potential $W:\R^N\to \R_+$. 

\bigskip

i) In the scalar case $N=1$ and $W(u)=\frac14(1-u^2)^2$, the long standing De Giorgi conjecture predicts that any bounded solution $u$ that is monotone in the $x_1$ variable is one-dimensional in dimension $d\leq 8$, i.e., 
the level sets $\{u=\lambda\}$ of $u$ are hyperplanes. The conjecture has been solved in dimension $d=2$ by Ghoussoub-Gui \cite{Ghoussoub:1998}, using a Liouville-type theorem and monotonicity formulas. Using similar techniques, Ambrosio-Cabr\'e \cite{AmbrosioCabre:2000} extended these results to dimension $d=3$, while Ghoussoub-Gui \cite{ghoussoub2003giorgi} showed that the conjecture is true for $d=4$ and $d=5$ under some antisymmetry condition on $u$. The conjecture was finally proved by Savin  \cite{Savin:2009} in dimension $d\leq 8$ under the additional condition $\lim_{x_1\to\pm\infty}u(x_1,x')=\pm 1$ pointwise in $x'\in\R^{d-1}$, the proof being based on fine regularity results on the level sets of $u$. Lately, Del Pino-Kowalczyk-Wei  \cite{del2011giorgi} gave a counterexample to the De Giorgi conjecture in dimension $d\geq 9$, which satisfies the condition $\lim_{x_1\to\pm\infty}u(x_1,x')=\pm1$. 

Finally, we mention that if the convergence $\lim_{x_1\to\pm\infty}u(x_1,x')=\pm 1$ is uniform in $x'$ and \(|u|\le 1\), then the one-dimensional symmetry holds without the monotonicity assumption $\partial_1 u>0$ and there is no restriction on the dimension $d$ (this is the so-called Gibbons' conjecture), see e.g.  
\cite{BBG,BHM,F2,CafCor}.

\bigskip

ii) Less results are available for the vector-valued case $N\geq 2$. In the case $N=2$ and $W(u_1,u_2)=\frac 12 u_1^2u_2^2$, the system \eqref{nodiv} corresponds to a phase separation model arising in a binary mixture of Bose-Einstein condensates; one-dimensional symmetry of solutions has been shown in \cite{berestycki2013phase, berestycki2013entire} in dimension $d=2$ under some monotonicity / growth / stability conditions on solutions $u$. Several extensions, eventually in higher dimensions $d\geq 2$ can be found in \cite{farina2013some, farina2017monotonicity,wang2014giorgi}. In the case of more general potentials $W$ and $N\geq 1$, one dimensional symmetry for solutions $u$ of \eqref{nodiv} has been proved in \cite{fazly2013giorgi} for low dimensions $d$ provided that each component $u_k$ of $u$ ($1\leq k\leq N$) is monotone.

\bigskip

As mentioned above, an important observation in the study of De Giorgi type problems is the relation between monotonicity of solutions (e.g., the condition $\partial_1 u>0$), stability (i.e., the second variation of the corresponding energy is nonnegative), and local minimization of the energy (in the sense that the energy does not decrease under compactly supported perturbations of $u$). We refer to \cite[section 4]{Alberti:2001} for a fine study of these properties. In particular, it is shown that the monotonicity condition in the De Giorgi conjecture implies that $u$ is a local minimizer of the energy (see \cite[Theorem 4.4]{Alberti:2001}).

In the context of our system \eqref{stokes}, we need a stronger condition: $u$ is assumed periodic in the $d-1$ last variables and $u$ globally minimizes the energy in $\Omega$. At our knowledge, there are no Liouville type theorems or monotonicity formulas for \eqref{stokes} in order to mimic the general strategy used for \eqref{nodiv}. In fact, the divergence constraint in \eqref{stokes} creates a phenomenological difference with respect to \eqref{nodiv}. For example, if $d=N=2$ and $W$ is the Ginzburg-Landau potential, the Aviles-Giga model \eqref{stokes} has a solution connecting two zeros $u^+\neq u^-$ of $W$ (with $u^+_1=u^-_1$) that is one-dimensional and globally minimizes the energy on $\dot{H}_{div}^1(\Omega\subset \R^2,\R^2)$ (see \cite{Jin:2000}), while the energy \eqref{EE} has no one-dimensional global minimizer $u$ under the boundary condition \eqref{BC} (no divergence constraint here) because $u^+$ and $u^-$ belong to the same connected component of zeros of $W$ and, obviously, $u$ cannot be constant.

\section{Main results} 

We start by presenting the results for Question 2, and then the results for Question 1.

\subsection{Existence of global minimizers for general potentials $W$}

For the existence of a global minimizer of $E$ under both the divergence constraint and the boundary condition \eqref{BC}, we will need some assumptions about the behavior of $W$ near the sets $\R^d_a$ and $S_a\subset\R^d_a$, defined by
\begin{equation}\label{slice}
\R^d_a:=\{z=(a,x')\in \R^{d}\;:\; x'\in \R^{d-1}\},\quad S_a:=\{z\in\R^d_a\;:\; W(z)=0\},
\end{equation}
where $a:=u_1^+=u_1^-$. Note that if $u$ is an admissible one-dimensional map, i.e. $u=u(x_1)$ and $\nabla\cdot u=u_1'(x_1)=0$, then $u_1$ is constant in $\R$. In particular, $u_1\equiv a$, that is $u(x_1)\in\R^d_a$ for all $x_1\in\R$. This partially justifies the fact that we mainly require assumptions on the restriction of $W$ to a neighborhood of $\R^d_a$ in $\R^d$. In particular, we assume:
\begin{description}
\item[(H1)]
$S_a$ is a finite set,
\item[(H2)]
$\liminf\limits_{z_1\to a, \, |z'|\to\infty} W(z_1,z')>0.$
\end{description}
First, we state the answer to Question 2 in the case where $W$ has only two zeros $u^-$ and $u^+$ on $\R^d_a$, and next, we present the general case where $S_a$ is an arbitrary finite set.

\paragraph{The case of two-well potentials $W$ in $\R^d_a$.} If $S_a=\{u^-,u^+\}$, this assumption together with {\textbf{(H2)}} are sufficient to prove existence in Question 2.

\bigskip

\begin{theorem}\label{thm1}
If $W:\R^d\to \R_+$ is a continuous function such that for some $a\in\R$,
\begin{description}
\item[(H1')]
$S_a$ has exactly two distinct elements $u^-$ and $u^+$,
\end{description}
and {\rm{\bf (H2)}} is fulfilled, then there exists a {solution of the minimization problem}
\[
\tag{$\mathcal{P}$}
\inf\left\{E(u)\;:\; u\in \dot{H}_{div}^1(\Omega,\R^d)\text{ with }\eqref{BC} \right\}.
\]
\end{theorem}

\bigskip

The proof is based on the following general compactness result which is somehow reminiscent from \cite[Lemma 1.]{Doring:2013} and \cite[Lemma 4.4.]{Goldman:2015}. For that, we introduce the following notation: if $u \in \dot{H}^1(\Omega,\R^d)$, the $x'$-average of $u$ is a continuous function on $\R$ denoted by
\begin{equation}\label{average}
\overline{u}(x_1):=\int_{\T^{d-1}} u(x_1,x')\diff x', \quad x_1\in\R.
\end{equation}
If in addition $\nabla\cdot u=0$, then first component 
$\overline{u}\cdot e_1$ of the average $\bar u$ is constant in $\R$ (see Lemma~\ref{lem_aver}). The boundary condition \eqref{BC} will be denoted shortly by $\overline{u}(\pm \infty)=u^\pm$.

\begin{theorem}[Compactness of bounded energy sequences]\label{bndenergy}
Let $W:\R^d\to \R_+$ be a continuous function, $a\in\R$ and assume that {\rm \bf (H1')} and {\rm \bf(H2)} are satisfied. Let $(u_n)_{n\ge 1}$ be a sequence in $\dot{H}_{div}^1(\Omega,\R^d)$ such that $\overline{u_n}(\pm\infty)=u^\pm\in S_a$ for each $n\ge 1$ and
\[
\sup_{n\ge 1} E(u_n)=\sup_{n\ge 1} \int_\Omega \frac 12 |\nabla u_n|^2+W(u_n)\diff x<\infty.
\]
Then there exists a sequence $(t_n)_{n\geq 1}$ in $\R$ such that $(u_n(\cdot+t_n,\cdot))_{n\geq 1}$ weakly converges (up to a subsequence) in $\dot{H}^1(\Omega,\R^d)$ to a limit $u\in\dot{H}^1(\Omega,\R^d)$ satisfying
$$
E(u)\le \liminf_{n\to\infty}E(u_n),\quad \nabla\cdot u=0\quad\text{and}\quad\overline{u}(\pm\infty)=u^\pm.
$$
\end{theorem}

\paragraph{The case of multi-well potentials $W$ in $\R^d_a$.} The case where $S_a$ has three or more elements requires more attention. In this case, even the existence of global minimizers in the class of admissible one-dimensional maps $\gamma$, i.e.
\begin{equation}
\label{min1Dintro}
\inf \left\{\int_\R \frac 12 |\dot{\gamma}(t)|^2+W(\gamma(t))\diff t\;:\; \gamma\in \dot{H}^1(\R,\R^d_a)\text{ and }\gamma(\pm\infty)=u^\pm\right\},
\end{equation}
is not guaranteed without additional assumptions on $W$. For instance, in dimension $d=2$, $\gamma$ writes $\gamma(t)=(a,\gamma_2(t))\in\R^2_a$ and the Euler-Lagrange equation $\ddot{\gamma}_2=\partial_2 W(a,\gamma_2)$ has no solution if $W$ vanishes at three points $(a,u_2^-)$, $(a,b)$ and $(a,u^+_2)$ with $u^-_2<b<u^+_2$. 
In order to avoid this obstruction in our $d$-dimensional minimization problem $(\mathcal{P})$, we will impose the strict triangle inequality on the {\it transition cost} between the wells of $W$ in $S_a$. For that, we first introduce the energy functional restricted to an arbitrary interval $I\subset\R$: for every map $u\in \dot{H}^1_{div}(I\times\T^{d-1},\R^d)$, we set
\[
E(u,I):= 
 \int_{I\times\T^{d-1}} \frac 1 2 |\nabla u|^2 + W(u) \diff x.
\]
We define the {\it transition cost} $c_W:\R^d_a\times\R^d_a\to\R_+$ as follows: for all $z^-,\, z^+\in\R^d_a$,
{\begin{multline}
\label{cw}
c_W(z^-,z^+):=
\inf\Big\{E(u,I) \;:\; I=(t^-,t^+),\, u\in\dot{H}_{div}^1(I\times\T^{d-1},\R^d)\text{, } \overline{u}(t^\pm)=z^\pm,\\
\text{with }t^\pm\in \R\text{ if }W(z^\pm)>0, \text{ resp. }t^\pm=\pm\infty\text{ if }W(z^\pm)=0\Big\},
\end{multline}}
which means that the energy $E(u,I)$ is minimized 
\smallbreak
$\bullet$ on $I=\R$ if $W(z^-)=W(z^+)=0$; 

$\bullet$ on $I=\R_+$ if $W(z^-)>0$ but $W(z^+)=0$; 

$\bullet$ on $I=\R_-$ if $W(z^-)=0$ but $W(z^+)>0$; 

$\bullet$ over all bounded intervals $I\subset\R$ if $W(z^-)>0$ and $W(z^+)>0$.

\medbreak
\nd When minimizing the energy on a bounded domain $(t^-,t^+)\times\T^{d-1}$, one can actually impose a more general boundary condition $u(t^\pm,\cdot)=v^\pm$ (as a trace on $\{t^\pm\}\times \T^{d-1}$) for maps $v^-$ and $v^+$ belonging to the set \footnote{For technical reasons, we restrict to $H^1$ maps on the torus $\T^{d-1}$ instead of $H^{1/2}$ maps which is the natural space for the trace of our admissible configurations.}
$$
H^{1}_a(\T^{d-1},\R^d):=\left\{v\in H^{1}(\T^{d-1},\R^d)\;:\;\textstyle\int_{\T^{d-1}}v_1(x')\diff x'=a\right\};
$$
this yields a pseudo-distance $d_W$ on $H^1_a(\T^{d-1},\R^d)$ defined for all $v^\pm\in H^1_a(\T^{d-1},\R^d)$ by 
\footnote{ 
Note that $d_W$ might be infinite at some points because $W(u)$ is not necessarily $L^1$ for every $u\in \dot{H}^1(I\times\T^{d-1},\R^d)$ and every continuous $W$. 
}
\begin{align}
\label{dw}
d_W(v^-,v^+)&:=\inf\left\{E(u,I) \;:\; I=[t^-,t^+]\subset\R,\, u\in\dot{H}_{div}^1(I\times\T^{d-1},\R^d),\,u(t^\pm,\cdot)=v^\pm\right\}\\
\nonumber&\in [0,+\infty].
\end{align}
Obviously, if $v^\pm\equiv z^\pm$ with $z^\pm\in S_a$, then $d_W(v^-,v^+)\ge c_W(z^-,z^+)$, whereas the opposite inequality is more delicate (see Proposition~\ref{dWdist} below). 

Our answer to Question 2 in the case of multiple-well potentials $W$ inside $\R^d_a$ is the following:
\begin{theorem}
\label{thm2}
Let $W:\R^d\to \R_+$ be a continuous function such that for some $a\in\R$, $S_a$ contains at least two distinct wells $u^-$ and $u^+$, and the assumptions {\rm{\bf (H1)}} and {\rm{\bf(H2)}} are fulfilled. In addition, we assume
\begin{description}
\item[(H3)]
for all $z\in S_a\setminus{\{u^\pm\}},\ d_W(u^-,z)+d_W(z,u^+)>d_W(u^-,u^+),$
\item[(H4)]
$d_W$ is lower semicontinuous in $H^1$ on the set $S_a\times S_a$ in the following sense: for every $z^\pm\in S_a$ and for every sequences $(v^\pm_n)_{n\ge 1}$ in $H^{1}_a(\T^{d-1},\R^d)$ strongly converging to $z^\pm$ in $H^1(\T^{d-1})$, one has
$$
d_W(z^-,z^+)\le \liminf_{n\to\infty}d_W(v^-_n,v^+_n).
$$
\end{description}
Then the problem ($\mathcal{P}$) has a solution, i.e., there exists $u\in \dot{H}_{div}^1(\Omega,\R^d)$ such that $E(u)=c_W(u^-,u^+)$ and $\overline{u}(\pm\infty)=u^\pm$.
\end{theorem}
In this multiple-well context, the triangle inequality \textbf{(H3)} is essential. It insures that, if the energy of an admissible function $u$ with $\overline{u}(\pm\infty)=u^\pm$ is almost minimal (i.e., $E(u)- c_W(u^-,u^+)$ is small enough), then the path $t\mapsto u(t,\cdot)$ cannot get too close in $H^1(\T^{d-1})$ to a zero of $W$ other than $u^-$ and $u^+$.
We also observe that, under the assumption {\rm \bf(H4)}, one has $d_W=c_W$ on $S_a\times S_a$:
 \begin{proposition}\label{dWcW}
 If $W:\R^d\to \R_+$ is a continuous function such that the assumption {\rm \bf(H4)} is fulfilled for some $a\in\R$, then one has $d_W(z^-,z^+)=c_W(z^-,z^+)$ for every $z^\pm\in S_a$.
 \end{proposition}

The assumption {\textbf{(H4)}} always holds in dimension $d=2$ for a continuous potential $W$, while in higher dimensions $d\geq 3$, a growth condition on $W$ should be assumed in addition:
\begin{proposition}\label{dWdist}
Let $W:\R^d\to\R_+$ be a continuous function. If $d\geq 3$, assume the following growth condition: 
\footnote{With the convention $\frac{2d}{d-3}=+\infty$ if $d=3$.}
\begin{equation}\label{growthMultiwell}
\text{there exist } C>0 \text{ and }
q\in \big(1,\frac{2d}{d-3}\big) \text{ s.t. }\ W(z)\leq C(1+|z|^q)\quad\forall z\in\R^d.
\end{equation}
Then for every $a\in\R$, $z^\pm\in S_a$ and for every sequences $(v^\pm_n)_{n\in\N}\subset H^{1}_a(\T^{d-1},\R^d)$ converging strongly in $H^1$ to $z^\pm$, one has $d_W(z^-,z^+)=\lim_{n\to\infty}d_W(v^-_n,v^+_n)$. In particular, the assumption {\rm \bf(H4)} holds true.
\end{proposition}

\subsection{One-dimensional symmetry of global minimizers\label{sym_intro}}

In order to prove symmetry of global minimizers of $(\mathcal P)$, we develop a calibration method for divergence-free maps in any dimension $d\geq 2$; this method is reminiscent in dimension $d=2$ from the couple (entropy, entropy-flux) used in scalar conservation laws (see e.g. \cite{Aviles:1999, Jin:2000}). More precisely, if $d\geq 2$, our key object is the so-called {\it entropy} designing a $\mathcal{C}^1$ map $\Phi:\R^d\to\R^d$ such that for every smooth divergence-free map $u:\Omega\to \R^d$ with $\nabla\cdot[\Phi(u)]$ being integrable on $\Omega$, one has
\begin{equation}
\label{entropy_intro}
\int_\Omega \nabla\cdot[\Phi(u)] \, \diff x\leq E(u)
\end{equation}
(see Definition~\ref{entropy_def} for a precise statement and comments). If $u$ satisfies the boundary condition $\overline{u}(\pm\infty)=u^\pm\in S_a$, then the LHS of \eqref{entropy_intro} is independent from $u$ since
\begin{equation}
\label{satur_intro}
\int_\Omega \nabla\cdot[\Phi(u)]\, \diff x=
\Phi_1(u^+)-\Phi_1(u^-).
\end{equation}
Next to the first condition \eqref{entropy_intro}, we impose the following second condition, called {\it saturation condition} for entropies $\Phi$:
\begin{equation}\label{saturation_condition}
\Phi_1(u^+)-\Phi_1(u^-)=E(u^{1D}),
\end{equation}
provided that there exists a minimizing one-dimensional transition layer $u^{1D}$ of \eqref{min1Dintro}.
Note that the existence of an entropy satisfying the saturation condition implies that $u^{1D}$ is a global minimizer of $(\mathcal P)$ (for details see Propositions \ref{minimal} and \ref{minimal_bndProp}). Moreover, if $u$ is another global minimizer of $(\mathcal P)$, then \eqref{entropy_intro} has to be an equality; this equality is the corner stone in our proofs for the one-dimensional symmetry of $u$ (see e.g. Proposition \ref{diagonal_entropy}).

The existence of an entropy $\Phi$ is a delicate issue and requires strong assumptions on $W$. Even if there is no general recipe, we will present three situations where we are able to determine potentials $W$ for which entropies $\Phi$ do exist.

\medskip
\nd {\bf Situation 1. Strong punctual condition ($\mathcal{E}_{strg}$)}. We look for potentials $W$ for which there exist $\mathcal{C}^1$ maps $\Phi:\R^d\to\R^d$ satisfying the saturation condition \eqref{saturation_condition} and the punctual estimate
\begin{equation}\label{strong_punct}
(\mathcal{E}_{strg}) \quad |\Pi_0\nabla\Phi(z)|^2\leq 2W(z)\quad\text{for a.e. }z\in\R^d,
\end{equation}
where $\Pi_0$ is the orthogonal projection onto the set of traceless matrices:
\begin{equation*}
\Pi_0 U=U-\frac{{\rm Tr}(U)}{d}I_d,\quad U\in \R^{d\times d},
\end{equation*}
and $I_d$ stands for the identity matrix in $\R^{d\times d}$. Then \eqref{strong_punct} is a sufficient condition for $\Phi$ to be an entropy. Indeed, for all smooth $u:\Omega\to \R^d$ with $\nabla\cdot u=0$, one has $\nabla u=(\partial_ju_i)_{i,j}\in{\rm Im}(\Pi_0)$ and therefore,
\begin{align}
\nonumber
\nabla\cdot[\Phi(u)]=\nabla\Phi(u):\Pi_0\nabla u^T&= \Pi_0\nabla\Phi(u):\nabla u^T\\
\label{entropy_strong_est} &\leq \frac 12(|\nabla u|^2+|\Pi_0\nabla\Phi(u)|^2)\stackrel{\eqref{strong_punct}}{\leq} \frac 12|\nabla u|^2+W(u),
\end{align}
where $U^T$ is the transpose of a matrix $U$ and the Euclidean scalar product on $\R^{d\times d}$ is denoted by
\[
U:V=\mathrm{Tr}(UV^T)=\sum_{i,j\in\{1,\dots,d\}}U_{ij}V_{ij}.
\]
{We refer to Section \ref{entropymethod} for more details.}

\medskip
{\nd \bf Situations 2 (resp. 3). Entropies with symmetric ($\mathcal{E}_{sym}$) (resp. antisymmetric ($\mathcal{E}_{asym}$)) Jacobian.}
Let $\Pi^+$ (resp. $\Pi^-$) be the projection of $\R^{d\times d}$ on the subspace of symmetric (resp. antisymmetric) matrices, that is 
$$\Pi^\pm U=\frac 12 (U\pm U^T) \quad \textrm{for every }  U\in\R^{d\times d}.$$ We want to find potentials $W$ for which there exist maps $\Phi\in\mathcal{C}^1(\R^d,\R^d)$ satisfying the saturation condition \eqref{saturation_condition} and 
\begin{align}
\label{ent_sym}
(\mathcal{E}_{sym})\quad  &
 \nabla\Phi(z) \, \textrm{ is symmetric and } \, |\Pi_0\nabla\Phi(z)|^2\leq 4W(z)\, \textrm{ for all } z\in\R^d,\\
\label{ent_asym}  \textrm{resp. } \quad  (\mathcal{E}_{asym})  \quad & \Pi_0\nabla\Phi(z) \,  \textrm{ is antisymmetric and }\,  |\Pi_0\nabla\Phi(z)|^2\leq 4W(z) \, \textrm{ for all } z\in\R^d.
\end{align}
Then \eqref{ent_sym} (resp. \eqref{ent_asym}) is a sufficient condition for $\Phi$ to be an entropy. Indeed, if $\Phi\in\mathcal{C}^1(\R^d,\R^d)$ satisfies $\nabla\Phi(z)\in\mathrm{Im}(\Pi^\pm)$ for every $z\in\R^d$, then for all smooth $u:\Omega\to \R^d$ with $\nabla\cdot u=0$, one has
\begin{align}
\label{11}
 \nabla\cdot[\Phi(u)]=\nabla\Phi(u):\nabla u^T=\Pi^\pm\nabla\Phi(u):\Pi_0\nabla u^T=\Pi_0\nabla\Phi(u):\Pi^\pm\nabla u^T.
\end{align}
By Young's inequality and \eqref{ent_sym} (resp. \eqref{ent_asym}), it yields
\begin{align*}
\int_\Omega \nabla\cdot[\Phi(u)]\diff x
&\leq \frac 12\Big(2\|\Pi^\pm\nabla u\|_{L^2(\Omega)}^2+\frac 12\|\Pi_0\nabla\Phi(u)\|_{L^2(\Omega)}^2\Big)\leq E(u)
\end{align*}
due to the following identities valid for $u\in \dot{H}_{div}^1(\Omega,\R^d)$ with $\bar u\in L^\infty(\R, \R^d)$ (see Proposition~\ref{antisymgradient})
\begin{equation}
\label{equiProjection}
\|\Pi^+\nabla u\|_{L^2(\Omega)}^2=\|\Pi^-\nabla u\|_{L^2(\Omega)}^2=\frac 12\|\nabla u\|_{L^2(\Omega)}^2.
\end{equation}

Following the criteria \eqref{strong_punct}, \eqref{ent_sym} or \eqref{ent_asym}, we will construct potentials $W$ for which all global minimizers of $(\mathcal P)$ are one-dimensional. We present these results in dimension $d=2$ and then in dimension $d\geq 3$. One important observation is that the existence of an entropy $\Phi$ satisfying the saturation condition \eqref{saturation_condition} and one of the conditions ($\mathcal{E}_{strg}$), ($\mathcal{E}_{sym}$) or ($\mathcal{E}_{asym}$) implies that any global minimizer of $(\mathcal P)$ satisfies the following first order PDE, which encodes the Euler-Lagrange equation \eqref{stokes}, the (second order) stability conditions and the equipartition of the energy density, i.e. \(\frac 12|\nabla u|^2=W(u)\) a.e. in $\Omega$:
\begin{description}
\item[1.] if $\Phi\in\mathcal{C}^1(\R^d,\R^d)$ satisfies ($\mathcal{E}_{strg}$), then any global minimizer $u$ of $(\mathcal{P})$ solves 
\begin{equation}\label{pdeStrong}
W(u)=\frac 12 |\Pi_0\nabla\Phi(u)|^2\quad\text{and}\quad
\nabla u^T=\Pi_0\nabla\Phi(u)\text{ a.e. in $\Omega$;}
\end{equation}
\item[2.] if $\Phi\in\mathcal{C}^1(\R^d,\R^d)$ satisfies ($\mathcal{E}_{asym}$), then any global minimizer $u$ of $(\mathcal{P})$ solves 
\begin{equation}\label{pdeAsym}
W(u)=\frac 14|\Pi_0\nabla\Phi(u)|^2\quad\text{and}\quad
2\Pi^-\nabla u^T=\Pi_0\nabla\Phi(u)\text{ a.e. in $\Omega$;}
\end{equation}
\item[3.] if $\Phi\in\mathcal{C}^1(\R^d,\R^d)$ satisfies ($\mathcal{E}_{sym}$), then any global minimizer $u$ of $(\mathcal{P})$ solves 
\begin{equation}\label{pdeSym}
W(u)=\frac 14|\Pi_0\nabla\Phi(u)|^2\quad\text{and}\quad
2\Pi^+\nabla u=\Pi_0\nabla\Phi(u)\text{ a.e. in $\Omega$}
\end{equation}
\end{description}
(see Proposition \ref{pde_opt}). 

\paragraph{Symmetry results in dimension $d=2$.} We prove one-dimensional symmetry in the minimization problem $(\mathcal P)$ for potentials  $W=\frac 12w^2$, where $w$ solves the Laplace equation or the wave equation and satisfies the following growth condition:
\begin{equation}
\label{growth_intro}
\textrm{there exist } \, C,\beta>0 \textrm{ such that } \quad
W(z)\leq C\, e^{\beta |z|^2} \quad \forall z\in\R^2.
\end{equation}

\begin{theorem} 
\label{thm:harm_wave}
Assume that $W= \frac12 w^2$ where $w\in\mathcal{C}^2(\R^2,\R)$ solves
\[
\Delta w=\partial_{11} w+\partial_{22} w=0  \quad \textrm{(resp.} \quad  \square w=\partial_{11} w-\partial_{22} w=0 \textrm{) \quad in } \, \R^2 
\] 
and let $u^\pm =(a,u_2^\pm)\in\R^2$ be two wells of $W$, i.e. $W(u^\pm)=0$, such that $w>0$ on the open segment $(u^-,u^+)$. Assume that $u$ is a global minimizer of $(\mathcal P)$ and either $u\in L^\infty$ or $W$ satisfies the growth condition \eqref{growth_intro}. Then $u$ is one-dimensional, i.e., $u=g(x_1)$ where $g$ is the unique minimizer in \eqref{min1Dintro}. 
\end{theorem}
The proof of Theorem \ref{thm:harm_wave} is based on the construction of an entropy $\Phi$ such that
\begin{equation}
\label{ansatz_harmonic}
\nabla\Phi(z)=
\left(
\begin{matrix} 
-\alpha (z)& w(z)\\
\mp w(z)&-\alpha(z)
\end{matrix}
\right)
\quad\text{for all }z\in\R^2,
\end{equation}
for some scalar function $\alpha$ that solves the same equation as $w$. The sign $\mp=+$ in \eqref{ansatz_harmonic} corresponds to Situation 2 of an entropy $\Phi$ with symmetric Jacobian which applies if $w$ solves the wave equation, resp. the sign $\mp=-$ corresponds to Situation 3 of an entropy $\Phi$ with antisymmetric Jacobian (i.e., $\Phi$ is holomorphic on $\R^2$) which applies if $w$ is an harmonic function in $\R^2$. The one-dimensional symmetry will follow by investigating the equality in \eqref{entropy_intro} (in particular \eqref{pdeAsym} and \eqref{pdeSym}).

If $w$ is harmonic, due to the classical maximum principle, the set $\{W=0 \}=\{w=0\}$ cannot contain an isolated point or a closed curve. In fact, if $w$ is not constant, then $\{w=0\}$ is a union (possibly infinite) of noncompact smooth curves (without end-points). An example is given by $w(z_1,z_2)=z_1z_2$, where $\{w=0\}$ is the union of two orthogonal straight lines \footnote{The potential $W(u)=\frac12u_1^2u_2^2$ appears naturally in phase separation models, see Example \ref{exa:separ}.}. In the case of the wave equation, we recover the Ginzburg-Landau potential for 
$w(z)=1-|z|^2$ and therefore, the one-dimensional symmetry of global minimizers in the Aviles-Giga model:
\begin{corollary}\label{symmetry_ag}
Let $W(z)= \frac12 (1-|z|^2)^2$ and $u^\pm =(a,\pm b)\in\R^2$ be two wells of $W$ with $a^2+b^2=1$. Assume that $u$ is a global minimizer of $(\mathcal P)$. Then $u$ is one-dimensional, i.e. $u(x)=g(x_1)$ with $g\in\mathcal{C}^2(\R,\R^2)$ being the unique minimizer of \eqref{min1Dintro} (up to translation in $x_1$-variable).
\end{corollary}
We will provide a symmetry result also for more general potentials $W=\frac 12 w^2$, where $w$ solves the Tricomi equation.

\begin{theorem} 
\label{thm:tricomi}
Assume that $W= \frac12 w^2$ where $w\in\mathcal{C}^2(\R^2,\R)$ 
satisfies the Tricomi equation 
\begin{equation}
\label{tricomi_eq}
\partial_{11}w-f(z_1) \partial_{22}w=0\quad\textrm{for every } z=(z_1,z_2)\in \R^2,
\end{equation} for a $\mathcal{C}^1$ function $f:\R\to \R$ with $|f|\leq 1$ in $\R$.
Let $u^\pm =(a,u_2^\pm)\in\R^2$ be two wells of $W$, i.e. $W(u^\pm)=0$, such that $w>0$ on the open segment $(u^-,u^+)$. If $u$ is a global minimizer of $(\mathcal P)$ and either 
$u\in L^\infty$ or $|\nabla w|$ satisfies the growth condition \eqref{growth_intro}, then $u$ is one-dimensional, i.e., $u=g(x_1)$ where $g$ is the unique minimizer in \eqref{min1Dintro}. 
\end{theorem}

\begin{remark}
\label{rem:tric}
i) If $f=\pm1$ in \eqref{tricomi_eq}, we recover the case of $w$ satisfying the wave equation and the Laplace equation, respectively. Another example is given by the potential $W=\frac12 w^2$ with $w=1-\delta x_1^2-x_2^2$ and the constant $\delta\in (-1,1)\setminus \{0\}$ (here, $f=\delta$ in \eqref{tricomi_eq}).

ii) We point out that the symmetry result for global minimizers of $(\mathcal{P})$ fails in general for potentials $W(z)=\frac 12w^2 (z)$ where $w\in\mathcal{C}^2(\R^2,\R)$ satisfies the Tricomi equation for some large function $f$. Indeed, Jin-Kohn \cite{Jin:2000} proved that for $w=1-\delta x_1^2-x_2^2$ where $f=\delta\gg1$ is large, the one-dimensional transition layer between the wells $(0, \pm 1)$ in direction $x_1$ is no longer a global minimizer because two-dimensional microstructures are energetically less expensive.
\end{remark}

\paragraph{Symmetry results in dimension $d\geq 2$.} 
Following the criterium ($\mathcal{E}_{sym}$) in \eqref{ent_sym}, we can construct a family of potentials $W$ in any dimension $d\geq 3$ for which the answer to Question 1 is positive. These potentials are the generalization of the $2D$ case of potentials $W=\frac12w^2$ where $w$ is a solution to the wave equation. As in $2D$, we will impose a growth condition on the potential $W$:
\begin{equation}
\label{growth_W}
\textrm{there exist } \,  C,\beta>0 \textrm{ such that for all }\,  z\in\R^d,\quad
W(z)\leq 
\begin{cases}
C\, e^{\beta |z|^2}&\text{if }d=2,\\
C\, [1+|z|^{2^\ast}]&\text{if }d\geq 3,
\end{cases}
\end{equation}
where $2^\ast = \frac{2d}{d-2}$ is the critical Sobolev exponent for the Sobolev embedding $H^1(\R^d)\hookrightarrow L^{2^\ast}(\R^d)$.
Moreover, we need to assume the existence of an entropy $\Phi\in\mathcal{C}^1(\R^d,\R^d)$ satisfying the saturation condition \eqref{saturation_condition};  more precisely, the minimal energy of one-dimensional transition layers between two zeros $u^\pm\in S_a$ of $W$ (see \eqref{slice}) is given by the geodesic pseudo-distance $\mathrm{geod}^a_W$ between $u^\pm$ in the hyperspace $\R^d_a$ endowed with the pseudo-metric $2Wg_0$ (with $g_0$ being the standard euclidean metric), i.e., 
\be
\label{geo_inter}
\mathrm{geod}^a_{W}(u^-,u^+):=\inf\bigg\{\int_{-1}^1 \sqrt{2W(\gamma(t))}|\dot{\gamma}|(t)\diff t\;:
\; \gamma\in {\rm Lip}([-1,1],\R^d_a),\,\gamma(\pm1)=u^\pm\bigg\}.
\ee

\begin{theorem}\label{thm:rigid_sym}
Let $\Phi=(\Phi_1, \dots, \Phi_d)\in \mathcal{C}^1(\R^d,\R^d)$ be a map such that for every $z\in\R^d$, $\nabla\Phi(z)$ is symmetric and
\be
\label{33} 
\partial_1\Phi_1(z)=\dots=\partial_d\Phi_d(z).
\ee
Consider the potential
\begin{equation}\label{calibW}
W(z)=\frac 12\sum_{1\leq i<j\leq d}|\partial_i\Phi_j(z)|^2
\end{equation}
and two wells $u^\pm\in\R^d_a\cap\{W=0\}$ of $W$ for some $a\in \R$. If the saturation condition 
\be
\label{interm_sat}
\Phi_1(u^+)-\Phi_1(u^-)=\mathrm{geod}^a_{W}(u^-,u^+)
\ee
is satisfied and if $u$ is a global minimizer of $(\mathcal{P})$ such that
either ($u\in L^\infty(\Omega,\R^d)$ and $W\in\mathcal{C}^2(\R^d,\R_+)$) or $W$ satisfies the growth condition \eqref{growth_W}, then $u$ is one-dimensional, i.e. $u=u(x_1)$.
\end{theorem}
In order to build explicit examples of potentials $W$ for which Theorem \ref{thm:rigid_sym} applies, we consider entropies $\Phi\in\mathcal{C}^1(\R^d,\R^d)$ of the form $\Phi(z)=\nabla\Psi (z)$ for all $z\in\R^d$, for some $\Psi\in \mathcal{C}^2(\R^d)$. Then, $\nabla\Phi=\nabla^2\Psi$ is symmetric and the condition \eqref{33} becomes in terms of $\Psi$:
\begin{equation}\label{wavePsi}
\partial_{11}\Psi =\dots =\partial_{dd}\Psi .
\end{equation}
By analogy with the wave equation in $\R^2$, solutions of \eqref{wavePsi} can be written
$$\Psi(z)=\sum_{\sigma\in\{\pm1\}^d}f_\sigma (\sigma\cdot z) ,$$
where $(f_\sigma)_\sigma$ is a family of scalar functions defined over $\R$. In particular, we recover the following extension of the $2D$ Ginzburg-Landau potential in any dimension $d\geq 3$:
$$W_d(z):=\frac 12(|z|^2-1)^2+2|z''|^2(z_1^2+z_2^2),\quad z=(z_1,z_2, z'')\in\R^d, \, z''=(z_3, \dots, z_d),$$
corresponding to $$\Psi_d (z)=-z_1z_2\left(\frac{z_1^2+z_2^2}{3}+|z''|^2-1\right),\quad z=(z_1,z_2, z'')\in\R^d.$$
We highlight that the saturation condition in Theorem \ref{thm:rigid_sym} is not always satisfied; in fact, we will see that for the perturbed Ginzburg-Landau potential $W_{3}$ in $3D$, the saturation condition fails for the wells $u^\pm=\pm e_3$ of $W$, i.e., $\Phi_1(u^+)-\Phi_1(u^-)=0<\mathrm{geod}^a_{W}(u^-,u^+)$ and in that case, there are no minimizers of $(\mathcal{P})$, see page \pageref{fail_satur}.

Following the criterium ($\mathcal{E}_{strg}$) in \eqref{strong_punct}, we present another situation where the answer to Question 1 is positive, consisting in a family of potentials $W$ with finite number of wells $X:=W^{-1}(\{0\})$. 
In this situation, the minimal energetic cost between two wells $u^\pm\in X$ is expected to be given by 
 the geodesic distance $\mathrm{geod}_W$ between $u^-$ and $u^+$ in $\R^d$ endowed with the metric $2Wg_0$ (recall that $g_0$ is the standard euclidean metric): \footnote{Note the difference with respect to the geodesic distance $\mathrm{geod}^a_W$ defined in \eqref{geo_inter}  where the curves $\gamma$ 
 are confined in the hyperspace $\R^d_a$.}
\begin{equation}
\label{geod_W_gen}
\mathrm{geod}_W(z^-,z^+):=\inf \left\{\int_{-1}^1 \sqrt{2W(\gamma)}|\dot{\gamma}| \;:\; \gamma\in\mathrm{Lip}([-1,1],\R^d),\, \gamma(\pm 1)=z^\pm\right\},\quad  z^-,z^+\in\R^d.
\end{equation}
We will prove that for any affine basis $X$ in $\R^d$ and any given metric $\delta$ on $X$, there exists a potential $W$ such that $\mathrm{geod}_W$ coincides with $\delta$ on $X$ and the optimal transition layers between any two wells of $X$ are one-dimensional. For that, we need to set some notation as the transition direction between two wells $u^\pm\in X$ is not necessarily $e_1$: if $u^\pm\in X$, $\nu\in\mathbb{S}^{d-1}$ with $\nu\cdot (u^+-u^-)=0$ and $R\in SO(d)$ such that $R\nu=e_1$, then we set $\Omega_R=R^{-1}\Omega$ and the energy on $\Omega_R$
\begin{equation}\label{minimalNu}
E_R(u):=\int_{\Omega_R} \frac 12|\nabla u|^2+W(u)\diff x,\quad \nabla \cdot u=0,
\end{equation}
where the minimisation of $E_R$ is considered
over divergence-free maps $u\in \dot{H}_{div}^1(\Omega_R,\R^d)$ that are periodic in the directions $R^{-1}e_k$, $k=2, \dots, d$
and satisfy the boundary condition
\begin{equation}
\label{constraintNu}
\lim\limits_{t\to\pm\infty}\int_{R^{-1}\T^{d-1}} u(t\nu+\sigma)\diff\mathcal{H}^{d-1}(\sigma)=u^\pm.
\end{equation}
\begin{theorem} 
\label{thm:finite_wells}
Let $X=\{x_0, \dots, x_d\}$ be an affine basis in $\R^d$ and let $\delta$ be a metric on $X$. Then there exists a Lipschitz potential $W:\R^d\to \R_+$ such that $W(z)=1$ for $|z|$ large enough, $X=W^{-1}(\{0\})$, $\mathrm{geod}_W=\delta$ on $X\times X$ and for every $u^\pm\in X$, $\nu\in\mathbb{S}^{d-1}$ with $\nu\cdot (u^+-u^-)=0$ and any $R\in SO(d)$ such that $R\nu=e_1$, if $u\in \dot{H}_{div}^1(\Omega_R,\R^d)$ is a global minimizer of $E_R$ in \eqref{minimalNu} over divergence-free configurations with the boundary condition \eqref{constraintNu}, then $u$ is one-dimensional, i.e., $u=g(x\cdot \nu)$ where $g\in \dot{H}^1(\R,\R^d)$ with $g(\pm\infty)=u^\pm$ and $E_R(u)=\delta(u^-, u^+)$.
\end{theorem}
The proof of Theorem \ref{thm:finite_wells} is based on the criterium ($\mathcal{E}_{strg}$) for the existence of entropies. More precisely, we construct a potential $W$ of the form $\frac12|\nabla \varphi|^2$ for a scalar function $\varphi:\R^d\to \R$ such that $|\varphi(u^+)-\varphi(u^-)|=\delta(u^-,u^+)$ for every two wells $u^\pm\in X$; the corresponding entropy $\Phi$ for a pair of wells $u^\pm\in X$ and a direction $\nu\in\mathbb{S}^{d-1}$ with $\nu\cdot (u^+-u^-)=0$ is given by $\Phi:=\varphi \nu$ which satisfies \eqref{strong_punct} and the saturation condition \eqref{interm_sat}.

In dimension $d\geq 3$, we will show that the criterium ($\mathcal{E}_{asym}$) in \eqref{ent_asym} is too rigid in order to construct entropies $\Phi$ leading to potentials $W$ for which the answer to Question 1 is positive (see Proposition \ref{rigidity_antisym} and Corollary \ref{cor_pde_opt}).

\paragraph{Approximation argument} In order to remove the smoothness condition on the admissible maps $u$ imposed in the definition of entropies \eqref{entropy_intro}, a key ingredient is the following regularization procedure that insures the avoidance of the so-called ``Lavrentiev gap'' with respect to the general set of admissible maps (not necessarily smooth) with finite energy $\{ u\in \dot{H}_{div}^1(\Omega,\R^d)\; :\; E(u)<\infty\}$. This is done under the growth condition \eqref{growth_W} imposed on the potential $W$.

\begin{lemma}\label{lavrentiev}
Let $W:\R^d\to\R_+$ be a continuous potential, $u^\pm\in S_a$ for some $a\in\R$ and $u\in \dot{H}_{div}^1(\Omega,\R^d)$ be a map such that $E(u)<\infty$ and $\overline{u}(\pm\infty)=u^\pm$. If either $W$ satisfies \eqref{growth_W} or ($u\in L^\infty(\Omega)$ and $W\in\mathcal{C}^2(\R^d,\R_+)$) then there exists a sequence $(u_k)_{k\geq 1}\subset \mathcal{C}^{\infty}\cap L^\infty \cap \dot{H}_{div}^1(\Omega,\R^d)$ such that $\overline{u_k}(\pm\infty)=u^\pm$ for each $k$, and
\[
u_k\underset{k\to\infty}{\longrightarrow} u\quad\text{in $\dot{H}^1(\Omega,\R^d)$}
,\quad e_{den}(u_k)\underset{k\to\infty}{\longrightarrow}e_{den}(u)\quad \text{in }L^1(\Omega),
\]
where $e_{den}(u)$ stands for the energy density, i.e. $e_{den}(u)=\frac 12 |\nabla u|^2+W(u)$. In the case where the growth condition \eqref{growth_W} holds true, then one can impose in addition that $u_k$ satisfies \(\overline{u_k}(\pm t)=u^\pm\) for all \(t\ge t_k\) with $t_k$ large, for every $k\in \N$.
\end{lemma}

\paragraph{Change of variables under rotation.} \label{pagina} The fact that we consider an infinite cylinder $\Omega$ oriented in $x_1$-direction, is an implicit way of fixing the direction of the transition between $u^-$ and $u^+$. Of course, one can always reduce to this case by rotation as follows. Let $\nu\in\mathbb{S}^{d-1}$ be a new transition axis and let $R\in SO(d)$ be a rotation such that $R\nu=e_1$. Define the rotated domain
\[
\Omega_R:=\{x_R=R^{-1}x\;:\;x\in\Omega\},
\]
which can be seen as a cylinder, infinite in the direction $\nu$. To be more precise, $\Omega_R$ is the set of equivalence classes in $\R^d$ endowed with the equivalence relation ($x\sim_R y$ iff $R(x-y)$ vanishes in $\Omega=\R\times\T^{d-1}$). Let us take an admissible configuration
\[
u:\Omega\to\R^d\quad \text{such that}\quad \nabla\cdot u=0,
\]
and define the rotated map $u_R$ by
\[
u_R:\Omega_R\to\R^d,\quad u_R(x_R)=R^{-1} u(R\, x_R),\quad x_R\in\Omega_R.
\]
Thus, $u_R$ is still divergence free thanks to the elementary computation,
\[
\nabla\cdot u_R(x_R)=\mathrm{Tr}(\nabla u_R(x_R))=\mathrm{Tr}(R^{-1}\nabla u(R\, x_R)R)=\nabla\cdot u(R\, x_R)=0,\quad x_R\in\Omega_R.
\]
Set $W_R:\R^d\to\R_+$ the new potential defined by
\[
W_R(z_R)=W(R\, z_R),\quad z_R\in\R^d.
\]
Then for any $u\in \dot{H}_{div}^1(\Omega,\R^d)$, one has
\[
E(u)=\int_\Omega \frac 12 |\nabla u|^2+W(u)\diff x= E_R (u_R)=\int_{\Omega_R}\frac 12 |\nabla u_R|^2+W_R(u_R)\diff x_R.
\]
Moreover, if $u^\pm$ are two wells of $W$ compatible with the divergence constraint and the boundary condition $\overline{u}(\pm\infty)=u^\pm$, i.e., $(u^+-u^-)\cdot e_1=0$, then the rotated wells,
\[
u^\pm_R:=R^{-1}u^\pm\in \{W_R=0\},
\]
are compatible with the divergence constraint $\nabla\cdot u_R=0$ and the new boundary condition:
\[
\lim\limits_{y\to\pm\infty}\int_{\{x_R\in\Omega_R\;:\; x_R\cdot\nu=y\}} u_R(x_R)\diff\mathcal{H}^{d-1} (x_R)=u^\pm_R.
\]
Moreover, if $\Phi:\R^d\to\R^d$ satisfies \eqref{entropy_intro} and \eqref{satur_intro}, then $\Phi_R$ defined by
\[
\Phi_R:\R^d\to\R^d,\quad \Phi_R(z_R)=R^{-1}\Phi (R\, z_R),\quad z_R\in\R^d,
\]
is an entropy adapted to the new energy functional $E_R$ and the rotated wells $u^\pm_R$, i.e. \eqref{entropy_intro} and \eqref{satur_intro} hold true for $(\Omega_R,\Phi_R,u_R,E_R)$ instead of $(\Omega,\Phi,u,E)$.

\medskip

From the previous analysis, we deduce that the study of the existence and the symmetry of global minimizers for a potential $W_R$ and a transition configuration $(\nu,u^+_R,u^-_R)$ (i.e. a transition axis $\nu$, and two wells $u_R^\pm$ of $W_R$ with $\nu\cdot(u^+_R-u^-_R)=0$) reduces to the same study for the potential $W(z)=W_R(R^{-1}z)$ and the transition configuration $(e_1,u^+,u^-)=(R\nu,Ru^+_R,Ru^-_R)$. For the existence of minimizers, the assumptions of Theorems \ref{thm1} and \ref{thm2} easily transpose to any configuration axis $\nu$. For the one-dimensional symmetry of global minimizers, note that the assumptions on $W=\frac12w^2$ in Theorem \ref{thm:harm_wave} (when $w$ is harmonic), Corollary \ref{symmetry_ag} and Theorem \ref{thm:finite_wells} are invariant by rotation, i.e. it remains true with $W_R$ instead of $W$. Indeed, the Laplace equation $\Delta w=0$ is invariant by rotation, while for the Ginzburg-Landau potential, $w(z)=1-|z|^2$ is radially symmetric. The previous analysis also implies that in Theorem \ref{thm:finite_wells}, it is enough to prove one-dimensional symmetry for the special transition axis $\nu=e_1$ (i.e. $R=I_d$ and $\Omega_R=\Omega$).

\subsection{Structure of the paper}
In Section \ref{existence}, we analyze Question 2 on the existence of global minimizers, proving Theorem~\ref{thm1} in Section \ref{two-wells} and Theorem~\ref{thm2} in Section \ref{MultipleWell}. Section \ref{one_dim_sym} is dedicated to Question 1, i.e., the study of one-dimensional symmetry of minimizers. We first make a quick analysis of the minimization problem \eqref{min1Dintro} in $1D$ in Section \ref{an_1Dprofile}. In Section \ref{sec:lav}, we prove the approximation argument in Lemma~\ref{lavrentiev}. In Section \ref{entropymethod}, we explain our main tool, the entropy method; as an immediate consequence, we present the structure of global minimizers as solutions to a first order PDE system in Section \ref{sec:structure}. Our main results on the symmetry of minimizers in $2D$ (i.e. Theorems~\ref{thm:harm_wave}, \ref{thm:tricomi} and Corollary~\ref{symmetry_ag}) are shown in Section \ref{one_dim_2D}. Finally, we extend our method to some situations in higher dimension by proving Theorems \ref{thm:rigid_sym} and \ref{thm:finite_wells} in Section \ref{higherdimension}. 

\section{Existence of global minimizers\label{existence}}

Due to the translation invariance in $x_1$-direction of the domain $\Omega$ and of the energy $E$, the existence of a global minimizer in ($\mathcal{P}$) under the boundary condition \eqref{BC} is not trivial. In order to overcome loss of compactness, we need a procedure that allows to concentrate the energy around the origin. This will be made possible by translating each element $u_n$ of a given minimizing sequence $(u_n)_{n\in\N}$, in such a way that the transition between $u^-$ and $u^+$ is roughly achieved in a fixed neighborhood of the origin in $\Omega$.

\subsection{Some preliminaries}

We denote by $(e_i)_{i=1,\dots d}$ the canonical basis of $\R^d$, and we recall the notations $\R^d_a$ and $S_a$ in \eqref{slice}.

\paragraph{About the boundary condition.} Given $u\in\dot{H}_{div}^1(I\times\T^{d-1},\R^d)$, where $I\subset\R$ is an interval, we recall that $\overline{u}$ is the $x'$-average of $u$ on $\T^{d-1}$ defined in \eqref{average}. The following observation will be useful in the sequel.
\begin{lemma}
\label{lem_aver}
If $I\subset\R$ is an interval and $u\in\dot{H}^1(I\times\T^{d-1},\R^d)$, then $\overline{u}\in \dot{H}^1(I,\R^d)\subset \mathcal{C}^{0,1/2}(I,\R^d)$. If in addition $\nabla\cdot u=0$, then
\[
\overline{u_1}(t):=\overline{u}(t)\cdot e_1\quad\text{is constant in $I$.}
\]
In particular, if $I=\R$ and $\overline{u}(\pm\infty)=u^\pm\in\R^d_a$ for some $a\in\R$, then $\overline{u_1}(t)\equiv a$, i.e. $\overline{u}(t)\in\R^d_a$ for all $t\in I$. 
\end{lemma}

\begin{proof}
The fact that $\overline{u}\in\dot{H}^1(I, \R^d)$ immediately follows from the Jensen inequality:
\[
\int_I|\frac{\diff}{\diff x_1}\overline{u}(x_1)|^2\diff x_1=\int_I\left|\int_{\T^{d-1}}\partial_{1}u(x_1,x')\diff x'\right|^2\diff x_1\leq \int_{I\times\T^{d-1}}|\nabla u|^2\diff x.
\]
In particular, $\overline{u}\in \mathcal{C}^{0,1/2}(I,\R^d)$ by Sobolev embedding. When $\nabla\cdot u=0$, since $\T^{d-1}$ has no boundary, one has
\[
\frac{\diff}{\diff x_1}\overline{u_1}=\int_{\T^{d-1}}\partial_{1}u_1\diff x'=- 
\int_{\T^{d-1}} \nabla'\cdot u'\diff x' =0 \quad\mbox{a.e. in $I$,}
\]
where $$\textrm{$\nabla'=(\partial_2,\dots,\partial_d)$ and $u'=(u_2,\dots,u_d)$}.$$ This entails that $\overline{u_1}$ is constant.
\end{proof}

We will use the following standard result several times in the proofs of the one-dimensional symmetry of minimizers:
\begin{lemma}\label{BC_unif}
If $u\in \dot{H}^1(\Omega,\R^d)$ satisfies $\overline{u}(\pm\infty)=u^\pm$, then there exist two sequences $(R_n^+)_{n\in\N}$ and $(R_n^-)_{n\in\N}$ such that $(R_n^\pm)_{n\in\N}\to\pm\infty$ and
$$\| u(R_n^\pm,\cdot)- u^\pm\|_{H^1(\T^{d-1},\R^d)} \underset{n\to\infty}{\longrightarrow}0,$$
where $u(R,\cdot)$ stands for the trace of the Sobolev function $u$ at $x_1=R$, for every $R\in\R$. 
\end{lemma}

\begin{proof}
As the function $x_1\mapsto \|\nabla' u(x_1,\cdot)\|_{L^2(\T^{d-1})}^2$ is integrable over $\R$, one can find two sequences $(R_n^\pm)_{n\in\N}\to\pm\infty$ such that $\|\nabla' u(R_n^\pm,\cdot)\|_{L^2}$ tends to $0$ as $n\to \infty$. Moreover, by assumption, $\overline{u}(x_1)$ tends to $u^\pm$ as $x_1\to\pm\infty$. This finishes the proof as convergence of gradients and convergence in average implies convergence in $H^1(\T^{d-1})$, by the Poincar\'e-Wirtinger inequality.
\end{proof}
\begin{remark}\label{rem:BC}
If $f\in \dot{H}^1(\Omega)$ with $\bar{f}\in L^\infty(\R)$, then there exist two sequences $(R_n^\pm)_{n\in\N}\to\pm\infty$ such that the product 
$\| f(R_n^\pm,\cdot)\|_{L^2(\T^{d-1})} \|\partial_1 f(R_n^\pm,\cdot)\|_{L^2(\T^{d-1})}\to 0$ as $n\to +\infty$. The proof follows as above since by the Poincar\'e-Wirtinger inequality, the $L^2$ norm of $f(R^\pm_n,\cdot)$ is kept bounded.
\end{remark}

\paragraph{About the cost function $c_W$.} A fundamental observation in proving the existence of global minimizers is the following nondegeneracy property of $c_W$ defined in \eqref{cw}:
\begin{proposition}\label{cdW}
Let $W:\R^d\to\R_+$ be a continuous function and assume that {\rm \bf (H1)} and {\rm \bf(H2)} are satisfied for some $a\in\R$. Then, for all $\delta>0$, there exists $\varepsilon>0$ such that for all $x,y\in\R^d_a$,
\[
|x-y|\geq\delta \Longrightarrow c_W(x,y)\geq\varepsilon.
\]
\end{proposition}
In order to prove Proposition~\ref{cdW}, we need to estimate the energy from below. This can be done by averaging in the $d-1$ last variables. Namely, given an interval $I\subset \R$ and $u\in\dot{H}^1(I\times \T^{d-1},\R^d)$, the Jensen inequality yields
\[
E(u,I)= \int_{I}\int_{\T^{d-1}}\left(\frac 12 |\partial_{1}u|^2+\frac 12|\nabla' u|^2+W(u)\right)\diff x'\diff x_1\geq \int_I \frac 12\Big|\frac{\diff}{\diff x_1}\overline{u}(x_1)\Big|^2+e(u(x_1,\cdot))\diff x_1,
\]
where the energy density $e$ is defined by 
$$e(v):=\int_{\T^{d-1}} \frac 12|\nabla' v|^2+W(v)\diff x' \quad \textrm{ for all }\, v\in H^1(\T^{d-1},\R^d).$$ Thus, if in addition $\nabla \cdot u=0$ and $\overline{u_1}\equiv a$ in $I$, one has
\be
\label{inegEV}
E(u,I)\geq \int_I \frac 12\Big|\frac{\diff}{\diff t}\overline{u}(t)\Big|^2+V(\overline{u}(t))\diff t,
\ee
where $V:\R^d_a\to\R_+$ is defined for all $z\in\R^d_a$ by 
\be
\label{defV}
V(z):=\inf\left\{ e(v)\;:\; v\in H^1(\T^{d-1},\R^d),\, \textstyle\int_{\T^{d-1}}v=z \right\}\geq 0.
\ee
This observation is the starting point in the proof of the following lemma:
\begin{lemma}\label{lemmaV}
Let $W:\R^d\to\R_+$ be a continuous function and assume that {\rm \bf (H1)} and {\rm \bf(H2)} are satisfied for some $a\in\R$. Then 
the function $V:\R^d_a\to\R_+$ defined in \eqref{defV} satisfies the following:
\begin{enumerate}
\item
$V$ is lower semicontinuous in $\R^d_a$,
\item
for all $z\in\R^d_a$, $V(z)\leq W(z)$, the infimum in \eqref{defV} is achieved and $\Big[V(z)=0\Leftrightarrow W(z)=0\Big]$,
\item
$V_\infty:=\liminf\limits_{z\in\R^d_a,\,|z|\to\infty}V(z)>0$,
\item
for all interval $I\subset\R$ and for all $u\in \dot{H}_{div}^1(I\times\T^{d-1},\R^d)$ such that $\overline{u}(t)\in\R^d_a$ for all $t\in I$, one has
\[
E(u,I)\geq E_V(\overline{u},I):=\int_I \frac 12\Big|\frac{\diff}{\diff t}\overline{u}(t)\Big|^2+V(\overline{u}(t))\diff t.
\]
\end{enumerate}
\end{lemma}
\begin{proof}[Proof of Lemma \ref{lemmaV}] The claim {\it 4} follows from \eqref{inegEV}. We divide the rest of the proof in three steps. 

\medskip
\noindent\textsc{Step 1: proof of claim {\it 2}.} Clearly, for all $z\in\R^d_a$, one has $V(z)\leq e(z)=W(z)$. By the compact embedding $H^1(\T^{d-1})\hookrightarrow L^2(\T^{d-1})$, the continuity of $W$ and Fatou's lemma, the direct method in the calculus of variations implies that the infimum is achieved in \eqref{defV}. If $W(z)=0$, then $V(z)=0$ (as $V\leq W$ in $\R^d_a$). 
Conversely, if $V(z)=0$ with $z\in\R^d_a$, then a minimizer $v\in H^1(\T^{d-1},\R^d)$ in \eqref{defV} satisfies $V(z)=e(v)=0$ so that $v\equiv z$ and $W(z)=0$.

\medskip
\noindent\textsc{Step 2: $V$ is lower semicontinuous in $\R^d_a$.} Let $(z_n)_{n\geq 1}$ be a sequence converging to $z$ in $\R^d_a$. We need to show that
\[
V(z)\leq\liminf_{n\to\infty}V(z_n).
\] 
W.l.o.g.\footnote{Without loss of generality.}, one can assume that $(V(z_n))_{n\ge 1}$ is a bounded sequence that converges to $\liminf_{n\to\infty} V(z_n)$. By Step 1, for each $n\geq 0$, there exists $v_n\in H^1(\T^{d-1},\R^d)$ such that 
\[
\int_{\T^{d-1}}v_n=z_n\quad\text{and}\quad e(v_n)=V(z_n).
\]
Since $(z_n)_{n\ge 1}$ and $(e(v_n))_{n\ge 1}$ are bounded, $(v_n)_{n\geq 1}$ is bounded in $H^1(\T^{d-1},\R^d)$ by the Poincar\'e-Wirtinger inequality. Thus, up to extraction, one can assume that $(v_n)_{n\ge 1}$ converges weakly in $H^1$, strongly in $L^2$ and a.e. in $\T^{d-1}$ to a limit 
$v\in H^1(\T^{d-1},\R^d)$. In particular, $\int_{\T^{d-1}} v=z$. Since $W$ is continuous, by Fatou's Lemma and since the $L^2$ norm is lower semicontinuous in weak $L^2$-topology, we deduce that $e$ is lower semicontinuous in weak $H^1(\T^{d-1},\R^d)$-topology. Thus,
\[
V(z)\leq e(v)\leq \liminf\limits_{n\to\infty} e(v_n)= \liminf\limits_{n\to\infty} V(z_n).
\]
\noindent\textsc{Step 3: proof of claim {\it 3}.} Assume by contradiction that there exists a sequence $(z_n)_{n\geq 1}\subset\R^d_a$ such that $|z_n|\to\infty$ and $V(z_n)\to 0$ as $n\to\infty$. Then, there exists a sequence of maps $(w_n)_{n\geq 1}$ in $H^1(\T^{d-1},\R^d)$ satisfying
\[
\int_{\T^{d-1}}w_n(x')\diff x'=0\quad\text{for each $n\in\N$}\quad\text{and}\quad e(z_n+w_n)\underset{n\to\infty}{\longrightarrow} 0.
\]
By the Poincar\'e-Wirtinger inequality, we have that $(w_n)_{n\geq 1}$ is bounded in $H^1$. Thus, up to extraction, one can assume that it converges weakly in $H^1$, strongly in $L^1$ and a.e. to a function $w\in H^1(\T^{d-1},\R^d)$. We claim that $w$ is constant since
$$0=\liminf\limits_{n\to\infty}e(z_n+w_n)\geq \liminf\limits_{n\to\infty} \int_{\T^{d-1}} |\nabla w_n|^2\diff x'\geq \int_{\T^{d-1}} |\nabla w|^2\diff x'.$$
We deduce $w\equiv 0$ since $\int_{\T^{d-1}} w=\lim_{n\to\infty}\int_{\T^{d-1}} w_n=0$. Thus $w_n\to 0$ a.e and 
{\rm \bf{(H2)}} implies that for a.e.\ $x\in \T^{d-1}$,
$$\liminf_{n\to\infty}W(z_n+w_n(x))\geq \liminf\limits_{|z'|\to\infty ,\, z_1\to a} W(z_1,z')>0,$$
which contradicts the fact that $e(z_n+w_n)\to 0$.
\end{proof}

The following lemma provides an estimate from below of the energy by the geodesic distance $\mathrm{geod}_{V}^a$ in $\R^d_a$ endowed with the singular metric $2Vg_0$ (note that $V$ vanishes on $S_a$), $g_0$ being the standard euclidean metric in $\R^d_a$; this geodesic distance is defined for every $ x,y\in\R^d_a$ by
\begin{multline}
\label{estgeod}
\mathrm{geod}_{V}^a(x,y):=\inf\bigg\{\int^1_{ -1} \sqrt{2V(\sigma(t))}|\dot{\sigma}|(t)\diff t\;:\; \sigma\in \mathrm{Lip}([ -1,1],\R^d_a),\, \sigma( -1)=x,\, \sigma (1)=y\bigg\}.
\end{multline}
 
\begin{lemma}\label{cWgeqGeod}
Let $W:\R^d\to\R_+$ be a continuous function such that {\rm \bf (H1)} and {\rm \bf(H2)} are satisfied for some $a\in\R$ and let $V:\R^d_a\to\R_+$ be the function defined in \eqref{defV}. Then the function $\mathrm{geod}_{V}^a:\R^d_a\times\R^d_a\to\R_+$ is continuous, it defines a distance over $\R^d_a$ and
\[
c_W(x,y)\ge \mathrm{geod}_{V}^a(x,y)\quad\text{ for every $x,y\in\R^d_a$}.
\]
Moreover, for every $\delta>0$, there exists $\varepsilon>0$ such that for every $x,y\in \R^d_a$ with $|x-y|\ge \delta$, we have 
$\mathrm{geod}_{V}^a(x,y)\ge \varepsilon$.
\end{lemma}

\begin{proof}[Proof of Lemma \ref{cWgeqGeod}]
\noindent\textsc{Step 1: Proof of the inequality $c_W\ge \mathrm{geod}_{V}^a$.}
Indeed, by Lemma~\ref{lemmaV} (point 4.) and Young's inequality, one has for every $x,y\in\R^d_a$,
\begin{align*}
c_W(x,y)\geq \inf\bigg\{\int_I \sqrt{2V(\sigma(t))}|\dot{\sigma}|(t)\diff t&\;:\; I\subset\R\text{ interval,} \\ 
&\sigma\in \dot{H}^{1}(I,\R^d_a),\,\sigma(\inf I)=x,\, \sigma (\sup I)=y\bigg\}.
\end{align*}
Therefore, we only need to prove that the value of the above infimum remains unchanged if minimizing on a set of more regular curves, namely $\sigma\in\mathrm{Lip}([-1,1],\R^d_a)$. W.l.o.g. we assume that $I$ is an open interval; then, we take $\sigma\in \dot{H}^{1}(I,\R^d_a)\subset \dot{W}_{loc}^{1,1}(I,\R^d_a)$ and we define the arc-length $s:I\to J:=s(I)\subset\R$ by
\[
s(t):=\int_{t_0}^t|\dot{\sigma}|(t')\diff t',\quad t\in I,
\]
where $t_0\in I$ is some fixed instant. Then, the arc-length reparametrization of $\sigma$, i.e.
\[
\overline{\sigma}(s(t)):=\sigma(t),\quad t\in I,
\]
is well-defined and provides a Lipschitz curve $\overline{\sigma}:J\to\R^d_a$ with constant speed, i.e. $|\dot{\bar\sigma}|=1$ a.e., and such that $\overline{\sigma}(\inf J)=x$ and $\overline{\sigma}(\sup J)=y$. Moreover, the change of variables $s=s(t)$ yields
\[
\int_I \sqrt{2V(\sigma(t))}\, |\dot{\sigma}|(t)\diff t=\int_J \sqrt{2V(\overline{\sigma}(s))}\diff s = \int_J \sqrt{2V(\overline{\sigma}(s))}|\dot{\overline{\sigma}}|(s)\diff s.
\]
If \(J\) is unbounded, we take a small parameter $\varepsilon>0$ and we choose a compact interval $[a,b]\subset J$ such that
\begin{equation}
\label{boundaryCost}
|x-\overline{\sigma}(a)| \sup_{[x,\overline{\sigma}(a)]}\sqrt{2V}\ 
+ |\overline{\sigma}(b)-y| \sup_{[\overline{\sigma}(b),y]}\sqrt{2V} \le\varepsilon
\end{equation}
(here, we used the fact that $V$ is locally bounded in $\R^d_a$ as $V\leq W$ by Lemma \ref{lemmaV}) and we replace \(\overline{\sigma}\big|_{[\inf J,a]}\) (resp. \(\overline{\sigma}\big|_{[b,\sup J]}\))  by a constant-speed parametrization of the line segment $[x,\overline{\sigma}(a)]$ (resp. $[\overline{\sigma}(b),y]$). The resulting curve \(\Tilde{\sigma}:\Tilde{J}\subset\R\to\R^d_a\) still connects \(x\) to \(y\) and by \eqref{boundaryCost}, it satisfies
\begin{equation*}
\int_{\Tilde{J}} \sqrt{2V(\Tilde{\sigma}(s))}\, |\dot{\Tilde{\sigma}}|(s)\diff s\le  \int_{[a,b]}\sqrt{2V(\overline{\sigma}(s))}\, |\dot{\overline{\sigma}}|(s)\diff s+\varepsilon.
\end{equation*}
Last of all, by affine reparametrization, we can actually assume that $\Tilde{J}=[ -1,1]$; the desired inequality follows by arbitrariness of \(\varepsilon>0\).

\medskip
\noindent\textsc{Step 2: $\mathrm{geod}_{V}^a:\R^d_a\times\R^d_a\to\R_+$ defines a distance over $\R^d_a$}. The only non-trivial axiom to check is the non-degeneracy, i.e.,
$\mathrm{geod}_{V}^a(x,y)>0$ whenever $x\neq y$. Indeed, any continuous curve $\sigma:[0,1]\to\R^d_a$ such that $\sigma(0)=x$ and $\sigma(1)=y$ has to cross the ring
\[
\mathcal{C}_\eta(x):=\{z\in\R^d_a\;:\; \frac\eta 2\leq |z-x|\leq \eta\}
\] 
for any $\eta\in (0,|x-y|]$, thus implying the estimate
\begin{equation}
\label{ringEstimate}
\mathrm{geod}_{V}^a(x,y)\geq \frac \eta 2\inf_{z\in\mathcal{C}_\eta(x)} \sqrt{2V(z)}.
\end{equation}
Since Lemma~\ref{lemmaV} yields $V$ is lower semicontinuous and vanishes only on the finite set $S_a$ (by {\rm \bf (H1)}), one can find a small enough $\eta\in (0,|x-y|)$ such that $\mathcal{C}_\eta(x)\cap S_a=\emptyset$, so that $V$ is bounded from below by a positive constant on $\mathcal{C}_\eta(x)$ and thus, $\mathrm{geod}_{V}^a(x,y)>0$. 

\medskip
\noindent\textsc{ Step 3: There exist \(R,C\in (0,+\infty)\) such that for every $x,y\in\R^d_a$ with $|x|\ge R$ and $|y|\ge R$, one has
\begin{equation}
\label{coerciveGeod}
\mathrm{geod}_{V}^a(x,y)\ge C |x-y|.
\end{equation}} 
By Lemma \ref{lemmaV}, there exists \(R\in (0,+\infty)\) such that
\[
V(z)\ge \frac{V_\infty}{2}>0\quad\text{for every \(z\in\R^d_a\) with \(|z|\ge R/2\)}.
\]
We take \(x,y\in\R^d_a\) such that \(|x|\ge R\) and \(|y|\ge R\) and w.l.o.g., we assume \(|x|\ge |y|\) and $x\neq y$. Then we apply \eqref{ringEstimate} to \(\eta = \frac{|x-y|}{4}\in (0, \frac{|x|}{2}]\); noticing that for every \(z\in\mathcal{C}_\eta(x)\), one has \(|z|\ge |x|-\eta\ge R/2\), we obtain
\[
\mathrm{geod}^a_V(x,y)\ge \frac{|x-y|}{8}\inf_{|z|\ge \frac R2}\sqrt{2V(z)}\ge \frac{\sqrt{V_\infty}}{8}\, |x-y|.
\]
\noindent\textsc{Step 4: $\mathrm{geod}_{V}^a:\R^d_a\times\R^d_a\to\R_+$ is continuous.} Let $x,\tilde{x}, y, \tilde{y}\in \R^d_a\cap B$ for some ball $B\subset \R^d$. As $\mathrm{geod}_{V}^a$ is a distance on 
$\R^d_a$, then 
\begin{align*}
|\mathrm{geod}_{V}^a(x,y)-\mathrm{geod}_{V}^a(\tilde x,\tilde y)|&\leq |\mathrm{geod}_{V}^a(x,y)-\mathrm{geod}_{V}^a(\tilde x,y)|+|\mathrm{geod}_{V}^a(\tilde x,y)-\mathrm{geod}_{V}^a(\tilde x,\tilde y)|\\
&\leq \mathrm{geod}_{V}^a(x,\tilde x)+\mathrm{geod}_{V}^a(y,\tilde y).
\end{align*}
Letting the transition $\sigma$ be the segment $[x, \tilde x]$ in the definition \eqref{estgeod}, one gets $\mathrm{geod}_{V}^a(x,\tilde x)\leq \sup_B \sqrt{2V} \cdot |x-\tilde x|$ (idem when $(x, \tilde x)$ is replaced by 
$(y, \tilde y)$) and the conclusion follows since $V$ is locally bounded in $\R^d_a$ as $V\leq W$ by Lemma \ref{lemmaV}.

\medskip 
\noindent\textsc{Step 5: For every $\delta>0$, there exists $\varepsilon>0$ such that for every $x, y\in \R^d_a$, $|x-y|\ge\delta$ implies $\mathrm{geod}_V^a(x,y)\ge \varepsilon$.}
Assume by contradiction that there exist $\delta>0$ and two sequences $(x_n)_{n\ge 1}$ and $(y_n)_{n\ge 1}$ in $\R^d_a$ such that
\(
|x_n-y_n|\geq\delta
\)
for each $n\ge 1$ and 
\(
\lim_{n\to\infty}\mathrm{geod}_V^a(x_n,y_n)=0.
\) 
In particular, by \eqref{coerciveGeod}, the sequence \((\min\{|x_n|,|y_n|\})_{n\ge 1}\) is bounded; up to exchange \(x_n\) and \(y_n\), one can assume that \((x_n)_{n\ge 1}\) is bounded and up to extraction, one can assume that it has a subsequence converging to some \(x\in \R^d_a\). Fixing \(R,C>0\) such that \eqref{coerciveGeod} holds true and \(z_0\in\R^d_a\) such that \(|z_0|\ge R\), we obtain for every \(n\ge 1\) such that \(|y_n|\ge R\),
\[
C |y_n-z_0|\le \mathrm{geod}_{V}^a(y_n,z_0)\le\mathrm{geod}_{V}^a(y_n,x_n)+\mathrm{geod}_{V}^a(x_n,z_0)\underset{n\to\infty}{\longrightarrow} \mathrm{geod}_{V}^a(x,z_0).
\]
Thus, the sequence \((y_n)_{n\ge 1}\) is bounded as well so that it has a subsequence converging to some \(y\in\R^d_a\); by continuity of \(\mathrm{geod}^a_V\), we have \(\mathrm{geod}^a_V(x,y)=0\) and so \(x=y\), thus contradicting the fact that \(|x-y|=\lim_{n\to \infty}|x_n-y_n|\ge\delta\).
\end{proof}
\begin{proof}[Proof of Proposition~\ref{cdW}]
Proposition~\ref{cdW} immediately follows from Lemma \ref{cWgeqGeod}.
\end{proof}

We finish this preliminary section by the following lemma which will be useful in proving that the boundary constraint $\overline{u}(\pm\infty)=u^\pm$ is preserved by limits of minimizing sequences for $E$:
\begin{lemma}
\label{closed_boundary}
With the function $V$ and the energy $E_V$ given by Lemma~\ref{lemmaV}, assume that $\sigma \in \dot{H}^1(\R,\R^d_a)$ is a map with finite energy $E_V(\sigma,\R)<+\infty$. Then there exist $z^-,\, z^+\in S_a$ such that
$
\lim\limits_{t\to\pm\infty}\sigma(t)=z^\pm.
$
\end{lemma}
\begin{proof}
By Lemma~\ref{cWgeqGeod}, we know that $\mathrm{geod}_{V}^a:\R^d_a\times\R^d_a\to\R_+$ defines a distance on $\R^d_a$.
If the target space $\R^d_a$ of $\sigma$ is endowed with the distance $\mathrm{geod}_{V}^a$, then the estimate $E_V(\sigma,\R)<+\infty$ yields a bound on the total variation of $\sigma:\R\to (\R^d_a,\mathrm{geod}_{V}^a)$. Indeed, for every sequence $t_1\le\dots\le t_N$ in $\R$, by the Young inequality, we have
$$
\sum_{i=1}^N \mathrm{geod}_{V}^a(\sigma(t_{i+1}),\sigma(t_i))\le \sum_{i=1}^N E_V(\sigma,[t_i,t_{i+1}])\le E_V(\sigma,\R)<+\infty.
$$
In particular, for every $\varepsilon>0$, there exists $R>0$ such that for all $t,s\in\R$ with $t,s\ge R$ or $t,s\le -R$, one has $\mathrm{geod}_{V}^a(\sigma(t),\sigma(s))<\varepsilon$. By Lemma~\ref{cWgeqGeod}, it follows that for every $\delta>0$, there exists $\varepsilon>0$ such that $\mathrm{geod}_{V}^a(x,y)< \varepsilon$ implies $|x-y|<\delta$; thus, we deduce that $\sigma$ has a limit $z^\pm\in\R^d_a$ at $\pm\infty$. Since $V(\sigma(\cdot))$ is integrable in $\R$, we have furthermore that $V(z^\pm)=0$, i.e. $z^\pm\in S_a$.
\end{proof}

\subsection{The case of double-well potentials in $\R^d_a$. Proof of Theorems~\ref{thm1}~and~\ref{bndenergy}\label{two-wells}}

Given a continuous potential $W:\R^d\to\R_+$ with only two wells $u^\pm$ in $\R^d_a$ for some $a\in\R$, i.e., $S_a=\{u^\pm\}$, our aim is to prove existence of a solution to the minimization problem $(\mathcal P)$.
We will actually prove relative compactness (up to translation in $x_1$-direction) of admissible configurations with uniformly bounded energy (not only minimizing sequences) as stated in Theorem~\ref{bndenergy}. The proof for double-well potentials in $\R^d_a$ will use Proposition~\ref{cdW}, whereas the case of multiple-well potentials in $\R^d_a$ requires more precise estimates on the energy and the relative compactness only holds for minimizing sequences. 

\paragraph{Strategy for proving Theorem~\ref{bndenergy}.}
Since $E$ is lower semicontinuous on $\dot{H}^1(\Omega,\R^d)$ endowed with the weak convergence (i.e., $L^2$-weak convergence of gradients and strong $L^2_{loc}$-convergence of maps), and since boundedness of the energy implies boundedness of the $L^2$-norm of gradients, it is enough to prove that the boundary condition \eqref{BC} is preserved in the limit (up to translation in $x_1$-direction).
We will present two proofs of Theorem~\ref{bndenergy}. The first proof is based on the following Lemma~\ref{focus} which 
does not use the fact that $S_a=\{u^\pm\}$ but only the fact that $S_a$ is finite (i.e. {\rm \bf(H1)}) and
is somehow reminiscent from the compactness result \cite[Lemma 4.4.]{Goldman:2015}, while the second proof is based on \cite[Lemma 1.]{Doring:2013}.

\begin{lemma}\label{focus}
Let $W:\R^d\to \R_+$ be a continuous function and assume that {\rm \bf (H1)} and {\rm \bf (H2)} are satisfied for some $a\in\R$. Let $(u_n)_{n\ge 1}$ be a sequence in $\dot{H}_{div}^1(\Omega,\R^d)$ such that $\overline{u_n}(\pm\infty)=u^\pm\in S_a$ for each $n\ge 1$ and
\[
\sup_n E(u_n)=\sup_n \int_\Omega \frac 12 |\nabla u_n|^2+W(u_n)\diff x<\infty.
\]
If $\varepsilon>0$ is a small radius such that the closed balls $B_\pm :=\overline{B}(u^\pm,\varepsilon)\subset \R^d$ are disjoint, then there exist a sequence $(t_n)_{n\geq 1}\subset\R$ and $T\geq 1$, such that, up to a subsequence, one has for every $ n\ge \log_2T$,
\begin{align*}
\overline{u}_{n}(t)\notin B_-\quad\text{for all }\ t\in [ t_n+T, t_n+2^n]\quad\text{and}\quad
\overline{u}_{n}(t)\notin B_+\quad\text{for all }\ t\in [t_n-2^n, t_n-T].
\end{align*}
\end{lemma}

\begin{proof}[Proof of Lemma~\ref{focus}] 
\textsc{Step 1: study of the oscillations of $(\overline{u_n})_{n\ge 1}$.} For each fixed $n\ge 1$, let us build a sequence of intervals $(I^n_k)_{k\ge 1}$ by induction as follows (see Figure \ref{oscillation}):

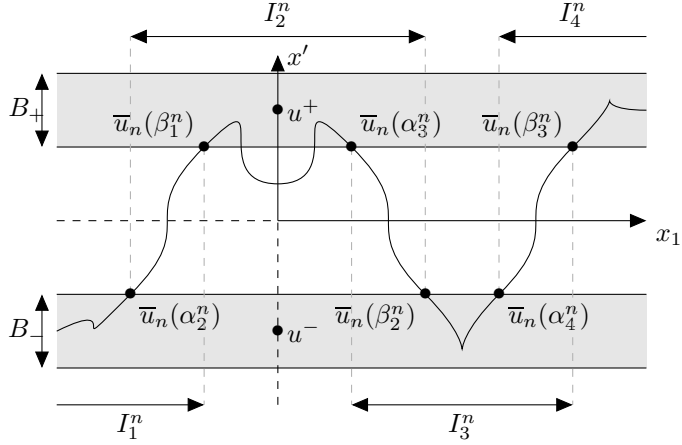
\begin{figure}[h]
\centering
\begin{tikzpicture}[scale=0.7,>=triangle 45,x=1.4cm,y=1.4cm]

\def\r{0.5}

\fill[color=gray!20]
(-3,1.5-\r)--(5,1.5-\r)--(5,1.5+\r)--(-3,1.5+\r)--cycle;

\fill[color=gray!20]
(-3,-1.5-\r)--(5,-1.5-\r)--(5,-1.5+\r)--(-3,-1.5+\r)--cycle;

\draw (0,1.5) node {$\bullet$} node[right] {$u^+$}; 
\draw (0,-1.5) node {$\bullet$} node[right] {$u^-$};
\draw (-3,1.5-\r)--(5,1.5-\r);
\draw (-3,1.5+\r)--(5,1.5+\r);
\draw (-3,-1.5-\r)--(5,-1.5-\r);
\draw (-3,-1.5+\r)--(5,-1.5+\r);

\draw[<->] (-3-0.2,1.5-\r)--(-3-0.2,1.5+\r);
\draw (-3-0.4,1.5) node {$B_+$};
\draw (-3-0.4,-1.5) node {$B_-$};
\draw[<->] (-3-0.2,-1.5-\r)--(-3-0.2,-1.5+\r);

\draw[->] (0,0)--(5,0) node[below right] {$x_1$};
\draw[->] (0,0)--(0,1.5+1.5*\r) node[right] {$x'$};
\draw[dashed] (0,-1.5-2*\r)--(0,0);
\draw[dashed] (-3,0)--(0,0);

\draw (-3,-1.5) .. controls +(1,0.5) and +(-1,-1)  .. (-2,-1.5+\r);
\draw[dashed, color=gray!60] (-2,-1.5+\r) -- (-2,1.5+2*\r);
\draw (-2,-1.5+\r) node {$\bullet$} node[below right] {$\overline{u}_n(\alpha^n_2)$};

\draw (-2,-1.5+\r) .. controls +(1,1) and +(-1,-1) .. (-1,1.5-\r);
\draw[dashed, color=gray!60] (-1,1.5-\r) -- (-1,-1.5-2*\r);
\draw (-1,1.5-\r) node {$\bullet$} node[above left] {$\overline{u}_n(\beta^n_1)$};

\draw  (-1,1.5-\r) .. controls +(1,1) and +(-1,0) .. (0,0.5);
\draw (0,0.5) .. controls +(1,0) and +(-1,1) .. (1,1.5-\r);
\draw[dashed, color=gray!60] (1,1.5-\r)  -- (1,-1.5-2*\r);
\draw (1,1.5-\r) node {$\bullet$} node[above right] {$\overline{u}_n(\alpha^n_3)$};

\draw (1,1.5-\r) .. controls +(1,-1) and +(-1,1) .. (2,-1.5+\r);
\draw[dashed, color=gray!60] (2,-1.5+\r) -- (2,1.5+2*\r);
\draw (2,-1.5+\r) node {$\bullet$} node[below left] {$\overline{u}_n(\beta^n_2)$};

\draw (2,-1.5+\r) .. controls +(1,-1) and +(-1,-1) .. (3,-1.5+\r);
\draw[dashed, color=gray!60] (3,-1.5+\r) -- (3,1.5+2*\r);
\draw (3,-1.5+\r) node {$\bullet$} node[below right] {$\overline{u}_n(\alpha^n_4)$};

\draw (3,-1.5+\r) .. controls +(1,1) and +(-1,-1) .. (4,1.5-\r);
\draw[dashed, color=gray!60] (4,1.5-\r) -- (4,-1.5-2*\r);
\draw (4,1.5-\r) node {$\bullet$} node[above left] {$\overline{u}_n(\beta^n_3)$};

\draw (4,1.5-\r) .. controls +(1,1) and +(-1,0) .. (5,1.5);

\draw[->] (-3,-1.5-2*\r)--(-1,-1.5-2*\r);
\draw (-2,-1.5-2*\r) node[below] {$I^n_1$};

\draw[<->] (-2,1.5+2*\r)--(2,1.5+2*\r);
\draw (0,1.5+2*\r) node[above] {$I^n_2$};

\draw[<->] (1,-1.5-2*\r)--(4,-1.5-2*\r);
\draw (2.5,-1.5-2*\r) node[below] {$I^n_3$};

\draw[<-] (3,1.5+2*\r)--(5,1.5+2*\r);
\draw (4,1.5+2*\r) node[above] {$I^n_{4}$};
\end{tikzpicture}
\caption{Possible trajectory for $\overline{u}_n$ and the corresponding $I^n_k$}
\label{oscillation}
\end{figure}

\begin{itemize}
	\item $I^n_1:=(\alpha_1^n,\beta_1^n)$, where $\alpha_1^n=-\infty$ and $\beta_1^n<\infty$ is the first instant for which $\overline{u}_n(\beta_1^n)\in \partial B_+$. In other words, $I^n_1$ is the first maximal interval in $\overline{u}_n^{-1}(B_+^c)$ (which exists since $\overline{u}_n(-\infty)=u^-$), where $B_+^c:=\R^d\setminus B_+$.
	\item $I^n_2=(\alpha_2^n,\beta_2^n)$ is the maximal interval in $\overline{u}_n^{-1}(B_-^c)$ containing $\beta_1^n$. Thus, either $\overline{u}_n(\beta_2^n)\in \partial B_-$ or $\beta_2^n=+\infty$.
	\item Given $k\geq 2$, assume that $I^n_k=(\alpha_k^n,\beta_k^n)$ has been constructed, and that $\beta_k^n<+\infty$. Then we define $I^n_{k+1}=(\alpha^n_{k+1},\beta^n_{k+1})$ as the maximal interval in $\overline{u}_n^{-1}(B^c)$ containing $\beta_k^n$, where either $B=B_-$ if $\overline{u}_n(I^n_k)\subset B_+^c$ or $B=B_+$ if $\overline{u}_n(I^n_k)\subset B_-^c$. Thus, either $\overline{u}_n(\beta_{k+1}^n)\in \partial B$ or ($B=B_-$ and $\beta^n_{k+1}=+\infty$).
\end{itemize}
This induction stops at the first iteration step $k_n\in\mathbb{N}\cup\{+\infty\}$ for which $\beta_{k_n}^n=+\infty$. Note that $k_n\geq 2$, for every $n$.

\medskip
\noindent\textsc{Step 2: $(k_n)_{n\ge 1}$ is a bounded sequence.} Indeed, by construction for $n$ fixed, the $\alpha_k^n$ and $\beta_k^n$ are ordered as follows:
\[
\alpha_1^n=-\infty<\alpha_2^n<\beta_1^n\leq\alpha_3^n<\beta_2^n\leq\alpha_4^n<\beta_3^n\leq\alpha_5^n\leq\dots<\beta_{k_n}^n=+\infty.
\]
In particular, for every index $k\in (1,k_n)$, one has $\alpha_k^n<\beta_{k-1}^n\leq\alpha_{k+1}^n<\beta_k^n$. Moreover, by construction of $I^n_k$, we know that either $\overline{u}_n(\beta_{k-1}^n)\in \partial B_+$ and both $\overline{u}_n(\alpha_k^n)$ and $\overline{u}_n(\beta_k^n)$ belong to $\partial B_-$, or $\overline{u}_n(\beta_{k-1}^n)\in \partial B_-$ and both $\overline{u}_n(\alpha_k^n)$ and $\overline{u}_n(\beta_k^n)$ belong to $\partial B_+$; in other words, $\overline{u}_n$ makes two transitions between $B_-$ and $B_+$ (one on $(\alpha^n_k,\beta^n_{k-1})$ and the other on $(\beta^n_{k-1},\beta^n_k)$). In particular, since the intervals $(I^n_{2k})_{1<2k<k_n}$ are disjoint, one has
\[
2c_W(B_-,B_+)\left\lfloor{\frac{k_n-1}{2}}\right\rfloor\leq \sum_{1<2k<k_n}E(u_n,I^n_{2k})\leq \sup_n E(u_n,\R)<\infty,
\]
where $\left\lfloor{\frac{k_n-1}{2}}\right\rfloor$ is the integer part of $\frac{k_n-1}{2}$, and
\[
c_W(B_-,B_+):=\inf\{c_W(x,y)\;:\; x\in B_-,\, y\in B_+\}.
\]
By Proposition~\ref{cdW}, one has $c_W(B_-,B_+)>0$, and so $\sup_n k_n<\infty$. 

\medskip
\noindent\textsc{Step 3.} We prove that there exist two indices $k_0$, $l_0$ and an unbounded set $X\subset\N^\ast$ (corresponding to the indices of a subsequence of $(u_n)_{n\ge 1}$) such that
\begin{itemize}
\item
for all $n\in X$, $1\leq k_0<l_0\leq k_n$,
\item
$({\rm Length}(I^n_{k_0}))_{n\in X}$ and $({\rm Length}(I^n_{l_0}))_{n\in X}$ converge to $+\infty$ as $n\to +\infty$,
\item
for all $n\in X$, $\overline{u}_n(I^n_{k_0})\subset B_+^c$ and $\overline{u}_n(I^n_{l_0})\subset B_-^c$,
\item
$(|\alpha^n_{l_0}-\beta^n_{k_0}|)_{n\in X}$ is bounded.
\end{itemize} 
In order to prove existence of $k_0,\, l_0$, we define a finite sequence $(\sigma_k)_{1\leq k\leq K}\subset\{0,+,-\}$ as follows. We first pick $K\geq 2$ to be a value that repeats infinitely many times in the sequence $(k_n)_{n\ge 1}$ (that is bounded in $\N$). We set $X$ to be the set of those (infinitely many) indices $n$ with $k_n=K$; then one has $I^n_K=(\alpha^n_K,+\infty)$ for every $n\in X$. Then, for each $k\in\N$ with $1\le k\le K$, we set $\sigma_k$ by the following algorithm (the set $X$ might change at some steps):
\begin{itemize}
\item
$\sigma_k:=0$ if the sequence $({\rm Length}(I^n_k))_{n\in X}$ is bounded;
\item
$\sigma_k:=+$ if there exists a sequence $(n_j)_{j\in \N}\subset X$ such that ${\rm Length}(I^{n_j}_k)\to \infty$ as $j\to \infty$ and $\overline{u}_{n_j}(I^{n_j}_k)\subset B_-^c$ for every $j\in \N$ (in this case, $X$ {\bf is replaced by} the sequence $(n_j)_j$);
\item
$\sigma_k:=-$ if there exists a sequence $(n_j)_{j\in \N}\subset X$ such that ${\rm Length}(I^{n_j}_k)\to \infty$ as $j\to \infty$ and $\overline{u}_{n_j}(I^{n_j}_k)\subset B_+^c$ for every $j\in \N$ (in this case, $X$ {\bf is replaced by} the sequence $(n_j)_j$).
\end{itemize}
Clearly, one has $\sigma_1=-$ and $\sigma_K=+$. Thus, the sequence $(\sigma_k)_{1\leq k\leq K}$ contains at least a subsequence $(\sigma_{k})_{k_0\leq k\leq l_0}$ of the form $(-,0,\dots,0,+)$. This means that $(I^n_{k_0})_{n\in X}$ and $(I^n_{l_0})_{n\in X}$ are unbounded, and that the intermediate interval between $I^n_{k_0}$ and $I^n_{l_0}$ is of uniformly bounded length, i.e. $(|\alpha^n_{l_0}-\beta^n_{k_0}|)_{n\in X}$ is bounded. Moreover, by construction, $\overline{u}_n(I^n_{k_0})\subset B_+^c$ and $\overline{u}_n(I^n_{l_0})\subset B_-^c$ for all $n\in X$.

\medskip
\noindent\textsc{Step 4: end of the proof.} By Step 3, the conclusion of Lemma~\ref{focus} holds true, up to a subsequence, with the choice $t_n=\beta_{k_0}^n$ (or alternatively, take $t_n=\alpha_{l_0}^n$), and $T=\max\{1, \sup_n |\alpha^n_{l_0}-\beta^n_{k_0}|\}$.
\end{proof}

\begin{proof}[First Proof of Theorem~\ref{bndenergy}]
Let us take $(t_n)_{n\ge 1}\subset\R$ and $T$ such that the conclusion of Lemma~\ref{focus} holds true: up to a subsequence, there exists $\varepsilon>0$ such that for each \(n\ge \log_2T\),
\begin{equation}\label{weakBC}
\overline{u}_{n}(t+t_n)\notin\overline{B}(u^-,\varepsilon)\ \text{ for all } t\in [T,2^n]\quad\text{and}\quad\overline{u}_{n}(t+t_n)\notin\overline{B}(u^+,\varepsilon)\ \text{ for all } t\in [-2^n,-T].
\end{equation}
Since $(u_n(\cdot+t_n, \cdot))_{n\ge 1}$ is bounded in $\dot{H}_{div}^1(\Omega,\R^d)$, up to a subsequence, it has a weak limit $u\in \dot{H}_{div}^1(\Omega,\R^d)$. In particular, by the Sobolev embedding $\dot{H}^1(\R,\R^d_a)\hookrightarrow \mathcal{C}^{ 0,\frac12}(\R,\R^d_a)$, one has $(\overline{u}_n(\cdot+t_n))_{n\ge 1}\to \overline{u}$ weakly in $\dot{H}^1$ and uniformly on compact subsets of $\R$. From \eqref{weakBC}, one deduces
\begin{equation}
\label{weakBCu}
\overline{u}(t)\notin B(u^-,\varepsilon)\quad\text{for all } t\geq T\quad\text{and}\quad\overline{u}(t)\notin B(u^-,\varepsilon)\quad \text{for all } t\leq -T.
\end{equation}
Now, from Lemma~\ref{lemmaV} and by lower semicontinuity of $E_V$ in weak $\dot{H}^1$-topology, we learn that
$$
E_V(\overline{u},\R)\le\liminf_{n\to\infty}E_V(\overline{u}_n,\R)\le \liminf_{n\to\infty}E(u_n)<+\infty.
$$
In particular, by Lemma~\ref{closed_boundary}, $\overline{u}$ has a limit $z^\pm\in S_a$ at $\pm\infty$. But \eqref{weakBCu} forces $z^\pm=u^\pm$ 
since \footnote{This is the only place where {\rm \bf (H1')} is needed instead of {\rm \bf (H1)} in the proof of Theorem \ref{bndenergy}.} $S_a=\{u^-,u^+\}$; thus, \eqref{BC} holds true. Since $E$ is lower semicontinuous in weak $\dot{H}^1(\Omega,\R^d)$-topology, the proof is now complete.
\end{proof}
We point out a second proof, based on the following compactness result \cite[Lemma 1.]{Doring:2013}, which can be seen as a generalization of Lemma~\ref{focus} in terms of the average sequence $\{\bar u_n\}$:
\begin{lemma}[L. D\"oring, R. Ignat, F. Otto \cite{Doring:2013}]\label{doring}
Let $(v_n)_{n\geq 1}$ be a sequence of scalar functions $v_n:\R\to\R$ uniformly bounded in $\dot{H}^1(\R)$, i.e., $\sup_n\|\dot{v}_n\|_{L^2(\R)}<\infty$, and such that
\[
\liminf\limits_{t\to +\infty} v_n(t)> 0\quad\text{and}\quad\limsup\limits_{t\to -\infty} v_n(t)< 0\quad\text{for each $n\geq 1$}. 
\]
Then up to a subsequence, there exist $(t_n)_{n\geq 1}\subset\R$ and $v\in\dot{H}^1(\R)$ such that $v_n(\cdot +t_n)\to v$ weakly in $\dot{H}^1(\R)$ with  $\liminf\limits_{t\to +\infty} v(t)\geq 0$ and $\limsup\limits_{t\to -\infty} v(t)\leq 0$.
\end{lemma}
\begin{proof}[Second proof of Theorem~\ref{bndenergy}] 
Let $(u_n)_{n\geq 1}$ be a sequence as in Theorem~\ref{bndenergy}, and apply Lemma~\ref{doring} to $(v_n)_{n\geq 1}$ given by
\[
v_n(t):=\left(\overline{u_n}(t)-\frac 12(u^++u^-)\right)\cdot (u^+-u^-),\quad n\geq 1,\, t\in\R.
\]
It is clear that $(v_n)_{n\geq 1}$ satisfies the assumptions in Lemma~\ref{doring} since $(\nabla u_n)_{n\ge 1}$ is bounded in $L^2$ and 
\[
\lim_{t\to\pm\infty}v_n(t)=\pm\frac 12 |u^+-u^-|^2. 
\]
Thus there exist $(t_n)_{n\geq 1}\subset\R$ and $v\in\dot{H}^1(\R)$ such that $v_n(\cdot +t_n)\to v$ weakly in $\dot{H}^1(\R)$ with  $\liminf\limits_{t\to +\infty} v(t)\geq 0$ and $\limsup\limits_{t\to -\infty} v(t)\leq 0$.

Moreover, as before, $(u_n(\cdot +t_n))_{n\ge 1}$ converges weakly in $\dot{H}^1$ to some limit $u\in \dot{H}^1_{div}(\Omega,\R^2)$ and we have also $(\overline{u_n}(\cdot +t_n))_{n\ge 1}\to\overline{u}$ weakly in $\dot{H}^1$. The fact that $\overline{u}(\pm\infty)=u^\pm$ is shown exactly as in the first proof of Theorem~\ref{bndenergy}, except that instead of \eqref{weakBCu} we use the following properties of the limit $v$:
\[
v(t)=\left(\overline{u}(t)-\frac 12(u^++u^-)\right)\cdot (u^+-u^-)\quad \textrm{for all } t\in \R, \quad \textrm{and} \quad \liminf\limits_{t\to\pm\infty}\ \pm v(t)\geq 0.
\]
By Lemma \ref{closed_boundary}, we conclude the second proof of Theorem~\ref{bndenergy}.
\end{proof}

\begin{proof}[Proof of Theorem~\ref{thm1}]
There exists a minimizing sequence with uniformly bounded energy (note that we have $E(u)<+\infty$ for some $1D$ smooth transition $u=u(x_1):\R\to\R^d_a$ between $u^-$ and $u^+$). Thus, Theorem~\ref{thm1} is a consequence of Theorem~\ref{bndenergy}.
\end{proof}

\subsection{The case of multiple-well potentials in  $\R^d_a$. Proof of Theorem~\ref{thm2}\label{MultipleWell}}
The proof of Theorem~\ref{thm2} relies on Lemma~\ref{focus} (which do not use Assumptions \textbf{(H3)} and \textbf{(H4)}) and the following Lemma~\ref{time_estimate} which aims to prevent lack of compactness due to the presence of a third well $z_0\in S_a\setminus\{u^-,u^+\}$ and uses also the assumptions \textbf{(H3)} and \textbf{(H4)}:
\begin{lemma}\label{time_estimate}
Let $W:\R^d\to \R_+$ be a continuous function such that {\rm \bf (H1)} - {\rm \bf (H4)} are satisfied for some $a\in\R$ and let $(u_n)_{n\ge 1}$ be a minimizing sequence, i.e. $u_n\in \dot{H}_{div}^1(\Omega,\R^d)$, $\overline{u_n}(\pm\infty)=u^\pm\in S_a$ for each $n\ge 1$, and
\[
\lim_{n\to\infty}\int_\Omega \frac 12 |\nabla u_n|^2+W(u_n)\diff x=\inf\left\{E(u)\;:\; u\in \dot{H}_{div}^1(\Omega,\R^d)\text{ with }\overline{u}(\pm\infty)=u^\pm \right\}.
\]
Then, for all $\delta< \min\{|x-y|\;:\; x,y\in S_a, \, x\neq y\}$ and for all $z\in S_a\setminus\{u^-,u^+\}$, one has $\sup_{n\ge 1} \mathcal{L}^1\big(\overline{u}_n^{-1}(B(z,\delta))\big)<+\infty$, where $\mathcal{L}^1$ stands for the Lebesgue  measure in $\R$.
\end{lemma}
\begin{proof}[Proof of Lemma \ref{time_estimate}]
Assume by contradiction that there exist $\delta<\inf\{|x-y|\;:\; x,y\in S_a, \, x\neq y\}$, a well $z_0\in S_a\setminus\{u^-,u^+\}$ and a subsequence $(u_{n_k})_{k\ge 1}$ such that $\mathcal{L}^1\big(I_k\big)\to +\infty$ as $k\to \infty$, where $I_k:=\overline{u}_{n_k}^{-1}(B(z_0,\delta))$ is an open set. Since
$$
\int_{I_k} \left(\int_{\T^{d-1}} \Big(\frac 12 |\nabla' u_{n_k}(x_1,x')|^2+W(u_{n_k}(x_1,x'))\Big)\diff x'\right)\diff x_1\le \sup_{n\ge 1} E(u_n)<+\infty,
$$
we deduce the existence of a sequence $(t_k)_{k\ge 1}$ such that for each $k\ge 1$, $t_k\in I_k$ and 
$$
\int_{\T^{d-1}} \Big(\frac 12 |\nabla' u_{n_k}(t_k,x')|^2+W(u_{n_k}(t_k,x'))\Big)\diff x'\underset{k\to\infty}{\longrightarrow}0.
$$
Since we have furthermore $\overline{u}_{n_k}(t_k)\in B(z_0,\delta)\subset (\R^d_a\setminus S_a)\cup\{z_0\}$, the sequence $(u_{n_k}(t_k,\cdot))_{k\ge 1}$ converges to a constant $z\in\R^d_a$ strongly in $H^1(\T^{d-1},\R^d)$, and this constant belongs to $S_a$ (since $W(z)=0$ by Fatou's Lemma) so that we have necessarily $z=z_0$.  Moreover, for each $k$, by Lemma~\ref{BC_unif} applied to $u_{n_k}$, there exist $R_k^-$ and $R_k^+$ such that $R^-_k<t_k<R^+_k$, with $\|u_{n_k}(R^\pm_k,\cdot)-u^\pm\|_{H^1}\to 0$ as $k\to\infty$. Then, by definition of $d_W$, we get
\begin{align*}
E(u_{n_k})\geq d_W(u_{n_k}(R^-_k,\cdot),u_{n_k}(t_k,\cdot))+d_W(u_{n_k}(t_k,\cdot),u_{n_k}(R^+_k,\cdot));
\end{align*}
since $(u_n)_{n\ge 1}$ is a minimizing sequence and by {\textbf{(H4)}}, we obtain in the limit $k\to\infty$,
\[
c_W(u^-,u^+)\geq d_W(u^-,z_0)+d_W(z_0,u^+),
\]
thus contradicting Hypothesis {\textbf{(H3)}} since $d_W(u^-,u^+)\ge c_W(u^-,u^+)$ (by definitions \eqref{cw} and \eqref{dw}).
\end{proof}

\begin{proof}[Proof of Theorem \ref{thm2}]
Let $(u_n)_{n\geq 1}\subset \dot{H}^1(\Omega,\R^d)$ be a minimizing sequence for the minimization problem $(\mathcal{P})$, i.e.
\[
+\infty>E(u_n)\underset{n\to\infty}{\longrightarrow}c_W(u^-,u^+),\quad 
\nabla\cdot u_n=0\quad
\text{and}\quad\overline{u_n}(\pm\infty)=u^\pm.
\]
From the first proof of Theorem~\ref{bndenergy} (based on Lemmas \ref{closed_boundary} and \ref{focus} that use only the assumptions {\rm \bf (H1)} and {\rm \bf (H2)}), we learn that $(u_n)_{n\ge 1}$ converges up to a subsequence (and up to translation) weakly in $\dot{H}^1(\Omega,\R^d)$ to a limit $u$ satisfying $\overline{u}(\pm\infty)=z^\pm\in S_a$. For potentials $W$ with more than two wells, we cannot deduce at this stage $z^\pm=u^\pm$; however, we can 
assert that $z^-\neq u^+$ and $z^+\neq u^-$ (thanks to Lemma \ref{focus}). The conclusion will follow from Lemma~\ref{time_estimate}. Indeed, by Fatou's lemma, we have for all $\delta>0$ such that $\delta<|x-y|$ whatever $x,y\in S_a$ with $x\neq y$, and $z\in S_a\setminus\{u^-,u^+\}$,
$$
\int_\R\mathbf{1}_{B(z,\delta)}(\overline{u}(t))\diff t
 \le\liminf_{n\to\infty}\int_\R\mathbf{1}_{B(z,\delta)}(\overline{u}_n(t))\diff t
= \liminf_{n\to\infty} \mathcal{L}^1\big(\overline{u}_n^{-1}(B(z,\delta))\big)<+\infty.
$$
In particular, $ \mathcal{L}^1\big(\overline{u}^{-1}(B(z,\delta))\big)$ is finite and $\overline{u}$ cannot converge to $z\in S_a\setminus\{u^-,u^+\}$ at $\pm \infty$. We have thus proved $\overline{u}(\pm\infty)=z^\pm=u^\pm$ and $u$ is a solution to the minimization problem $(\mathcal{P})$ since $E(u)\le \liminf_{n\to\infty}E(u_n)$ by lower semicontinuity of $E$.
\end{proof}

\subsection{Analysis of the transition cost. Proof of Propositions \ref{dWcW} and \ref{dWdist}}\label{lscTransitionCost}
\begin{proof}[Proof of Proposition~\ref{dWcW}]
By definitions \eqref{cw} and \eqref{dw}, we have $d_W(z^-,z^+)\ge c_W(z^-,z^+)$ for all $z^\pm\in S_a$ (there is no need of any assumption on $W$, in particular, no need of {\rm \bf (H4)}). Indeed, $d_W$ is defined by minimizing $E(u,I)$ on finite intervals $I$ with Dirichlet conditions $u=z^\pm$ on the boundary of $I\times\T^{d-1}$; extending $u$ by setting $u=z^\pm$ out of $I\times\T^{d-1}$ ($z^-$ at the left side and $z^+$ at the right side of $I$) yields an admissible function $\tilde{u}$ with the same energy $E(\tilde{u})=E(u,I)$ since $W(z^\pm)=0$, and then $E(u,I)\ge c_W(z^-,z^+)$.

Conversely, we now prove that $d_W(z^-,z^+)\le c_W(z^-,z^+)$ for all $z^\pm\in S_a$ under the assumption {\rm \bf{(H4)}}. Indeed, let $u\in\dot{H}_{div}^1(\Omega,\R^d)$ be a map such that $E(u)<+\infty$ and $\overline{u}(\pm\infty)=z^\pm$. Thanks to Lemma~\ref{BC_unif}, there exist two sequences $(R_n^\pm)_{n\ge 1}\to\pm\infty$ with $(\| u(R_n^\pm,\cdot)- z^\pm\|_{H^1})_{n\ge 1} \to 0$ and it follows from \textbf{(H4)},
$$
d_W(z^-,z^+)\leq \liminf_{n\to\infty}d_W(u(R_n^-,\cdot),u(R_n^+,\cdot))\le\liminf_{n\to\infty}E(u,[R_n^-,R_n^+])\le E(u).
$$
The conclusion follows by taking the infimum over $u$.
\end{proof}
Proposition~\ref{dWdist} is a consequence of the following technical but standard lemma:
{\begin{lemma}\label{Main_est}
If $W:\R^d\to\R_+$ is a continuous potential satisfying \eqref{growthMultiwell}, then the quantity 
\[
J(v_0,v):=\inf\big\{E(u,[0,1])\; :\; u\in H^1([0,1]\times\T^{d-1}),\,\nabla\cdot u=0,\, u(0,\cdot)=v_0,\, u(1,\cdot)=v\big\},
\]
defined for every $v_0\in S_a$ and every $v\in H^1(\T^{d-1},\R^d)$ such that $\int_{\T^{d-1}} v\cdot e_1=a$, satisfies $(J(v_0, {v_n}))_{n\ge 1}\to 0$ for every sequence $(v_n)_{n\ge 1}\to v_0$ in $H^1(\T^{d-1},\R^d)$ with $\int_{\T^{d-1}} v_n\cdot e_1=a$ for each \(n\ge 1\).
\end{lemma}}
\begin{proof}
Up to replacing $v_n$ by $v_n+v_0$, $u$ by $u+v_0$ (in the infimum defining $J(v_0,v)$) and $W(\cdot)$ by $W(v_0+\cdot)$, one can assume that $v_0=0\in S_a$ with $a=0$. So, let $(v_n)_{n\ge 1}$ be a sequence converging to $0$ in $H^1(\T^{d-1},\R^d)$ such that for each $n$, $\int_{\T^{d-1}}e_1\cdot v_n=0$. We shall prove $J(0,v_n)\to 0$ as $n\to\infty$. To this aim, we look for an admissible map of the form $u_n=Pv_n - w_n$,
where $P:H^1(\T^{d-1},\R^d)\to H^{\frac32^-}([0,1]\times\T^{d-1},\R^d)$ is the harmonic extension operator such that for all $v\in H^1(\T^{d-1},\R^d)$, $Pv(0,\cdot)=0$, $Pv(1,\cdot)=v$ and $\|Pv\|_{H^{\frac32^-}}\le C\|v\|_{H^1}$, where $\frac32^-$ stands for a real number strictly less than $\frac32$.  
In order that $\nabla\cdot u_n=0$, we impose on $w_n$ the following conditions:
\begin{equation}
\label{brezis_div}
\begin{cases}
\nabla \cdot w_n=f_n&\text{in }Q:=[0,1]\times\T^{d-1},\\
w_n=0&\text{on }\partial Q=\{0,1\}\times\T^{d-1},
\end{cases}
\end{equation}
where $f_n=\nabla\cdot (Pv_n)$. By Sobolev embedding, for any $p<\frac{2d}{d-1}$, there exist some constants 
$C_1,C_2,C_3>0$, two real numbers $\frac12^-\in (0, \frac12)$ and $\frac32^-\in(1+\frac12^-, \frac32)$
(all depending on $p$) such that
\[
\|f_n\|_{L^p(Q)}\leq C_1 \|f_n\|_{ H^{1/2^-}(Q)}\leq C_2 \|Pv_n\|_{H^{3/2^-}(Q)}\leq C_3\|v_n\|_{H^1(\T^{d-1})}.
\]
It is known (see for instance  \cite[Theorem 2]{Bourgain:2003}) that there exists a solution $w_n$ of \eqref{brezis_div} such that 
\begin{equation*}
\|w_n\|_{W^{1,p}(Q)}\leq C_4 \|f_n\|_{L^{p}(Q)}, 
\end{equation*}
where $C_4>0$ is a constant only depending on $d$. Since $v_n\to 0$ in $H^1(\T^{d-1})$, then $f_n\to 0$ in $L^{p}
(Q)$, and $w_n\to 0$ in $W^{1,p}(Q)$. Moreover $Pv_n\to 0$ in $H^{3/2-}(Q)$, and $H^{3/2-}(Q)$ is continuously embedded in $W^{1,p}(Q)$. Thus, we have proved that $u_n=Pv_n-w_n\to 0$ in $W^{1,p}(Q)$. Since $E(u_n,[0,1])=\frac12\|\nabla u_n\|_{L^2(Q)}^2+\|W(u_n)\|_{L^1(Q)}$ and we can choose $p\geq 2$, it remains to prove that $W(u_n)\to 0$ in $L^1(Q)$.

In dimension $d=2$, one can choose $p\in (2,4)$ so that $W^{1,p}(Q)$ is continuously embedded in $\mathcal{C}^{0,\alpha}$ for some $\alpha>0$; thus, both $(u_n)_{n\ge 1}$ and $(W(u_n))_{n\ge 1}$ converge uniformly to $0$, in particular, $W(u_n)\to 0$ in $L^1(Q)$. In dimension $d\geq 3$, 
for $q$ given in \eqref{growthMultiwell}, we can choose $p$ close to $\frac{2d}{d-1}$ so that $W^{1,p}(Q)$ is continuously embedded in $L^{q}$ yielding $u_n\to 0$ in $L^{q}(Q)$; up to a subsequence, one can assume that $u_n\to 0$ a.e. in $Q$. But \eqref{growthMultiwell} provides a constant $C>0$ such that $W(z)\leq C|z|^{q}$ for every $z\in\R^d$ with $|z|\ge 1$. Then $\|W(u_n)\mathbf{1}_{|u_n|>1}\|_{L^1}\leq C \|u_n\|_{L^q}^q\to 0$ as $n\to \infty$, and it is clear that $W(u_n)\mathbf{1}_{|u_n|\leq 1}\to 0$ in $L^1(Q)$ by the dominated convergence theorem. This concludes $W(u_n)\to 0$ in $L^1(Q)$.
\end{proof} 

\begin{proof}[Proof of Proposition~\ref{dWdist}]
Since one can always glue together two $\dot{H}_{div}^1$ maps when their traces coincide (with the new map still belonging to $\dot{H}_{div}^1$), the following immediately holds for each $v_n^\pm\in H^1_a(\T^{d-1},\R^d)$:
\begin{align*}
J(z^-,v_n^-)+d_W(v_n^-,v_n^+)+J(z^+,v_n^+)&\geq d_W(z^-,z^+),\\
J(z^-,v_n^-)+d_W(z^-,z^+)+J(z^+,v_n^+)&\geq d_W(v_n^-,v_n^+).
\end{align*}
If $(v_n^\pm)_{n\ge 1}\to z^\pm$ in $H^1(\T^{d-1})$, then Lemma~\ref{Main_est} implies that $(d_W(v_n^-,v_n^+))_{n\ge 1}$ has a limit given by $d_W(z^-,z^+)$.
\end{proof}

\subsection{Regularity of minimizers.} 
Under some regularity assumption on $W$, any global minimizer $u$ of the problem $(\mathcal P)$ solves \eqref{stokes} so that classical regularity results for the Stokes equation apply:
\begin{proposition}
\label{pro:reg}
Let $W\in \mathcal{C}^1(\R^d,\R)$ and $u\in\dot{H}_{div}^1(\Omega,\R^d)$ be a {solution of $(\mathcal{P})$}. Assume in addition that either 
$ u\in L_{loc}^\infty(\Omega)$ or $W$ is globally Lipschitz on $\R^d$. Then there exists a pressure $p$ such that \eqref{stokes} holds true; moreover, $p\in W_{loc}^{1,q}(\Omega,\R)$ and $u\in W_{loc}^{2,q}(\Omega,\R^d)$ for every $q\in (1,+\infty)$.
\end{proposition}
\begin{proof}
By minimality, one has $E(u)\leq E(u+\varepsilon v)$ for every $\varepsilon>0$ and every smooth test map $v$ compactly supported in $\Omega$ with $\nabla\cdot v=0$. By the Taylor-Lagrange formula applied to $W\in \mathcal{C}^1$, for all $x\in\Omega$, there exists $t(x)\in [0,1]$ such that 
$
W(u(x)+\varepsilon v(x))=W(u(x))+\varepsilon\nabla W(u(x)+\varepsilon\, t(x) v(x))\cdot v(x)$.
Thus, one has the inequality
\[
0\leq \frac{E(u+\varepsilon v)-E(u)}{\varepsilon}=\int_\Omega\nabla u\cdot\nabla v\diff x+\int_\Omega \nabla W(u(x)+\varepsilon t(x)v(x))\cdot v(x)\diff x+\frac{\varepsilon}{2}\int_\Omega |\nabla v|^2\diff x.
\]
In the limit when $\varepsilon\to 0$, by the dominated convergence theorem (note that, within our assumptions, $\nabla W(u+\varepsilon tv)$ is locally bounded in $\Omega$ as $\varepsilon\to 0$), one gets
\[
\int_\Omega \nabla u\cdot \nabla v+\nabla W(u)\cdot v\diff x\geq 0.
\]
Replacing $v$ by $-v$, it follows that the above LHS vanishes. In other words, the distribution $-\Delta u+\nabla W(u)$ vanishes when tested against smooth compactly supported divergence-free maps $v$, which means that there exists a distribution $p\in \mathcal{D}'(\Omega)$ such that 
\[
-\Delta u+\nabla W(u)=\nabla p \quad \textrm{ in } \quad \mathcal{D}'(\Omega).
\]
Since either $u\in L_{loc}^\infty(\Omega)$ or $W$ is globally Lipschitz on $\R^d$, then $ \nabla W(u)\in L_{loc}^\infty(\Omega)$ which implies the claimed regularity results for $u$ and $p$ thanks to standard regularity for the Stokes equation (see e.g. \cite[Theorem IV.2.1]{Galdi}).
\end{proof}
\begin{remark}
\label{rem_reg}
Without assuming $u\in L_{loc}^\infty(\Omega)$ or $W$ being globally Lipschitz on $\R^d$, one can still show that a global minimizer $u\in\dot{H}_{div}^1(\Omega,\R^d)$ of $(\mathcal P)$ solves the Stokes system \eqref{stokes} within some (weaker) growth conditions on $\nabla W$. Indeed, as $u$ belongs to $\dot{H}^1(\Omega,\R^d)$, one has $u\in L^{2^\ast}_{loc}(\Omega,\R^d)$ with $2^\ast:=\frac{2d}{d-2}$ if $d\geq 3$, and, by the Moser-Trudinger estimate, $ e^{\alpha |u(x)|^2}\in L^1_{loc}(\Omega)$ for all $\alpha>0$ if $d=2$. Thus, it is enough to assume that
\[
|\nabla W(z)|\leq 
\begin{cases}
Ce^{\alpha |z|^2}&\text{for some $C>0$ and $\alpha>0$ if $d=2$,}\\
C(1+|z|^r)&\text{for some $C>0$ and $r< 2^\ast$ if $d\geq 3$.}
\end{cases}
\]
If $d=2$, one has $ \nabla W(u)\in L_{loc}^q(\Omega,\R^d)$ for all $q\in (1,+\infty)$, and we get the same regularity result: $p\in W^{1,q}_{ loc}(\Omega,\R)$ and $u\in W^{2,q}_{ loc}(\Omega,\R^d)$ for every $q\in (1,+\infty)$. If $d\geq 3$ and $r<2^\ast$, one gets $p\in W^{1,q}_{ loc}(\Omega,\R)$ and $u\in W^{2,q}_{ loc}(\Omega,\R^d)$ for every $q\in (1, \frac{2^\ast}{r}]$.
\end{remark}

\section{One dimensional symmetry of global minimizers\label{one_dim_sym}}
\subsection{Analysis of the one-dimensional profile\label{an_1Dprofile}}

In this section, we study existence, uniqueness and properties of one-dimensional minimizers in \eqref{min1Dintro} which are essential for our aim of analyzing the one-dimensional symmetry in Question 1. In particular, the sufficient conditions we will find for $W$ in order to prove existence in \eqref{min1Dintro} are more general than the ones presented in Theorems \ref{thm1} and \ref{thm2} for the $d$-dimensional problem $(\mathcal P)$. The problem  \eqref{min1Dintro} shares the same difficulties (translation invariance, multiple zeros of the potential, etc\dots), but the proofs will be easier to establish since \eqref{min1Dintro} is a minimization problem in $1D$. 

\paragraph{Minimal energy of one-dimensional transitions.} 
We will focus here on maps $u:\Omega\to \R^d$, only depending on the first variable $x_1$, such that $\nabla\cdot u=0$ and $\overline{u}(\pm\infty)=u^\pm\in S_a$, $a\in\R$. Since $\nabla\cdot u=\partial_{1} u_1=0$, one has $u_1\equiv a$. Thus, $u$ writes $u(x)=\bar{u}(x_1)=\gamma(x_1)\in\R^d_a$ for a.e. $x_1\in\R$. Our aim is to 
analyze
solutions $\gamma$ of the 1D minimization problem
\begin{equation}\label{energy_1D}
\gamma\in\mathrm{Argmin}\left\{E(\gamma)\;:\;\gamma\in\dot{H}^1(\R,\R^d_a),\,\gamma(\pm\infty)=u^\pm\right\},
\end{equation}
where the one-dimensional energy $E$ is defined for all $\gamma\in H_{loc}^1(\R,\R^d_a)$ by
$$
E(\gamma)=\int_\R \frac 12|\dot{\gamma}(t)|^2+W(\gamma(t))\diff t.
$$
By Young's inequality, one has 
\begin{equation}\label{young_e1D}
E(\gamma)\geq L_W(\gamma):= \int_\R\sqrt{2W(\gamma(t))}\, |\dot{\gamma}(t)|\diff t . 
\end{equation}
The RHS integral in \eqref{young_e1D} is invariant by monotone reparametrization and represents the length of the curve $\gamma$ in $\R^d_a$ endowed with the singular Riemannian metric $g_W(z)=2W(z)g_0$, where $g_0$ is the usual Euclidean metric on $\R^d_a$. 

If $W\in W^{1,\infty}_{loc}(\R^d)$, any solution $\gamma$ of \eqref{energy_1D} satisfies the Euler-Lagrange equation,
\begin{equation*}
\ddot{\gamma}(t)=\nabla W(\gamma(t)).
\end{equation*}
In particular, $\ddot{\gamma}\in L^\infty_{loc}$ (i.e., $\dot{\gamma}\in W^{1,\infty}_{loc}$). Moreover, multiplying the above equation by $\dot{\gamma}$ and integrating provides the equipartition of the $1D$ energy density:
$$\frac 12|\dot{\gamma}(t)|^2=W(\gamma (t))\quad\text{in }\R.$$
Note that if a curve $\gamma:\R\to\R^d_a$ satisfies the equipartition identity, then the inequality in \eqref{young_e1D} becomes equality. Therefore, every global minimizer $\gamma$ of $E$ connecting $u^-$ to $u^+$ is expected to lie on an energy-minimizing geodesic between $u^-$ and $u^+$ in $(\R^d_a,g_W)$ (see Proposition \ref{inf_1D}). The metric $g_W$ is singular on $S_a=\{W=0\}$ but, at least if $S_a$ is discrete, it induces the distance defined in \eqref{estgeod} for $W$ instead of $V$, i.e., for all $u^\pm\in\R^d_a$,
\begin{equation}\label{geod_H1}
\mathrm{geod}_{W}^a(u^-,u^+)=\inf \left\{\int_{-1}^1 \sqrt{2W(\gamma(s))}\, |\dot{\gamma}(s)|\diff s \;:\; \gamma\in{\rm Lip}([-1,1],\R^d_a),\, \gamma(\pm 1)=u^\pm\right\}.
\end{equation}
First, we prove that the infimum of $E$ in \eqref{min1Dintro} always coincides with $\mathrm{geod}_{W}^a(u^-,u^+)$ under very weak assumptions on $W$. This result is quite standard, but we prove it for completeness: 
\begin{proposition}\label{inf_1D}
For every continuous potential $W:\R^d\to \R_+$ and any two wells $u^\pm\in S_a$ for some $a\in\R$, one has
\begin{equation}
\label{identity_inf1D}\mathrm{geod}_{W}^a(u^-,u^+)=\inf \left\{E(\gamma)\;:\;\gamma\in\dot{H}^1(\R,\R^d_a),\,\gamma(\pm\infty)=u^\pm\right\}.
\end{equation}
In particular, if there exists a minimizer $\gamma$ in \eqref{energy_1D}, then $E(\gamma)=L_W(\gamma)=\mathrm{geod}_{W}^a(u^-,u^+)$ and we have equipartition of the energy density, i.e. $|\dot{\gamma}(t)|=\sqrt{2W(\gamma (t))}$ for a.e. $t\in\R$.
\end{proposition}
\begin{proof}
One can assume that $u^-\neq u^+$. The inequality $\leq$ in \eqref{identity_inf1D} follows from \eqref{young_e1D} and \eqref{geod_H1} (within the reparametrization and smoothing argument presented at Step 1 of the proof of 
Lemma~\ref{cWgeqGeod}). For the reverse inequality, we divide the proof according to whether $W$ has only two wells or not in $\R^d_a$.

\medskip
\noindent\textsc{Step 1: double-well potentials in $\R^d_a$.} Assume that $W(z)>0$ for all $z\in\R^d_a\setminus \{u^\pm\}$. We need to prove that for every curve $\gamma\in\mathrm{Lip}([-1,1],\R^d_a)$ such that $\gamma(\pm 1)=u^\pm$, one has
\[
L_W(\gamma)\geq\inf\left\{E(\lambda)\;:\;\lambda\in\dot{H}^1(\R,\R^d_a),\,\lambda(\pm\infty)=u^\pm\right\}.
\]
Up to reparametrizing $\gamma$, one can assume that it is injective and has constant speed so that $W(\gamma(\cdot))>0$ and $|\dot{\gamma}|\equiv v>0$ in $(-1,1)$. We then reparametrize $\gamma$ by equipartition: we set $\lambda:=\gamma\circ\sigma$ in such a way that $|\dot{\lambda}|=\sqrt{2W(\lambda)}$, i.e. we need $\sigma:\R\to [-1,1]$ to solve
\[
\dot{\sigma}(t)=\frac{\sqrt{2W(\gamma(\sigma(t)))}}v \quad\text{for all }t\in\R.
\]
Indeed, by the Peano-Arzela Theorem, there exists such a (maximal) solution $\sigma\in\mathcal{C}^1(\R,[-1,1])$ which is nondecreasing and converges to $\pm 1$ at $\pm\infty$. The claimed inequality then follows since
\[
L_W(\gamma)=L_W(\lambda)=E(\lambda)\quad\text{and}\quad \lambda(\pm\infty)=u^\pm.
\]
\textsc{Step 2: multi-well potentials in $\R^d_a$. } Take a continuous function $\xi :\R^d\to\R_+$ such that $\xi=0$ on $\{u^-,u^+\}$ and $\xi >0$ on $\R^d\setminus\{u^-,u^+\}$, and set $W_\varepsilon:= W+\varepsilon^2\xi$ for each $\varepsilon>0$. By Step 1, one has
\[
\mathrm{geod}_{W_\varepsilon}^a(u^-,u^+)\geq\inf \left\{\int_\R \frac 12|\dot{\gamma}(t)|^2+W_\varepsilon(\gamma(t))\diff t \;:\;\gamma\in\dot{H}^1(\R,\R^d_a),\,\gamma(\pm\infty)=u^\pm\right\}.
\]
Since $W_\varepsilon\geq W$, it is enough to prove that 
\[
 \mathrm{geod}_{W}^a(u^-,u^+)\ge \limsup_{\varepsilon\to 0}\mathrm{geod}_{W_\varepsilon}^a(u^-,u^+).
\]
To prove this last fact, observe that for every curve $\gamma\in\mathrm{Lip}([-1,1],\R^d_a)$ such that $\gamma(\pm 1)=u^\pm$, we have 
\(
\mathrm{geod}_{W_\varepsilon}^a(u^-,u^+)\leq L_{W_\varepsilon}(\gamma)\leq L_W(\gamma)+\varepsilon L_{\xi}(\gamma)
\)
by subadditivity of $t\mapsto\sqrt{t}$.
The desired inequality follows by taking the $\limsup$ as $\varepsilon\to 0$ and then, the infimum over $\gamma$.

The last statement follows from \eqref{young_e1D} where the equality holds for a minimizer $\gamma$ of \eqref{energy_1D}. In particular, 
$|\dot{\gamma}(t)|=\sqrt{2W(\gamma (t))}$.
\end{proof}

As consequence, we deduce that any global minimizer of the $d$-dimensional problem $(\mathcal P)$ having the image confined in the hyperplane $\R^d_a$ is necessarily one-dimensional:
\begin{corollary}\label{1Dcriterium}
Let $W:\R^d\to\R_+$ be a continuous potential and $a\in\R$ such that $S_a$ contains at least two wells $u^\pm$ of $W$. If $u\in\dot{H}^1_{div}(\Omega,\R^d)$ is a global minimizer of $(\mathcal P)$ with $u(x)\in\R^d_a$ for a.e.\ $x\in\Omega$, then $u$ is one-dimensional, i.e. there exists $g\in\dot{H}^1(\R,\R^d_a)$ with $u(x)=g(x_1)$ for a.e.\ $x=(x_1,x')\in\Omega$.
\end{corollary}
\begin{proof}
By Lemma~\ref{BC_unif}, there exist two sequences $(R_n^\pm)_{n\ge 1}$ such that $R_n^\pm\to\pm\infty$ and the Sobolev trace $u(R_n^\pm,\cdot)\to u^\pm$ a.e.\ in $\T^{d-1}$ as $n\to\infty$; by Fubini's theorem, it follows
\begin{align*}
E(u)\geq\int_{\T^{d-1}}\left\{\int_{R_n^-}^{R_n^+} \frac 12|\partial_{x_1}u(x_1,x')|^2+W(u(x_1,x'))\diff x_1\right\}\diff x' +\frac 12\int_\Omega |\nabla' u|^2\diff x\\
\geq\int_{\T^{d-1}}\left\{\int_{R_n^-}^{R_n^+} |\partial_{x_1}u(x_1,x')|\sqrt{2W(u(x_1,x'))}\diff x_1\right\}\diff x' +\frac 12\int_\Omega |\nabla' u|^2\diff x,
\end{align*}
where $\nabla'=(\partial_2,\dots,\partial_d)$. Since $u(R_n^\pm,\cdot)\in\R^d_a$ a.e.\ by assumption, we have by definition \eqref{geod_H1} of $\mathrm{geod}_{W}^a$ and by the reparametrization and smoothing argument presented at Step 1 of the proof of Lemma \ref{cWgeqGeod},
\begin{align*}
E(u)\geq\int_{\T^{d-1}}\mathrm{geod}_{W}^a(u(R^-_n,x'),u(R_n^+,x'))\diff x' +\frac 12\int_\Omega |\nabla' u|^2\diff x.
\end{align*}
By Fatou's Lemma, and by continuity of $\mathrm{geod}_{W}^a(\cdot,\cdot)$ (see Step 4 in the proof of Lemma \ref{cWgeqGeod}), one gets in the limit $n\to\infty$,
\begin{align*}
E(u)\geq\mathrm{geod}_{W}^a(u^-,u^+)+\frac 12\int_\Omega |\nabla' u|^2\diff x.
\end{align*}
By Proposition~\ref{inf_1D}, $\mathrm{geod}_{W}^a(u^-,u^+)$ is the infimum of $E$ restricted to $1D$ transitions connecting $u^-$ to $u^+$, thus $\mathrm{geod}_{W}^a(u^-,u^+)\geq c_W(u^-, u^+)=E(u)$. Combined with the above inequality, the minimality of $u$ yields $\nabla' u=0$ a.e., that is, $u$ only depends on $x_1$. 
\end{proof}

\paragraph{Existence of $1D$ minimizers in dimension 2.} When $d=2$, the situation is very simple since $\R^2_a$ is of dimension $1$. Set $u^\pm=(a,u_2^\pm)$ and $\gamma (t)=(a,\varphi(t))$ for $t\in\R$, where $\varphi:\R\to\R$ satisfies $\varphi(\pm\infty)=u_2^\pm$. It is clear that the infimum \eqref{geod_H1} can be restricted to those curves $\gamma$ such that $\varphi$ is monotone. Then the change of variables $y=\varphi(t)$ yields
\begin{equation}
\label{min_en_1D}
\mathrm{geod}_{W}^a(u^-,u^+)=\int_{[u_2^-,u_2^+]}\sqrt{2W(a,y)}\diff y,
\end{equation}
where $[u_2^-,u_2^+]=\{tu_2^-+(1-t)u_2^+\;:\; 0\leq t\leq 1\}$ (which makes sense even if $u_2^->u_2^+$). Recall that solutions of \eqref{energy_1D} are characterized by the equipartition identity, which is equivalent to
\(
|\varphi'|=\sqrt{2W(a,\varphi)}.
\)
The existence and uniqueness of the one-dimensional profile is given in the following:
\begin{proposition}\label{Exist_profile1D}
If $W:\R^2\to\R_+$ is a continuous function and $u^\pm=(a,u_2^\pm)\in\R^2_a$ are two distinct zeros of $W$ for some $a\in\R$, then one has equivalence of the two following assertions:
\begin{itemize}
\item the minimization problem \eqref{energy_1D} has a solution,
\item the function $y\mapsto \frac{1}{\sqrt{W(a,y)}}$ belongs to $L^1_{loc}((u_2^-,u_2^+))$, with the convention $\frac{1}{\sqrt{0}}=+\infty$.
\end{itemize}
In case of existence, minimizers $\gamma=(a,\varphi):\R\to\R^d_a$ are characterized by
\begin{equation}\label{ode_1Dprof}
\varphi\in\mathcal{C}^1(\R,\R),\quad \varphi(\pm\infty)=u_2^\pm
\quad\text{and}\quad\varphi'(t)=
\begin{cases}
\sqrt{2W(a,\varphi(t))}&\text{if }u_2^-\leq u_2^+,\\
-\sqrt{2W(a,\varphi(t))}&\text{if }u_2^-> u_2^+,
\end{cases}
\quad
 t\in\R.
\end{equation}
Moreover, there exists at most one minimizer (up to translation) $\gamma=(a,\varphi)$ of \eqref{energy_1D} such that $\varphi$ is strictly monotone. In particular, the problem \eqref{energy_1D} has a unique solution if $W(a,y)>0$ for $y\in (u_2^-,u_2^+)$.
\end{proposition}
\begin{proof}
W.l.o.g., one can assume that $u_2^-<u_2^+$. First, we prove that \eqref{ode_1Dprof} holds true if and only if $\varphi$ solves \eqref{energy_1D}. If $\gamma=(a,\varphi)$ solves \eqref{energy_1D} with $\varphi\in\dot{H}^1(\R,\R)$, then, by the minimality of $\gamma$, $\varphi$ is monotone and, since we have assumed $u_2^-<u_2^+$, it is actually nondecreasing on $\R$. Moreover, by Proposition \ref{inf_1D}, one has equality in \eqref{young_e1D}, therefore $|\dot{\gamma} |=\sqrt{2W(\gamma)}$, i.e. $\varphi$ solves \eqref{ode_1Dprof}. 
Conversely, if $\varphi$ solves \eqref{ode_1Dprof}, then $\gamma$ solves \eqref{energy_1D} since by \eqref{min_en_1D}, we have
\[
E(\gamma)=\int_\R \frac 12|\varphi'|^2+W(a,\varphi)\diff t=\int_\R \sqrt{2W(a,\varphi)}|\varphi'|\diff t=\int_{u_2^-}^{u_2^+}\sqrt{2W(a,y)}\diff y=\mathrm{geod}_{W}^a(u^-,u^+).
\]
Second, we prove the equivalence of the two assertions for existence in \eqref{energy_1D}. On the one hand, if $F(\cdot):=(2W(a,\cdot))^{-1/2}\in L^1_{loc}((u_2^-,u_2^+))$, a solution $\varphi:\R\to\R$ to \eqref{ode_1Dprof} is given by $\varphi(t)=u^-$ if $t\le G(u_2^-)$, $\varphi(t)=G^{-1}(t)$ if $t\in (G(u_2^-),G(u_2^+))$ and $\varphi(t)=u^+$ if $t\ge G(u_2^+)$, where $G$ is an antiderivative of $F$, i.e. $G\in W^{1,1}_{loc}((u_2^-,u_2^+))$ with $G'=F>0$ a.e. in $(u_2^-,u_2^+)$ (in particular, $G$ has an inverse $G^{-1}$ on its range $G((u_2^-,u_2^+))$, so $\varphi$ is well defined).
Thus, by the above argument, $(a,\varphi)$ is a minimizer in \eqref{energy_1D}.
On the other hand, if \eqref{energy_1D} has a minimizer, i.e., \eqref{ode_1Dprof} has a solution $\varphi$, then for all $\eta<\frac 12 (u_2^+-u_2^-)$, there exist two real numbers $t^\pm\in\R$ such that $t^-\leq t^+$, $\varphi(t^-)=u_2^-+\eta$ and $\varphi(t^+)=u_2^+-\eta$. Then, by the change of variables $y=\varphi(t)$,
\[
\int_{u_2^-+\eta}^{u_2^+-\eta}\frac{\diff y}{\sqrt{2W(a,y)}}=\int_{t^-}^{t^+}\frac{\varphi'(t)\diff t}{\sqrt{2W(a,\varphi(t))}}=t^+-t^-<+\infty,
\]
which implies local integrability of $(2W(a,\cdot))^{-1/2}$ on $(u_2^-,u_2^+)$.
\end{proof}

\paragraph{Existence of $1D$ minimizers in dimension $d\geq 3$.} In higher dimension $d\geq 3$, the problem of the existence of a one-dimensional minimizer in \eqref{energy_1D} is more delicate. Proposition~\ref{1D_transition} below gives sufficient conditions on $W$ for solving the existence problem  in \eqref{energy_1D}. These sufficient conditions are more general than the ones in Theorems \ref{thm1} and \ref{thm2}. \footnote{A generalization of Proposition~\ref{1D_transition} can be found in \cite[Theorem 3]{Monteil:2016}, \cite[Theorem 2.5]{Zuniga:2016} and \cite[Theorem 1]{Sourdis:2016}. The proof of Proposition \ref{1D_transition} is new, therefore we present it here.} 

\begin{proposition}\label{1D_transition}
Let $W:\R^d\to\R_+$ be a continuous potential and $a\in\R$ such that:
\begin{enumerate}
\item $S_a=\{z\in\R^d_a\;:\; W(z)=0\}$ is finite and contains at least two wells $u^\pm$;
\item $\liminf_{|z'|\to \infty} W(a,z')>0$;
\item $\mathrm{geod}_{W}^a(u^-,u^+)<\mathrm{geod}_{W}^a(u^-,z)+\mathrm{geod}_{W}^a(z,u^+)$ for all $z\in S_a\setminus\{u^\pm\}$. 
\end{enumerate}
Then the one-dimensional minimization problem \eqref{energy_1D} has a solution.
\end{proposition}
\begin{proof}
The proof follows the same arguments as in the second proof of Theorem~\ref{thm2} based on Lemma \ref{doring} (see \cite[Lemma 1.]{Doring:2013}). 
Take a minimizing sequence $(\gamma_n)_{n\ge 1}\subset \dot{H}^1(\R,\R^d_a)$ such that $\gamma_n(\pm\infty)=u^\pm$ and $E(\gamma_n)\to 
\mathrm{geod}_{W}^a(u^-,u^+)$ as $n\to \infty$. Then there exist $\varepsilon>0$ and $N>0$ such that for $n\geq N$ we have $\gamma_n(\R)\cap {B}(z, \eps)=\emptyset$ for every $z\in S_a\setminus\{u^\pm\}$.
This is because otherwise, there would be sequences $(n_k)_{k\ge 1}\subset\N$ and $(s_{n_k})_{k\ge 1}\subset \R$, and $z\in S_a\setminus\{u^\pm\}$ such that $n_k\to\infty$ and $\gamma_{n_k}(s_{n_k})\to z$ as $k\to \infty$; since
$$
E(\gamma_{n_k})\stackrel{\eqref{young_e1D}}\geq L_W(\gamma_{n_k})\geq  \mathrm{geod}_{W}^a(u^-,\gamma_{n_k}(s_{n_k}))+ \mathrm{geod}_{W}^a(\gamma_{n_k}(s_{n_k}), u^+),
$$ 
the continuity of $\mathrm{geod}_{W}^a$ (see Step 4 in the proof of Lemma \ref{cWgeqGeod}) would yield in the limit $k\to \infty$: 
$\mathrm{geod}_{W}^a(u^-,u^+)\geq 
\mathrm{geod}_{W}^a(u^-,z)+\mathrm{geod}_{W}^a(z,u^+)$ which contradicts the assumption 3. Now, for $n\geq N$, we define
\[
v_n(t)=\left({\gamma_n}(t)-\frac 12(u^++u^-)\right)\cdot (u^+-u^-)\quad \textrm{for all } t\in \R.
\]
Since $(\dot{v}_n)_{n\ge 1}$ is bounded in $L^2(\R)$ (because $(\dot{\gamma}_n)_{n\ge 1}$ does it) and $v_n(\pm \infty)=\pm \frac12 |u^+-u^-|^2$, by 
Lemma~\ref{doring}, there exist $v\in\dot{H}^1(\R)$ and a sequence $(t_n)_{n\ge 1}\subset \R$ such that, up to a subsequence, one has $v_n(\cdot+t_n)\to v$ weakly in $\dot{H}^1(\R)$ and locally uniformly in $\R$, and $\liminf_{t\to \pm \infty} \pm v(t)\geq 0$. Also, up to a subsequence, $(\gamma_n(\cdot+t_n))_{n\ge 1}$ converges weakly in $\dot{H}^1$ and locally uniformly in $\R$ to a curve $\gamma\in \dot{H}^1(\R,\R^d_a)$; in particular, $E(\gamma)\leq \mathrm{geod}_{W}^a(u^-,u^+)$ (by the lower semicontinuous of $E$ in weak $\dot{H}^1$-topology), $\gamma(\R)\cap {B}(z, \eps)=\emptyset$ for every $z\in S_a\setminus\{u^\pm\}$ and $v(t)=\big({\gamma}(t)-\frac 12(u^++u^-)\big)\cdot (u^+-u^-)$ for all $t\in \R$. By Lemma \ref{closed_boundary}, we know that $\gamma(\pm\infty)=z^\pm$ with $z^\pm\in S_a$. Since $\gamma$ stays away from $S_a\setminus\{u^\pm\}$, then 
$\{z^\pm\}=\{u^\pm\}$; finally, the sign of $v$ at $\pm \infty$ implies that $z^\pm=u^\pm$, leading to $E(\gamma)= \mathrm{geod}_{W}^a(u^-,u^+)$.
\end{proof}

\subsection{A density result. Proof of Lemma \ref{lavrentiev}}\label{sec:lav}
The aim of this section is to prove Lemma \ref{lavrentiev}, i.e., the set $\dot{H}^1_{div}\cap\mathcal{C}^{\infty}\cap L^\infty (\Omega,\R^d)$ is dense in the admissible set $\{ u\in \dot{H}_{div}^1(\Omega,\R^d)\; :\; E(u)<\infty\}$ within the topology induced by the energy $E$. A situation where this property fails was pointed out by  Lavrentiev \cite{Lavrentiev:1926}: he gave an example of an energy functional whose infimum over smooth functions is strictly greater than the infimum over all finite energy admissible configurations. This phenomenon is now usually called ``Lavrentiev gap'' in the literature. The role of Lemma \ref{lavrentiev} is to give a sufficient condition on $W$ such that the ``Lavrentiev gap'' is avoided for the energy $E$.

\begin{proof}[Proof of Lemma~\ref{lavrentiev} when $W$ satisfies \eqref{growth_W}] We will define a sequence $(u_k)_{k\ge 1}\subset L^\infty\cap\mathcal{C}^\infty$ converging strongly to $u$ in $\dot{H}^1_{div}(\Omega,\R^d)$ s.t. for each $k\ge 1$, $u_k(x)=u^\pm$ in a neighborhood of~$\pm\infty$.

\medskip
\noindent\textsc{Step 1. Cutting $u$ by $u^\pm$ at $\pm\infty$.} By Lemma~\ref{BC_unif}, there exist two sequences $(R_n^+)_{n\ge 1}$ and $(R_n^-)_{n\ge 1}$ such that $R_n^\pm\to\pm\infty$ and $u(R_n^\pm,\cdot)\to u^\pm$ strongly in $H^1(\T^{d-1},\R^d)$ as $n\to +\infty$. Then, since the growth condition \eqref{growth_W} is more restrictive than \eqref{growthMultiwell}, Lemma~\ref{Main_est} allows to construct for each $n\ge 1$ a new map $v_n\in\dot{H}^1(\Omega,\R^d)$ such that $v_n(x)=u(x)$ if $x\in [R_n^-,R_n^+]\times\T^{d-1}$, $v_n(x)=u^+$ if $x_1\ge R_n^++1$, $v_n(x)=u^-$ if $x_1\le R_n^--1$, and $E(v_n)\to E(u)$ as $n\to\infty$. In particular, $(v_n)_{n\ge 1}\to u$ in $\dot{H}^1(\Omega,\R^d)$.

\medskip
\noindent\textsc{Step 2. Smoothing $u$ by convolution.} By Step 1, one can assume that there exists $R>0$ with $u=u^\pm$ for $\pm x_1\ge R$. Let us take a smooth mollifying kernel $\rho\in\mathcal{C}^\infty (\R)$ such that $\rho\geq 0$, $\int_\R \rho =1$, and $\rho$ is even on $\R$ and compactly supported in $(-1/2,1/2)$. For each $k\geq 1$, we set $\rho_k (t)=k\rho (k t)$ for every $t\in\R$ and
$$\rho_k^{\otimes d}(y):=\prod_{i=1}^d \rho_k(y_i),\quad \text{for every }y=(y_1,\dots,y_d)\in\R^d.$$
Since $(-1/2, 1/2)^d$ isometrically embeds into $\Omega=\R\times\T^{d-1}$ via the quotient map, $\rho_k^{\otimes d}$ induces a smooth kernel on $\Omega$ (corresponding to the periodized kernel in $\T^{d-1}$ for every $x_1\in \R$) that is still denoted by $\rho_k^{\otimes d}$. This allows to define the regularization by convolution of $u$ by 
\[
u_k(x):=\rho^{\otimes d}_k\ast u(x)=\int_{\Omega} u(x-y)\rho^{\otimes d}_k(y)\diff y\quad\text{for every $x\in \Omega$.}
\] 
Then $(u_k)_{k\ge 1}\to u$ in $\dot{H}^1(\Omega,\R^d)$, i.e. $u_k\to u$ in $L^2_{loc}(\Omega,\R^d)$ and $\nabla u_k=\rho_k^{\otimes d}\ast \nabla u\to \nabla u$ in $L^2(\Omega)$ as $k\to\infty$. Moreover, $u_k\in\mathcal{C}^\infty( \Omega,\R^d)$ and, since $u=u^\pm$ for $\pm x_1\ge R>0$ and $\rho_k(x_1)$ is supported in $x_1\in (-1/2,1/2)$, one has $u_k=u^\pm$ for $\pm x_1\ge R+1$; in particular $u\in L^\infty(\Omega,\R^d)$. Concerning the divergence constraint, we observe that $\nabla\cdot u_k=\rho_k^{\otimes d}\ast (\nabla\cdot u)=0$.

\medskip
\noindent\textsc{Step 3. Convergence of the energy densities $e_{den}(u_k)$ in $L^1(\Omega)$.} Since $(\nabla u_k)_{k\ge 1}\to \nabla u$ in $L^2(\Omega)$, it is enough to prove convergence of $(W(u_k))_{k\ge 1}$ in $L^1(\Omega)$. Note that, since $u_k=u^\pm$ out of the set
$$\omega_{R+1}:=(-R-1,R+1)\times\T^{d-1}$$ and $W(u^\pm)=0$, $W(u_k)$ is compactly supported in $\omega_{R+1}$. Thus, by Vitali's convergence theorem, it is enough to prove that $(W(u_k))_{k\ge 1}$ is uniformly integrable in $\omega_{R+1}$. We use \eqref{growth_W} which means that there exist $\alpha,C>0$ such that $W(z)\leq F(z)$, where
\begin{equation*}\label{F}
F(z):=\begin{cases}
C\, e^{\alpha |z|^2}&\text{if }d=2,\\
C\, [1+|z|^{2^\ast}]&\text{if }d\geq 3.
\end{cases}
\end{equation*}
By the Gagliardo-Nirenberg-Sobolev inequality (if $d\geq 3$) and by the Moser-Trudinger inequality (if $d=2$) respectively, one has $F(u)\in L^1(\omega_{R+ 2})$. We prove that $(F(u_k))_{k\ge 1}$, and so $(W(u_k))_{k\ge 1}$, are uniformly integrable in $\omega_{R+1}$, the main ingredient being the convexity of $F$. Indeed, for all measurable set $A\subset \omega_{R+1}$, by Jensen's inequality and Fubini's theorem, we have
\begin{align*}
\int_{A}W(u_k)\diff x\leq \int_A F(u_k)&\diff x= \int_{A}F\left(\int_{\Omega}u(x-y) \rho^{\otimes d}_k(y)\diff y\right)\diff x\\
&\leq \int_A\int_{\Omega}F(u(x-y))\rho^{\otimes d}_k(y)\diff y\diff x=\int_{\Omega}\rho^{\otimes d}_k(y)\left\{\int_{A-y}F(u(z))\diff z\right\}\diff y;
\end{align*}
the last integral goes to $0$ when the Lebesgue measure of $A$ tends to $0$, uniformly in $k$.
\end{proof}

\begin{proof}[Proof of Lemma~\ref{lavrentiev} when $W\in\mathcal{C}^2(\R^d,\R_+)$ and $u\in L^\infty(\Omega,\R^d)$] We shall define a sequence $(u_k)_{k\ge 1}\subset L^\infty\cap\mathcal{C}^\infty$ converging to $u$ in $\dot{H}^1_{div}(\Omega,\R^d)$ with $\overline{u}_k(\pm\infty)=u^\pm$ for each $k\ge 1$.

\medskip
\noindent\textsc{Step 1. Smoothing $u$ by convolution.} We follow the strategy in Step 2 in the preceding proof and we obtain for each $k\ge 1$ a map $u_k=\rho_k^{\otimes d}\ast u\in \mathcal{C}^\infty\cap\dot{H}^1_{div}(\Omega,\R^d)$ such that $(u_k)_{k\ge 1}\to u$ in $\dot{H}^1(\Omega,\R^d)$ and a.e. It is also clear that $(u_k)_{k\ge 1}$ is bounded in $L^\infty(\Omega,\R^d)$ since for each $k\ge 1$, $\|u_k\|_{L^\infty}\le \|u\|_{L^\infty}<+\infty$. We now check that $\overline{u_k}(\pm\infty)=u^\pm$. Indeed, for every $x_1\in\R$, we have by Fubini's theorem,
\begin{align*}
\overline{u_k}(x_1)=\int_{\T^{d-1}} u_k(x_1,x')\diff x' &= \int_\R \rho_k(y_1)\int_{\T^{d-1}}\rho^{\otimes (d-1)}_k(y')\int_{\T^{d-1}} u(x_1-y_1,x'-y')\diff x'\diff y'\diff y_1\\
&=\int_\R \rho_k(y_1)\bigg\{\int_{\T^{d-1}}\rho^{\otimes (d-1)}_k(y')\diff y'\bigg\}\bigg\{\int_{\T^{d-1}} u(x_1-y_1,z')\diff z'\bigg\}\diff y_1\\
&=\rho_k\ast\overline{u}(x_1).
\end{align*}
In particular, $\overline{u_k}=\rho_k\ast\overline{u}$ has the same limit as $\overline{u}$ at $\pm\infty$, that is $u^\pm$. 

\medskip
\noindent\textsc{Step 2. Convergence of $e_{den}(u_k)$ in $L^1(\Omega)$.} It is enough to prove convergence of $(W(u_k))_{k\ge 1}$ in $L^1(\Omega)$. By continuity of $W$, we know that $(W(u_k))_{k\ge 1}\to W(u)$ a.e.\ and by Fatou's lemma, we deduce 
$$
\int_\Omega W(u)\diff x\le \liminf_{k\to\infty}\int_\Omega W(u_k)\diff x.
$$
Therefore, it is enough to prove that
\be
\label{limsu}
\limsup_{k\to\infty}\int_\Omega W(u_k)\diff x\le\int_\Omega W(u)\diff x.
\ee
We shall use the following $\lambda$-convexity type inequality for the potential $W\in \mathcal{C}^2(\R^d)$: for every $z_1,z_2\in B:=\overline{B}(0,\|u\|_{L^\infty})$,
\[
W(z_1)\ge W(z_2)+\nabla W(z_2)\cdot (z_1-z_2)-\lambda |z_1-z_2|^2,
\]
where $\lambda=\frac 12\sup\{|\nabla^2W(z)|\;:\; z\in B\}$. Applying this inequality to $z_1=u(x)$ and $z_2=u_k(y)$ with $x\in\Omega$ and $y\in\omega_R:=(-R,R)\times\T^{d-1}$ for some fixed $R>1$, one gets
\begin{align}
\nonumber
\int_\Omega & W(u(x))\diff x =\int_\Omega\int_{\omega_R}W(u(x))\rho^{\otimes d}_k(x-y)\diff y\diff x\\
\label{numar120}
&\quad\ge\int_\Omega\int_{\omega_R}\left[ W(u_k(y))+\nabla W(u_k(y))\cdot\big(u(x)-u_k(y)\big)-\lambda |u(x)-u_k(y)|^2\right]\rho^{\otimes d}_k (x-y)\diff y\diff x,
\end{align}
where every integrand in \eqref{numar120} is integrable on $(x,y)\in \Omega\times\omega_R$ for vanishing
{on $\{|x_1|>R+1\}$}. We first claim that the second term (involving $\nabla W$) in \eqref{numar120} will disappear since, as $\rho^{\otimes d}_k (z)=\rho^{\otimes d}_k (-z)$ for all $z\in\R^d$ and by Fubini's theorem, one has 
\begin{align*}
\int_\Omega\int_{\omega_R}\nabla W(u_k(y))\cdot u(x)\, \rho^{\otimes d}_k (x-y)\diff y\diff x
&=\int_{\omega_R}\nabla W(u_k(y))\cdot u_k(y)\diff y\\
&=\int_{\Omega}\int_{\omega_R}\nabla W(u_k(y))\cdot u_k(y)\,\rho^{\otimes d}_k (x-y)\diff y\diff x.
\end{align*}
We then claim that the integral 
\[
I:=\int_\Omega\int_{\omega_R}|u(x)-u_k(y)|^2\rho^{\otimes d}_k(x-y)\diff y\diff x
\] 
vanishes in the limit $k\to\infty$. Indeed, by Jensen's inequality, we have
\begin{align*}
|u(x)-u_k(y)|^2&=\left|\int_{\Omega} (u(z)-u(x))\rho^{\otimes d}_k (y-z)\diff z\right|^2\\
&\le \int_{\Omega}  |u(z)-u(x)|^2\rho^{\otimes d}_k  (y-z)\diff z\\
&= \int_{\Omega}  \left|\int_0^1\nabla u(x+t(z-x)) (z-x)^T\diff t\right|^2\rho^{\otimes d}_k  (y-z)\diff z\\
&\le \int_{\Omega} \int_0^1|\nabla u(x+t(z-x))|^2|z-x|^2\rho^{\otimes d}_k  (y-z)\diff t\diff z.
\end{align*}
We have thus obtained the estimate
\[
I\le\int_\Omega\int_{\omega_R}\int_\Omega\int_0^1|\nabla u(x+t(z-x))|^2|z-x|^2\rho^{\otimes d}_k  (y-z)\rho^{\otimes d}_k(x-y)\diff t\diff z\diff y\diff x.
\] 
We now use the changes of variables $s=z-x$ and $w=y-x$, and the fact that $\rho^{\otimes d}_k (y-z)\rho^{\otimes d}_k(x-y)=0$ if $|x-y|=|w|\ge \frac{\sqrt{d}}{2k}$ or $|z-y|=|s-w|\ge \frac{\sqrt{d}}{2k}$ (here, \(|\cdot|\) is the norm induced on \(\Omega=\R\times\T^{d-1}\) by the euclidean norm on \(\R^d\) via the quotient map). We obtain
\begin{align*}
I&\le\int_\Omega\int_{B(0,\frac{\sqrt{d}}{2k})}\int_{B(w,\frac{\sqrt{d}}{2k})} \int_0^1|\nabla u(x+ts)|^2 |s|^2\rho^{\otimes d}_k(s-w)\rho^{\otimes d}_k(w)\diff t\diff s\diff w\diff x\\
& \leq \frac{d}{k^2}\int_{B(0,\frac{\sqrt{d}}{2k})}\int_{B(w,\frac{\sqrt{d}}{2k})} \int_0^1\rho^{\otimes d}_k(s-w)\rho^{\otimes d}_k(w)\left(\int_\Omega |\nabla u(x+ts)|^2\diff x\right)\diff t\diff s\diff w= \frac{d}{k^2}\|\nabla u\|^2_{L^2(\Omega)},
\end{align*}
where we have used the inequality $|s|\le |s-w|+|w|\le \frac{\sqrt{d}}{k}$. Finally, we have obtained
\[
\int_\Omega W(u(x))\diff x\ge\int_{\omega_R}W(u_k(y))\diff y-\frac{\lambda d}{k^2}\|\nabla u\|^2_{L^2(\Omega)}.
\]
Passing to the limit as $R\to\infty$ by the monotone convergence theorem, and then taking the $\limsup$ as $k\to\infty$, we obtain \eqref{limsu}. 
\end{proof}

\subsection{Entropy method\label{entropymethod}}

Our main tool in the study of the global minimizers of the energy $E$ under both the divergence constraint and the boundary condition $\overline{u}(\pm\infty)=u^\pm\in S_a$ is the entropy method that we develop here in any dimension $d\geq 2$. In dimension $d=2$, this method has reminiscence in the works of Aviles-Giga \cite{Aviles:1987,Aviles:1999}, Jin-Kohn \cite{Jin:2000} and has been formalized in Ignat-Merlet \cite{Ignat:2011} for obtaining lower bounds for the energy of Bloch walls and in DeSimone-Kohn-M\"uller-Otto \cite{DKMOcomp} to obtain compactness in the Aviles-Giga model for the potential $W(u)=(1-|u|^2)^2$. If the one-dimensional transition layer in \eqref{energy_1D} is known to be a global minimizer in 
$(\mathcal P)$ in the Aviles-Giga model in dimension $d=2$ (see \cite{Jin:2000}), i.e., $\mathrm{geod}_{W}^a(u^-,u^+)=c_W(u^-,u^+)$,  we will prove that the one-dimensional transition layer is actually the unique global minimizer in $(\mathcal P)$ (up to translation). Surprisingly, this can be done by use of the entropy method which was initially design to prove optimality rather than uniqueness.

It is instructive to think of the entropy method as an adaptation of the calibration method to the framework of divergence-free maps. In fact, the calibration method has been already used by Alberti-Ambrosio-Cabr\'e in \cite{Alberti:2001} in the context of the (scalar) De Giorgi conjecture in order to prove that the monotonicity assumption required on entire scalar solutions $u$ of \eqref{nodiv} (e.g., $\partial_1u>0$) implies local minimality of $u$. The outlook of the calibration method is the following: 

Assume that $\mathcal{E}$ is a functional defined on some functional space $\mathcal{A}$ composed of functions $u:\Omega\to X$ (e.g., $\Omega\subset\R^d$ and $X=\R^N$) and let $u_0\in\mathcal{A}$ be a critical point of $\mathcal E$. Then a calibration associated to the functional $\mathcal{E}$ and the critical point $u_0$ is a functional $\mathcal{F}$ defined on $\mathcal{A}$ such that:
\begin{description}
\item[(C1)]
$\mathcal{F}(u_0)=\mathcal{E}(u_0)$,
\item[(C2)]
$\mathcal{F}(u)\leq \mathcal{E}(u)$ for all $u\in\mathcal{A}$,
\item[(C3)]
$\mathcal{F}$ is a \emph{null-lagrangian}, i.e. $\mathcal{F}(u)=\mathcal{F}(v)$ whenever $u=v$ on $\partial\Omega$.
\end{description}
This immediately implies that $u_0$ is a global minimizer of $\mathcal{E}$ under Dirichlet boundary conditions (namely, $u=u_0$ on $\partial \Omega$). In that case, if $u$ is another global minimizer of $\mathcal{E}$ with $u=u_0$ on $\partial\Omega$, then $\mathcal{F}(u)=\mathcal{E}(u)$. In some cases, this equality will allow us to prove one-dimensional symmetry of a global minimizers. Recall that null-lagrangians can be expressed in a divergence form, see e.g. \cite{Giaquinta:1996}. Here, we use calibrations $\mathcal{F}$ of the form 
\[
\mathcal{F}(u)=\int_\Omega \nabla\cdot [\Phi(u)] \diff x,
\]
where $\Phi$ will denote an entropy. This yields the following definition (see also \cite{Ignat:2011} for alternative definitions), which amounts to 
imposing \textbf{(C2)} on the above $\mathcal{F}$ in the space $\mathcal{A}=\mathcal{C}^\infty\cap L^\infty\cap { \dot{H}_{div}^1}(\Omega, \R^d)$ and 
the energy functional $\mathcal{E}=E$:
\begin{definition}\label{entropy_def}
A map $\Phi \in \mathcal{C}^1(\R^d,\R^d)$ is called {\it entropy} if for all $u \in \mathcal{C}^\infty\cap L^\infty\cap { \dot{H}_{div}^1}(\Omega, \R^d)$ with $E(u)<\infty$, one has $\nabla\cdot[\Phi(u)]\in L^1(\Omega)$ and
\begin{equation}
\label{ineg_ent}
\int_\Omega\nabla\cdot[\Phi(u)]\diff x\leq E(u).
\end{equation}
\end{definition}

\begin{remark}\label{entropy_rem}
The condition $ \nabla\cdot[\Phi(u)]\in L^1$, imposed for every $u\in  \mathcal{C}^\infty\cap L^\infty$ with $\nabla\cdot u=0$ and $E(u)<\infty$, can be insured by the punctual condition:
\[
\exists C>0,\quad |\nabla \Phi(z)|^2\leq CW(z),\quad\forall z\in\R^d.
\]
Indeed, we have by Cauchy-Schwarz and Young inequalities,
\begin{equation}\label{BV_est}
\int_\Omega|\nabla\cdot[\Phi(u)]|\leq \int_\Omega|\nabla\Phi(u)|\,|\nabla u|\leq \int_\Omega\frac 12 \left(CW(u)+|\nabla u|^2\right)\leq \max\bigg\{1,\frac C2\bigg\} E(u)<\infty.
\end{equation}
Thus, an alternate definition of an entropy, stronger than Definition~\ref{entropy_def}, would be to impose $|\nabla\Phi|^2\leq 2W$, i.e. $C=2$ in the preceding inequalities (so that  \eqref{BV_est} implies \eqref{ineg_ent}). However, for the potentials $W$ we will look at, this condition is often too strong to allow the existence of an entropy. This is for instance the case of the Aviles-Giga potential in dimension $d=2$: 
\[
W(z)=(1-|z|^2)^2. 
\]
Indeed, if $|\nabla\Phi (z)|^2\leq CW(z)$ for all $z\in\R^2$, with $C>0$, then $\Phi$ must be constant on $\{W=0\}=\mathbb{S}^1$ and \textbf{(C1)} cannot be satisfied if $u^-\neq u^+$ since 
$$\mathcal{F}(u_0)=\int \nabla\cdot [\Phi(u_0)]\diff x=\Phi_1(u^+)-\Phi_1(u^-)=0<E(u_0)$$ by the Gauss-Green formula (see Lemma~\ref{gauss-green} below). More generally, this condition is too strong when $u^-$ and $u^+$ are on the same connected component in $\{W=0\}$.
\end{remark}
Now, if we set $\mathcal{F}(u):=\int_\Omega\nabla\cdot [\Phi (u)]\diff x$, then \textbf{(C3)} is automatically satisfied because of the (nonstandard) Gauss-Green formula. More precisely, if $\overline{u}(\pm\infty)=u^\pm\in S_a$ then one has the identity (see Lemma~\ref{gauss-green}):
\[
\mathcal{F}(u)=\Phi_1(u^+)-\Phi_1(u^-).
\]
Since our goal is to identify potentials $W$ such that the optimal $1D$ transition $u_0$ in \eqref{energy_1D} (given by Propositions \ref{Exist_profile1D} and \ref{1D_transition}) minimizes $(\mathcal P)$, it remains to check \textbf{(C3)}, i.e. $\mathcal{F}(u_0)=E(u_0)$. Since $E(u_0)=\mathrm{geod}^a_W(u^-,u^+)$, this condition, called \emph{saturation condition} for the entropy $\Phi$, depends on $u^\pm$ and reads
\begin{equation}
\label{satur_cond}
\Phi_1(u^+)-\Phi_1(u^-)=\mathrm{geod}^a_{W}(u^-,u^+).
\end{equation}
If there exists an entropy $\Phi$ satisfying the saturation condition, then the calibration method yields optimality of $u_0$ in the class $\mathcal{A}$ of smooth bounded divergence-free maps. Thanks to Lemma \ref{lavrentiev}, $\mathcal{F}$ is a calibration in the larger class $\tilde{\mathcal{A}}=\dot{H}_{div}^1(\Omega,\R^d)$ (thus yielding optimality of $u_0$ in this larger class) provided the growth condition \eqref{growth_W} on $W$ (needed in Lemma \ref{lavrentiev}).

The above arguments on how the entropy method proves the optimality of one-dimensional transition layers in $(\mathcal P)$ are summarized in the following:
\begin{proposition}
\label{minimal}
Let $W:\R^d\to\R_+$ be a continuous potential satisfying the growth condition \eqref{growth_W} and let $a\in\R$ be such that $S_a$ contains at least two points $u^\pm$. If there exists an entropy $\Phi\in\mathcal C^1(\R^d,\R^d)$ satisfying the saturation condition \eqref{satur_cond} then  one-dimensional transitions are optimal in $(\mathcal P)$, i.e.
\begin{equation}\label{minimal_unbnd}
\mathrm{geod}_{W}^a(u^-,u^+)=\inf\left\{E(u)\;:\; u\in \dot{H}_{div}^1(\Omega,\R^d),\,\overline{u}(\pm\infty)=u^\pm\right\}\ (=c_W(u^-,u^+)).
\end{equation}
\end{proposition}
In the case when a minimizing one-dimensional transition layer $u_0$ exists (see Propositions \ref{Exist_profile1D} and \ref{1D_transition}), \eqref{minimal_unbnd} yields $u_0$ to be a global minimizer of the $d$-dimensional problem $(\mathcal P)$. A-priori, in that context, other global minimizers in $(\mathcal P)$ might exist. We will see later some situations where the $1D$ minimizer $u_0$ is indeed the unique minimizer, i.e.,
the answer to Question 1 is positive.

When the growth condition \eqref{growth_W} on $W$ is dropped out, the conclusion of Proposition \ref{minimal} still holds provided that the infimum of $E$ is considered over \emph{bounded} admissible function $u$ in $\Omega$:
\begin{proposition}\label{minimal_bndProp}
If $W\in\mathcal{C}^2(\R^d,\R_+)$ and $a\in\R$ are such that $S_a$ contains at least two points $u^\pm$ and there exists an entropy $\Phi\in \mathcal C^1(\R^d,\R^d)$ satisfying the saturation condition \eqref{satur_cond}, then
\begin{equation}\label{minimal_bnd}
\mathrm{geod}_{W}^a(u^-,u^+)=\inf\left\{E(u)\;:\; u\in \dot{H}_{div}^1(\Omega,\R^d)\cap L^\infty(\Omega,\R^d),\,\overline{u}(\pm\infty)=u^\pm\right\}.
\end{equation}
\end{proposition}
Before proving Propositions \ref{minimal} and \ref{minimal_bndProp}, we need the following Gauss-Green type formula, applied on the unbounded domain $\Omega$ for bounded admissible maps:
\begin{lemma}\label{gauss-green}
For all $u\in L^\infty\cap \dot{H}^1(\Omega,\R^d)$ such that $\overline{u}(\pm\infty)=u^\pm$ and $\Phi\in\mathcal C^1(\R^d,\R^d)$ such that $\nabla\cdot[\Phi(u)]\in L^1(\Omega)$, one has
$$\int_\Omega \nabla\cdot [\Phi(u)]\diff x = \Phi_1(u^+)-\Phi_1(u^-).$$
\end{lemma}

\begin{proof}[Proof of Lemma~\ref{gauss-green}]
As a consequence of Lemma~\ref{BC_unif}, there exist two sequences $(R_n^\pm)_{n\ge 1}$ with $u(R_n^\pm,x')\to u^\pm$ for a.e.\ $x'\in\T^{d-1}$ as $n\to\infty$. By the Gauss-Green formula, applied to $ \Phi(u)\in H^1(\omega_n,\R^d)$ on the bounded domain $\omega_n:=(R_n^-,R_n^+)\times\T^{d-1}$, one has
\[
\int_{\omega_n}\nabla\cdot [\Phi(u)]\diff x=\int_{\T^{d-1}} \Bigl(\Phi_1(u(R_n^+,x'))-\Phi_1(u(R_n^-,x'))\Bigr)\diff x'.
\]
The conclusion follows from the dominated convergence theorem.
\end{proof}

\begin{proof}[Proof of Propositions \ref{minimal} and \ref{minimal_bndProp}]
The inequalities $\geq$ in \eqref{minimal_unbnd} and \eqref{minimal_bnd} follow from Proposition~\ref{inf_1D}. Conversely, we want to 
prove $\mathrm{geod}_{W}^a(u^-,u^+)\leq E(u)$ for every $u\in \dot{H}_{div}^1(\Omega,\R^d)$ such that $E(u)<\infty$ when either ($u\in L^\infty(\Omega,\R^d)$ and $W\in\mathcal{C}^2(\R^d,\R_+)$) or $W$ satisfies the growth condition \eqref{growth_W}. By Lemma~\ref{lavrentiev}, there exists a sequence $(u_k)_{k\in\N}\subset\mathcal{C}^\infty\cap L^\infty \cap \dot{H}_{div}^1(\Omega,\R^d)$ with $\overline{u_k}(\pm\infty)=u^\pm$ such that
\[
u_k\to u\quad \textrm{in } \dot{H}^1(\Omega)\quad\text{and}\quad E(u_k){\to} E(u) \quad \textrm{as } k\to\infty.
\]
Moreover, by Definition~\ref{entropy_def}, Lemma~\ref{gauss-green} and the saturation condition \eqref{satur_cond}, one has
\[
\mathrm{geod}_{W}^a(u^-,u^+)=\Phi_1(u^+)-\Phi_1(u^-)= \int_\Omega\nabla\cdot [\Phi(u_k)]\leq E(u_k)\quad\text{for each $k\in\N$}.
\]
In the limit $k\to\infty$, we obtain $\mathrm{geod}_{W}^a(u^-,u^+)\leq E(u)$, thus ending the proof.
\end{proof}

\paragraph{Our strategy to find entropies.} The easiest way to ensure that a given $\mathcal C^1$ map $\Phi$ is an entropy is to impose the punctual inequality $|\nabla\Phi|^2\leq 2W$ (see Remark \ref{entropy_rem}). This condition is too strong to have the saturation condition \eqref{satur_cond} fulfilled when $u^-$ and $u^+$ lie on the same connected component in $\{W=0\}$. However, due to the constraint $\nabla\cdot u=0$ imposed on our admissible maps, this condition can be relaxed in the weaker condition $|\Pi_0\nabla\Phi|^2\leq 2W$, where $\Pi_0$ is the projection onto the set of traceless matrices (see \eqref{strong_punct}), as explained in Section \ref{sym_intro} in Situation 1.
\begin{remark}
Knowing $\Pi_0\nabla\Phi$ is equivalent to knowing $\nabla\Phi$ up to an affine homothety. Indeed, if $\Phi$ is a $ \mathcal C^1$ map with $\Pi_0\nabla\Phi(z)=0$, namely $\nabla\Phi(z)=\alpha(z) {I}_d$ with $\alpha(z)\in\R$ for all $z\in\R^d$, then it is well known that 
$\Phi$ is an affine homothety, i.e. $\Phi(z)=\alpha z+\beta$ for every $z\in\R^d$, with $\alpha\in\R$ and $\beta\in\R^d$. Moreover, affine homotheties $\Phi$ are trivial for our problem in the sense that the corresponding calibration $\mathcal{F}(u)=\int_\Omega\nabla\cdot[\Phi(u)]\diff x$ vanishes for every $u\in\dot{H}^1_{div}(\Omega,\R^d)$.
\end{remark}

\medskip
\nd {\bf Situation 1. Strong punctual condition ($\mathcal{E}_{strg}$)}.
A sufficient condition for a  $ \mathcal{C}^1$  map $\Phi$ to be an entropy is the inequality \eqref{strong_punct} and this fact is shown by inequality 
\eqref{entropy_strong_est} for all $u \in \mathcal{C}^\infty\cap L^\infty\cap {\dot{H}_{div}^1}(\Omega, \R^d)$ with $E(u)<\infty$.
In fact, \eqref{strong_punct} is equivalent to imposing \eqref{entropy_strong_est} for all $u$:
\begin{proposition}\label{criterium_strong}
Given a map $\Phi\in \mathcal{C}^1(\R^d,\R^d)$, the two following conditions are equivalent:
\begin{itemize}
\item
$\nabla\cdot [\Phi (u)]\leq \frac 12|\nabla u|^2+W(u)$  for all $u\in L^\infty\cap\mathcal{C}^\infty(\Omega,\R^d)$ with $\nabla\cdot u=0$,
\item
$|\Pi_0\nabla\Phi (z)|^2\leq 2W(z)$ for all $z\in\R^d$.
\end{itemize}
In particular, if \eqref{strong_punct} holds true, then $\Phi$ is an entropy.
\end{proposition}
\begin{proof}
It is clear, by  \eqref{entropy_strong_est}, that the second condition implies the first one. 
Assume now that the first condition is satisfied. Fix $z_0\in\R^d$, an invertible matrix $p\in \R^{d\times d}$ such that $\mathrm{Tr}(p)=0$, and take a periodic map $u\in L^\infty\cap\mathcal{C}^\infty(\Omega,\R^d)$ 
with $\nabla\cdot u=0$ such that $u(x)=z_0+p^T x$ for all $x$ in a small ball $B\subset\Omega$ centered at the origin. Such a map exists; it can be constructed in two steps as follows.
\begin{enumerate}[leftmargin=15pt]
\item {\it Cut-off:} first, consider a divergence-free map $w\in \mathcal{C}^1\bigl(\overline{3B\setminus 2B},\R^d\bigr)$ such that $w(x)=z_0+p^T x$ on $\partial(2B)$ and $w=0$ on $\partial(3B)$ (such a function $w$ exists because the normal component $w\cdot \nu$ at $\partial(2B)$ has vanishing average on $\partial(2B)$ due to the assumption $\mathrm{Tr}(p)=0$); then we define the map $v:\Omega\to\R^d$ by $v(x)=z_0+p^T x$ on $2B$, $v(x)=w(x)$ on $3B\setminus 2B$ and $v(x)=0$ on $\Omega\setminus 3B$. In particular, $v$ is a $\mathcal{C}^1$ divergence-free map in $\Omega$.

\item {\it Mollification:} if $v$ is not smooth, we set $u:=\rho\ast v$ in $\Omega$, where $\rho$ is a smooth mollifying kernel concentrated on $B$ such that $\int_B y\rho(y)\diff y=0$ (this is true if for instance $\rho(-y)=\rho(y)$ for every $y\in B$); thus, $u$ is smooth, bounded, divergence-free and it is easy to see that $u(x)=z_0+p^T x$ on $B$.
\end{enumerate}
By the first inequality in the statement of the proposition, one has
$$\nabla\cdot[\Phi(u)](x)=\nabla\Phi (u(x)):p\leq \frac 12 |p|^2+W(u(x))\quad \text{for all }x\in B.$$
In particular, for $x=0$ and $u(0)=z_0$, we obtain
\[
\nabla\Phi (z_0):p\leq \frac 12 |p|^2+W(z_0).
\]
Since the set of invertible matrices $p\in {\rm Im}(\Pi_0)$ is dense in $ {\rm Im}(\Pi_0)$, we deduce that the above inequality is actually satisfied for every $p\in {\rm Im}(\Pi_0)$. By making the choice $p=\Pi_0\nabla\Phi(z_0)$, we conclude $\frac 12 |p|^2\leq W(z_0)$.
\end{proof}

The criterium provided by Situation 1 is not applicable in the Aviles-Giga situation. In that case, we need more sophisticated computations in the estimation of $\nabla\cdot[\Phi(u)]$ which we explain in the following.

\medskip
\nd {\bf Situations 2 and 3. Entropies with symmetric/antisymmetric Jacobian ($\mathcal{E}_{sym}$)/($\mathcal{E}_{asym}$)}. The main tool in the Situations 2 and 3 presented in Section \ref{sym_intro} relies on the following computation, which is inspired by the technique of Jin-Kohn \cite{Jin:2000} in dimension $d=2$: 
\begin{proposition}\label{antisymgradient}
For all $u=(u_1,\dots ,u_d)\in \dot{H}_{div}^1(\Omega,\R^d)$ such that $\bar{u}\in L^\infty(\R,\R^d)$, one has
\begin{equation}\label{symgradient}
\int_\Omega |\nabla u|^2=\int_\Omega \sum_{1\le i<j\le d}|\partial_i u_j-\partial_j u_i|^2=\int_\Omega \sum_{1\le i<j\le d}|\partial_i u_j+\partial_j u_i|^2 +2\int_\Omega\sum_{i=1}^d |\partial_i u_i|^2.
\end{equation}
In other words, if $\Pi^+$ (resp. $\Pi^-$) denotes the projection on the set of symmetric (resp. of antisymmetric) matrices, then
\eqref{equiProjection} holds true.
\end{proposition}
\begin{proof} First notice that \eqref{equiProjection} is a rewriting of \eqref{symgradient} in terms of $\Pi ^\pm\nabla u$ because
\[
|\Pi^-\nabla u|^2=\frac 12\sum_{i<j}|\partial_iu_j-\partial_ju_i|^2\quad\text{while}\quad|\Pi^+\nabla u|^2=\sum_i|\partial_iu_i|^2+\frac 12\sum_{i<j}|\partial_iu_j+\partial_ju_i|^2\quad a.e.
\]
We now prove \eqref{symgradient}. Since $u\in \dot{H}^1(\Omega,\R^d)$, up to convolution with a smooth kernel (as in the proof of Lemma~\ref{lavrentiev}), one can assume that $u\in\mathcal{C}^\infty\cap \dot{H}_{div}^1(\Omega,\R^d)$. Then we compute
\begin{equation}\label{identity}
\sum_{i<j} |\partial_i u_j\pm \partial_j u_i|^2=\sum_{i\neq j} |\partial_i u_j|^2\pm 2\sum_{i<j}\partial_i u_j \partial_j u_i= \sum_{i\neq j}\left(|\partial_i u_j|^2\pm \partial_i u_j \partial_j u_i\right) .
\end{equation}
Together with the identity $0=|\nabla\cdot u|^2=|\sum_i \partial_i u_i|^2=\sum_{i,j}\partial_i u_i \partial_j u_j$, \eqref{identity} implies
\begin{equation*}
\sum_{i<j} |\partial_i u_j - \partial_j u_i|^2
=|\nabla u|^2-\sum_{i,j} \partial_i u_j \partial_j u_i
=|\nabla u|^2-\sum_{i,j}\left(\partial_i u_j \partial_j u_i - \partial_i u_i \partial_j u_j\right).
\end{equation*}
In order to prove the first identity in \eqref{symgradient}, we have to prove that integrating the last term of the above RHS, we obtain $0$. Let us use the notation $b_{ij}=\partial_i u_j \partial_j u_i - \partial_i u_i \partial_j u_j\in L^1(\Omega)$ because $u\in \dot{H}^1(\Omega, \R^d)$. For all $i,j\in\{1,\dots,d\}$, one has $b_{ij}=b_{ji}$ and $b_{ii}=0$. Moreover, if $i\neq 1$ and $j\neq 1$, since $x_i$ and $x_j$ lie on the torus $\T$ which has no boundary, integrating by parts twice yields
\[
\int_{\T^{d-1}}\partial_iu_i\partial_ju_j\diff x'=-\int_{\T^{d-1}}u_i \partial_{ij}u_j\diff x'=\int_{\T^{d-1}}\partial_ju_i\partial_iu_j\diff x',
\quad\text{i.e. }\int_\Omega b_{ij}=0, \quad i\neq 1, j\neq 1.
\]
It remains to prove that
$$\sum_{j\neq 1}\int_{\Omega} b_{1j}\diff x=0.$$
Indeed, for every $R^-,R^+\in\R$ with $R^-<R^+$, integrating $\partial_1u_1\partial_ju_j$ by parts on $[R^-,R^+]\times \T^{d-1}$ twice (so as to switch $\partial_1$ and $\partial_j$), and using the divergence constraint, it yields
\begin{align*}
\bigg|\sum_{j\neq 1}\int_{[R^-,R^+]\times \T^{d-1}} b_{1j}\diff x\bigg|&
=\bigg|\sum_{j\neq 1}\int_{[R^-,R^+]\times \T^{d-1}}(\partial_1u_1\partial_ju_j-\partial_1 u_j\partial_j u_1)\diff x\bigg|\\
&=\bigg|\sum_{j\neq 1}\int_{\T^{d-1}}u_1(R^+,x')\partial_ju_j(R^+,x')-u_1(R^-,x')\partial_ju_j(R^-,x')\diff x'\bigg|\\
&=\bigg|\int_{\T^{d-1}}u_1(R^+,x')\partial_1u_1(R^+,x')-u_1(R^-,x')\partial_1u_1(R^-,x')\diff x'\bigg|\\
&\leq\|u_1(R^+,\cdot)\|_{L^2}\|\partial_1u_1(R^+,\cdot)\|_{L^2}+\|u_1(R^-,\cdot)\|_{L^2}\|\partial_1u_1(R^-,\cdot)\|_{L^2}.
\end{align*}
Now, Remark \ref{rem:BC} yields two sequences $(R_n^\pm)_{n\ge 1}\to\pm\infty$ with 
$\|u_1(R_n^\pm,\cdot)\|_{L^2}\|\partial_1u_1(R_n^\pm,\cdot)\|_{L^2}\to 0$; the above inequality applied to $R^\pm=R^\pm_n$ yields the claimed identity in the limit $n\to\infty$. 

For the second equality in \eqref{symgradient}, because of the condition $\nabla\cdot u=0$ and \eqref{identity}, one has
\begin{align*}
\begin{split}
\sum_{i<j} |\partial_i u_j + \partial_j u_i|^2
&=|\nabla u|^2+\sum_{i,j} \partial_i u_j \partial_j u_i -2\sum_i |\partial_i u_i|^2\\
&=|\nabla u|^2-2\sum_i |\partial_i u_i|^2+\sum_{i,j}\left(\partial_i u_j \partial_j u_i - \partial_i u_i \partial_j u_j\right).
\label{reste_ND}
\end{split}
\end{align*}
Again, \eqref{symgradient} follows from the fact that $\int_\Omega \sum_{i,j} b_{ij}=0$.
\end{proof}
We now explain how we use this proposition to find entropies. Let us consider a map {$\Phi\in\mathcal{C}^1(\R^d,\R^d)$}, assume that $\Pi_0\nabla\Phi (z)$ is either symmetric for all $z$ or antisymmetric for all $z$, denoted shortly by
\[
\Pi_0\nabla\Phi (z)\in {\rm Im}(\Pi^\pm)\quad \text{ for all }z\in\R^d. 
\]
Then for all $u\in \mathcal{C}^\infty\cap L^\infty\cap \dot{H}_{div}^1(\Omega,\R^d)$ and $E(u)<+\infty$, by self-adjointness of an orthogonal projection, one obtains \eqref{11}.
Now, by Young's inequality $2st\leq s^2/2+2t^2$ and Proposition~\ref{antisymgradient} (note that $\bar u\in L^\infty$ because $u$ does it), then
\begin{align*}
\int_\Omega \nabla\cdot[\Phi(u)]\diff x
&\leq \frac 12\left(\frac 12\|\Pi_0\nabla\Phi(u)\|_{L^2(\Omega)}^2+2\|\Pi^\pm\nabla u\|_{L^2(\Omega)}^2\right)\\
&= \frac 12\left(\frac 12\|\Pi_0\nabla\Phi(u)\|_{L^2(\Omega)}^2+\|\nabla u\|_{L^2(\Omega)}^2\right).
\end{align*}
Moreover, since $E(u)<\infty$, the condition $\nabla\cdot[\Phi(u)]\in L^1(\Omega)$ can be insured by imposing $|\Pi_0\nabla\Phi|^2\leq C W$ for some constant $C>0$. Thus, if $C=4$, the above argument yields the following proposition:
\begin{proposition}\label{criterium_weak}
{Let $\Phi\in\mathcal{C}^1(\R^d,\R^d)$} be a map such that $\Pi_0\nabla\Phi $ is either symmetric in $\R^d$ or antisymmetric in all $\R^d$, and such that $|\Pi_0\nabla\Phi|^2\leq 4W$. Then $\Phi$ is an entropy.
\end{proposition}

\begin{remark}\label{hessian}
It is well known that ($\mathcal{E}_{sym}$) (i.e., \eqref{ent_sym}) implies that there exists {$\Psi\in \mathcal{C}^2(\R^d,\R)$} such that
$\nabla\Phi(z)=\nabla^2\Psi(z)$ for all $z\in\R^d$, where $\nabla^2\Psi(z)$ is the Hessian matrix of $\Psi$. In other words, there exists a constant $\Phi_0\in\R^d$ such that
\[
\Phi(z)=\Phi_0+\nabla\Psi(z)\quad\text{for all }z\in\R^d.
\]
\end{remark}

\medskip
\nd {\bf The saturation condition}. It remains to confront the above estimates to the saturation condition \eqref{satur_cond} for two fixed zeros $u^\pm$ of $W$ such that $u^\pm\in\R^d_a$ for some $a\in\R$. Assume that there exists $\gamma$ which achieves the infimum in the definition \eqref{geod_H1} of $\mathrm{geod}_{W}^a$, i.e.
\begin{equation}\label{gamma_geod}
\gamma\in {\rm Argmin}\left\{\int_{-1}^1 \sqrt{2W(\gamma(s))}\, |\dot{\gamma}(s)|\diff s \;:\; \gamma\in\mathrm{Lip}([-1,1],\R^d_a),\, \gamma(\pm 1)=u^\pm\right\}
\end{equation}
(see Propositions \ref{Exist_profile1D} and \ref{1D_transition} for sufficient conditions, and Proposition \ref{inf_1D} for the link between $\mathrm{geod}_{W}^a$ and the 1D minimization problem in \eqref{energy_1D}).
For a map $\Phi\in\mathcal{C}^1(\R^d,\R^d)$, the saturation condition \eqref{satur_cond} rewrites as
\begin{equation}
\label{satur_diff}
\int_{-1}^1 \nabla \Phi_1(\gamma(t)) \cdot \dot{\gamma}(t)\diff t = \int_{-1}^1 \sqrt{2W(\gamma (t))}\, |\dot{\gamma}(t)|\diff t.
\end{equation}
We now combine \eqref{satur_diff} with the conditions assumed on $\Pi_0\nabla\Phi$ in Proposition~\ref{criterium_strong} or Proposition~\ref{criterium_weak}, that is one of the criteria ($\mathcal{E}_{strg}$), ($\mathcal{E}_{sym}$) or ($\mathcal{E}_{asym}$) in \eqref{strong_punct}, \eqref{ent_sym} or \eqref{ent_asym}.
 In fact, the condition \eqref{satur_diff} implies a saturation of the inequalities in ($\mathcal{E}_{strg}$), ($\mathcal{E}_{sym}$) and ($\mathcal{E}_{asym}$) on the range ${\rm Im}(\gamma)$ so that $\Pi_0\nabla\Phi$ is fully determined on ${\rm Im}(\gamma)$:
\begin{proposition}\label{satur}
Assume that there exists $\gamma\in {\rm Lip}([-1,1],\R^d_a)$ satisfying \eqref{gamma_geod} and a map $\Phi\in\mathcal{C}^1(\R^d,\R^d)$ satisfying the saturation condition \eqref{satur_cond}. Then
\begin{itemize}
\item
$\Pi_0\nabla\Phi(\gamma(t))=\sqrt{2W(\gamma(t))}\ e_1\otimes \frac{\dot{\gamma}(t)}{|\dot{\gamma}(t)|}$ a.e.\ in $[-1,1]$ if $\Phi$ satisfies ($\mathcal{E}_{strg}$) in \eqref{strong_punct},
\item
$\Pi_0\nabla\Phi(\gamma(t))=2\sqrt{2W(\gamma(t))}\ \Pi^+\left(e_1\otimes \frac{\dot{\gamma}(t)}{|\dot{\gamma}(t)|}\right)$ a.e.\ in $[-1,1]$ if $\Phi$ satisfies ($\mathcal{E}_{sym}$) in \eqref{ent_sym},
\item
$\Pi_0\nabla\Phi(\gamma(t))=2\sqrt{2W(\gamma(t))}\ \Pi^-\left(e_1\otimes \frac{\dot{\gamma}(t)}{|\dot{\gamma}(t)|}\right)$ a.e.\ in $[-1,1]$ if $\Phi$ satisfies ($\mathcal{E}_{asym}$) in \eqref{ent_asym}.
\end{itemize}
\end{proposition}
\begin{proof}
First assume that ($\mathcal{E}_{strg}$) is fulfilled. Since $\dot \gamma_1(t)=0$ a.e.\ in $[-1,1]$ and $\Pi_0\nabla\Phi$ and $\nabla\Phi$ coincide out of the diagonal, the Cauchy-Schwarz inequality and \eqref{strong_punct} imply 
\[
\nabla\Phi_1 (\gamma)\cdot \dot{\gamma}=\Pi_0\nabla\Phi(\gamma): (e_1\otimes\dot{\gamma})
\le\sqrt{2W(\gamma)}\, |\dot{\gamma}|\quad\text{a.e. in $[-1,1]$.}
\]
Combined with  \eqref{satur_cond} which rewrites as \eqref{satur_diff}, we deduce that $\nabla\Phi_1 (\gamma)\cdot \dot{\gamma}=\sqrt{2W(\gamma)}\, |\dot{\gamma}|$ a.e. in $[-1,1]$. In other words, we have
\[
\Pi_0\nabla\Phi(\gamma): \Big(e_1\otimes \frac{\dot{\gamma}}{|\dot{\gamma}|}\Big)=\sqrt{2W(\gamma)}\quad\text{a.e. in $[-1,1]$}
\]
which implies the claim by the case of equality in the Cauchy-Schwarz inequality $A:B\le |A|\,|B|$ for matrices $A,B\in\R^{d\times d}$.

Similarly, in the cases where ($\mathcal{E}_{sym}$) (read $\pm=+$ in the following) or ($\mathcal{E}_{asym}$) (read $\pm=-$ in the following) are fulfilled, we have $|\Pi_0\nabla\Phi|^2\le 4W$ and we deduce
\[
\nabla\Phi_1 (\gamma)\cdot \dot{\gamma}=\Pi_0\Pi^\pm\nabla\Phi(\gamma): (e_1\otimes\dot{\gamma})=\Pi_0\nabla\Phi(\gamma): \Pi^\pm(e_1\otimes\dot{\gamma})
\le\sqrt{2W(\gamma)}\, |\dot{\gamma}|\quad\text{a.e. in $[-1,1]$,}
\]
where we used the equality $|\Pi^\pm(e_1\otimes\dot{\gamma})|=\frac{|\dot{\gamma}|}{\sqrt{2}}$. The claim follows by saturation of \eqref{satur_diff} and by the case of equality in the Cauchy-Schwarz inequality.
\end{proof}

\subsection{Structure of global minimizers}\label{sec:structure}

The aim of this section is to highlight that the existence of an entropy $\Phi$ satisfying the saturation condition \eqref{satur_cond} and one of the conditions ($\mathcal{E}_{strg}$), ($\mathcal{E}_{sym}$) or ($\mathcal{E}_{asym}$) in \eqref{strong_punct}, \eqref{ent_sym} or \eqref{ent_asym}, implies that any solution to the global minimization problem $(\mathcal P)$ satisfies a first order PDE which encodes in particular the equipartition of the energy density, and implies in general one-dimensional symmetry:

\begin{proposition}\label{pde_opt}
Let $W:\R^d\to\R_+$ be a continuous potential and $a\in\R$ such that $S_a$ contains at least two wells $u^\pm$ of $W$, and assume that there exists $\Phi\in\mathcal{C}^1(\R^d,\R^d)$ satisfying the saturation condition \eqref{satur_cond} and either ($\mathcal{E}_{strg}$), ($\mathcal{E}_{sym}$) or ($\mathcal{E}_{asym}$) in \eqref{strong_punct}, \eqref{ent_sym} or \eqref{ent_asym}. 

If $u$ is a global minimizer of $(\mathcal{P})$ and if either ($u\in L^\infty(\Omega,\R^d)$ and $W\in\mathcal{C}^2(\R^d,\R_+)$) or $W$ satisfies the growth condition \eqref{growth_W}, then \eqref{pdeStrong} (resp. (\eqref{pdeAsym}), (\eqref{pdeSym})) in the case ($\mathcal{E}_{strg}$) (resp. ($\mathcal{E}_{asym}$), ($\mathcal{E}_{sym}$)) holds true.

Conversely, if $W$ satisfies the growth condition \eqref{growth_W} and $u\in\dot{H}_{div}^1(\Omega,\R^d)$ solves \eqref{pdeStrong} (resp. \eqref{pdeAsym}, \eqref{pdeSym}) in the case ($\mathcal{E}_{strg}$) (resp. ($\mathcal{E}_{asym}$), ($\mathcal{E}_{sym}$)), then $u$ is a global minimizer of $(\mathcal{P})$.
\end{proposition}

\begin{remark}
If $u$ is one-dimensional, then $u=u(x_1)=(a,\varphi(x_1))\in\R^d_a$, and the first order PDE in \eqref{pdeStrong}, 
\eqref{pdeAsym} and \eqref{pdeSym} is equivalent to the ODE
\(
\dot{\varphi}(t)=[\partial_2\Phi_1,\dots,\partial_d\Phi_1](a,\varphi(t)).
\)
\end{remark}
\begin{proof}
We will focus on the third case ($\mathcal{E}_{sym}$), the first and second cases are similar. Namely, assume that $\Phi\in\mathcal{C}^1(\R^d,\R^d)$ satisfies \eqref{satur_cond}, that $\nabla\Phi$ is symmetric and $|\Pi_0\nabla\Phi|^2\leq 4W$. If $u\in\mathcal{C}^\infty\cap L^\infty$ with $\nabla\cdot u=0$, then by \eqref{11},
\begin{align*}
 \nabla\cdot[\Phi(u)]
 =\Pi_0\nabla\Phi(u):\Pi^+\nabla u^T= \frac 14|\Pi_0\nabla\Phi(u)|^2+|\Pi^+\nabla u|^2-\frac 14\left|\Pi_0\nabla\Phi(u)-2\Pi^+\nabla u\right|^2.
\end{align*}
If $E(u)<\infty$, we have $\nabla\cdot[\Phi(u)]\in L^1(\Omega)$ due to $|\Pi_0\nabla\Phi(u)|^2\leq 4W(u)$ in \eqref{ent_sym}; thus, from Lemma~\ref{gauss-green}, Proposition~\ref{antisymgradient} and the boundary condition $\overline{u}(\pm\infty)=u^\pm$, we deduce by integrating the preceding identity that
\begin{equation}
\label{22}
\Phi_1(u^+)-\Phi_1(u^-)=E(u)- \int_\Omega \Big(W(u)-\frac 14|\Pi_0\nabla\Phi(u)|^2\Big)\diff x-\frac 14\int_\Omega\left|\Pi_0\nabla\Phi(u)-2\,\Pi^+\nabla u\right|^2\diff x.
\end{equation}
Since each term of the RHS is controlled by the energy density $\frac 12|\nabla u|^2+W(u)$ and since the integrands depend continuously on \(u\), we deduce by Lemma~\ref{lavrentiev} that the relation \eqref{22} still holds for all $u\in H^1_{div}(\Omega,\R^d)$ with $E(u)<+\infty$ and $\overline{u}(\pm\infty)=u^\pm$, without assuming that $u$ is smooth, but only that $u$ is bounded whenever $W$ does not satisfy \eqref{growth_W} and $W\in \mathcal{C}^2$. Moreover, since $W\ge\frac 14|\Pi_0\nabla\Phi|^2$, the last two terms in \eqref{22} are nonnegative; in particular, $E(u)\ge \Phi_1(u^+)-\Phi_1(u^-)$.
Now, by the saturation condition \eqref{satur_cond} and Proposition \ref{inf_1D}, $\Phi_1(u^+)-\Phi_1(u^-)=\mathrm{geod}_{W}^a(u^-,u^+)$ coincides with the infimum of the energy over 1D transitions. Thus, if $u$ is a global minimizer of $(\mathcal P)$ and either ($u\in L^\infty(\Omega,\R^d)$ and $W\in\mathcal{C}^2(\R^d,\R_+)$) or $W$ satisfies the growth condition \eqref{growth_W}, then the last two terms of \eqref{22} vanish, i.e. \eqref{pdeSym} holds true.

Conversely, if $u$ solves \eqref{pdeSym}, then \eqref{22} yields $E(u)=\Phi_1(u^+)-\Phi_1(u^-)$. Since $W$ satisfies \eqref{growth_W} and $\Phi$ satisfies ($\mathcal{E}_{sym}$), then \eqref{22} gives also $E(v)\ge \Phi_1(u^+)-\Phi_1(u^-)$ for every $v\in H^1_{div}(\Omega,\R^d)$ with $E(v)<+\infty$ and $\overline{v}(\pm\infty)=u^\pm$ (without assuming $v$ bounded). In particular, $u$ is a global minimizer of $(\mathcal{P})$.
\end{proof}

In dimension $d\geq 3$ there is no hope for uniqueness of global minimizers, even up to a translation in $x_1$-direction; in fact, $1D$ solutions of \eqref{energy_1D} need not be unique when $\R^d_a$ is of dimension $d-1\geq 2$ since there could be two distinct minimizers of $\mathrm{geod}_{W}^a$ connecting $u^-$ to $u^+$ within the hyperspace $\R^d_a$.
Therefore, in these cases, there is no uniqueness in the first order PDE in \eqref{pdeStrong}, 
\eqref{pdeAsym} and \eqref{pdeSym}.
Nevertheless, we will prove in the following that a necessary condition in having uniqueness is given by a punctual condition $u(x_0)=u_0\in\R^d$. For that, we will focus on the cases ($\mathcal{E}_{strg}$) and ($\mathcal{E}_{sym}$) because in the case ($\mathcal{E}_{asym}$) there are only ``trivial'' entropies as we shall see in Proposition~\ref{rigidity_antisym} (thus, it is useless in proving uniqueness in that case). 
Note that in those two cases, the first order PDE system in \eqref{pdeStrong} and \eqref{pdeSym}  is of the form $\nabla u=F(u)$ when ($\mathcal{E}_{strg}$) holds true, and $\Pi^+\nabla u=F(u)$ when ($\mathcal{E}_{sym}$) is satisfied, where $F$ maps $\R^d$ into the set of square matrices. If $F$ is locally Lipschitz \footnote{Note that if $\Phi\in W_{loc}^{2,\infty}$ then $F$ corresponds to a locally Lipschitz map.}, it is clear that $\mathcal{C}^1$-solutions of $\nabla u=F(u)$ such that $u(x_0)=u_0$ are unique by the Cauchy-Lipschitz theorem (applied to $t\mapsto u(x_0+tv)$ which satisfies an ODE whatever $v\in\R^d$). Equations of the form $\Pi^+\nabla u=F(u)$ are weaker (obviously, they cover the first class of equations since if $\nabla u=F(u)$ then $\Pi^+\nabla u=\Pi^+F(u)=\tilde{F}(u)$) and we show that they enjoy a similar uniqueness property in the case of Lipschitz solutions (which is coherent with the regularity in Proposition \ref{pro:reg}):
\begin{proposition}
If $F\in\mathrm{Lip}_{loc}(\R^d,\R^{d\times d})$ and $v,w\in \mathrm{Lip}_{loc}(\Omega,\R^d)$ are two solutions of the system 
\(
\Pi^+\nabla u=F(u)\text{ a.e.}
\)
such that $v(x_0)=w(x_0)$ for some $x_0\in\Omega$, then $v=w$.
\end{proposition}
\begin{proof}
Let us fix $R>1$ such that $x_0\in B(0,R-1)$. Since $u:=v-w$ is Lipschitz in $B(0,R)$ and $u(x_0)=0$, one has
\[
|u(x)|\leq L_R\; \mathrm{dist}_\Omega(x,x_0)\quad\text{for all $x\in B(0,R)$,}
\]
where $L_R>0$ is the Lipschitz constant of $u$ in $B(0,R)$ and \(\mathrm{dist}_\Omega\) is the distance induced on \(\Omega=\R\times\T^{d-1}\) by the euclidean distance in \(\R^d\) (via the quotient map). By Korn's inequality and a compactness/scaling argument, 
we deduce that for all $r<1$,
\be
\label{77}
\int_{B(x_0,r)}|u|^2\diff x\le C_1r^2\int_{B(x_0,r)} |\Pi^+\nabla u|^2\diff x,
\ee
where the constant $C_1>0$ only depends on $R$, $L_R$ and the dimension $d$. Since $F$ is Lipschitz in a ball containing $v(B(0,R))\cup w(B(0,R))$, by the ODE satisfied by $v$ and $w$, we have $|\Pi^+\nabla u|=
|F(v)-F(w)|\leq C_F|u|$ for some $C_F>0$. Combined with \eqref{77}, since $B(x_0,r)\subset B(0,R)$, we finally deduce
\[
\int_{B(x_0,r)}|u|^2\diff x\le C_2r^2\int_{B(x_0,r)}|u|^2\diff x,
\]
where the constant $C_2>0$ depends on $R$, but is independent of $x_0\in B(0,R-1)$ and $r<1$. This implies that $u=v-w=0$ on $B(x_0,r_0)$ for the choice of a small radius $r_0:=\frac 12C_2^{-1/2}$. Since $r_0$ is uniform for all $x_0\in B(0,R-1)$, applying the same reasoning when the point $x_0$ is replaced by any point in $B(x_0,r_0)$, and repeating the procedure inductively, yield the equality $v=w$ a.e.\ in $B(0,R)$. Since this is true for all $R>0$, we have proved $v=w$.
\end{proof}
The preceding result does not imply one-dimensional symmetry of solutions of the first order PDE in \eqref{pdeStrong}, 
\eqref{pdeAsym} and \eqref{pdeSym} in dimension $d\geq 3$ (thus, of global minimizers of $(\mathcal P)$) due to the additional assumption $v(x_0)=w(x_0)$. 
In the cases ($\mathcal{E}_{strg}$) or ($\mathcal{E}_{sym}$), a simple situation where this one-dimensional symmetry holds is given by entropies $\Phi$ satisfying additionally \eqref{33}.

\begin{proposition}\label{diagonal_entropy}
Let $W:\R^d\to\R_+$ be a continuous potential and $a\in\R$ such that $S_a$ contains at least two wells $u^\pm$ and assume that there exists an entropy $\Phi=(\Phi_1,\dots,\Phi_d)\in\mathcal{C}^1(\R^d,\R^d)$ satisfying either ($\mathcal{E}_{strg}$) or ($\mathcal{E}_{sym}$), the saturation condition \eqref{satur_cond} and the condition \eqref{33}. If $u$ is a global minimizer of $(\mathcal P)$ such that either ($u\in L^\infty(\Omega,\R^d)$ and $W\in\mathcal{C}^2(\R^d,\R_+)$) or $W$ satisfies the growth condition \eqref{growth_W} then $u$ is one-dimensional, i.e. $u=g(x_1)$ for some $g\in\dot{H}^1(\R,\R^d_a)$.
\end{proposition}
\begin{proof}
Since the diagonal of $\Pi_0\nabla\Phi$ vanishes, it follows from Proposition~\ref{pde_opt} that $\partial_iu_i=0$ a.e.\ for all $i\in\{1,\dots,d\}$. In particular, $u_1(x_1,x')$ does not depend on $x_1\in\R$. By Lemma~\ref{BC_unif}, there exist two sequences $(R_n^\pm)_{n\ge 1}$ such that $R_n^\pm\to\pm\infty$ and $u(R_n^\pm,x')\to u^\pm\in S_a$ for a.e.\ $x'\in \T^{d-1}$ as $n\to +\infty$. In particular, $u_1\equiv a$, i.e. $u\in\R^d_a$ a.e. in $\Omega$, and the conclusion follows from Corollary~\ref{1Dcriterium}.
\end{proof}
We will explain in Section \ref{higherdimension} how Theorem \ref{thm:rigid_sym} is a consequence of the above 
Proposition~\ref{diagonal_entropy}.

\subsection{One-dimensional symmetry in dimension $2$. Proof of Theorems \ref{thm:harm_wave} and \ref{thm:tricomi} \label{one_dim_2D}}
Our aim is to identify potentials $W$ for which one has existence of an entropy, and so optimality of the $1D$ transition layers (by Proposition~\ref{minimal}). We also want to deduce rigidity results from the entropy method, i.e., every global minimizer in $(\mathcal P)$ is one-dimensional symmetric. We will restrict ourselves to very specific potentials. Namely, we impose that there exists $w\in\mathcal{C}^2(\R^2,\R)$ such that $W(z)=\frac 12w^2 (z)$ for all $z\in\R^2$ and
\begin{equation*}
\partial_{11}w\pm \partial_{22}w=0 \quad \textrm{in } \R^2.
\end{equation*}
We will also discuss the case of more general potentials $W=\frac 12w^2$ with $w$ being a solution of the Tricomi equation.

\paragraph{Existence of entropies.} { We start with the case of $W=\frac 12w^2$ where $\Delta w=0$ or $\square w=0$ in $\R^2$.} Motivated by Proposition~\ref{criterium_weak}, we look for entropies $\Phi\in\mathcal{C}^1(\R^2,\R^2)$ that subscribe to Situation 2 or 3 (i.e., \eqref{ent_sym} or \eqref{ent_asym} hold true), namely, we impose the punctual condition \eqref{ansatz_harmonic} on $\Phi$:
\begin{equation*}
\Pi_0\nabla\Phi(z)=\nabla\Phi (z)+\alpha(z)I_2=\left(
\begin{matrix} 
0& w(z)\\
\mp w(z)&0
\end{matrix}
\right)
\quad\text{for all }z\in\R^2,
\end{equation*}
where $\alpha$ is a scalar function to be determined. In the case $\mp=-$ corresponding to the antisymmetry of $\Pi_0\nabla\Phi$
in \eqref{ent_asym}, by Cauchy-Riemann, this condition implies that $\Phi$ is holomorphic on $\R^2\simeq \C$ which is coherent with the assumption on $w$ to be harmonic. Indeed, writing $\Phi (z)=\Phi_1 (x,y)+i\Phi_2(x,y)$ for $z=x+iy\in\mathbb{C}\simeq \R^2$, \eqref{ansatz_harmonic} implies that $\Phi$ is holomorphic with the (complex) derivative $-\partial_z \Phi(z)=\alpha(z)+iw(z)$; then $\alpha$ is the harmonic conjugate of $w$ (defined up to an additive constant).
In the case $\mp=+$, the symmetry of $\Pi_0\nabla\Phi$ is coherent with the assumption on $w$ to be a solution of the wave equation and corresponds to $(\mathcal{E}_{sym})$ in \eqref{ent_sym}.

{
\begin{lemma}
\label{lem_entropy_harmonic}
Let $W=\frac 12w^2$ with $w\in\mathcal{C}^2(\R^2,\R)$ such that $\partial_{11}w\pm\partial_{22} w=0$ in $\R^2$ and for some $a\in\R$, $w$ vanishes at $u^\pm=(a,u^\pm_2)\in S_a$ and $w\geq 0$ on the segment $[u^-,u^+]$.
Then there exists an entropy $\Phi\in\mathcal{C}^3(\R^2, \R^2)$ satisfying the saturation condition \eqref{satur_cond} together with \eqref{ansatz_harmonic} for some $\alpha\in\mathcal{C}^2(\R^2)$. 
\end{lemma}

\begin{proof} By the Poincar\'e lemma, we know that the existence of a map $\Phi$ satisfying \eqref{ansatz_harmonic} is equivalent to the system
$$
-\partial_2\alpha -\partial_1 w=\pm\partial_2 w-\partial_1\alpha =0,
$$
which rewrites $\nabla\alpha=(\pm\partial_2 w,-\partial_1 w)$. Applying again the Poincar\'e lemma, the last equality for $\alpha$ is equivalent to the equation $\partial_{11}w\pm \partial_{22}w=0$ as stated in 
our assumption. Moreover, in this case, $\alpha$ satisfies the same equation as $w$.

Let us check now that a map $\Phi$ with \eqref{ansatz_harmonic} is an entropy that satisfies the saturation condition \eqref{satur_cond}. Since we assume that $w$ is $\mathcal{C}^2$, we know that $\alpha$ is $\mathcal{C}^2$ and that $\Phi$ is $\mathcal{C}^3$. The fact that $\Phi$ is an entropy is a consequence of Proposition~\ref{criterium_weak} since $|\Pi_0\nabla\Phi(z)|^2=4W(z)$. The saturation condition \eqref{satur_cond} follows from the equality $\partial_2 \Phi_1 (z)=w(z)=\sqrt{2W(z)}$ for all $z\in [u^-,u^+]$, where we use the assumption $w\ge 0$ on $[u^-,u^+]$. Indeed, one has
\be
\label{eq2018}
\Phi_1(u^+)-\Phi_1(u^-)=\int_{u_2^-}^{u_2^+} \partial_2\Phi_1(a,z_2)\diff z_2=\int_{u_2^-}^{u_2^+} \sqrt{2W(a,z_2)}\diff z_2
\stackrel{\eqref{min_en_1D}}{=}\mathrm{geod}_{W}^a(u^-,u^+).
\ee
\end{proof}

\paragraph{One-dimensional symmetry for $(\mathcal P)$. Proof of Theorem \ref{thm:harm_wave}.} For the two previous classes of potentials, Lemma \ref{lem_entropy_harmonic} and Proposition \ref{pde_opt} yield one-dimensional symmetry of global minimizers. We start with the case where $w$ solves $\square w=0$:
\begin{theorem}\label{symmetry_wave}
Let $W:\R^2\to\R_+$ be a continuous potential and $a\in\R$ such that $S_a$ contains at least two points $u^\pm=(a,u^\pm_2)$. Assume that $W\geq \frac 12w^2$ on $\R^2$ and ($W=\frac 12w^2$ and $w>0$ on $(u^-,u^+)$), where $w\in\mathcal{C}^2(\R^2,\R)$ solves the wave equation $\square w=0$ in $\R^2$.
If $u$ is a global minimizer in $(\mathcal P)$ such that either $u\in L^\infty$ or $|w|$ satisfies the growth condition \eqref{growth_intro}, then $u$ is one-dimensional, i.e. $u(x)=g(x_1)$ a.e.\ where $g:\R\to\R^2_a$ is, up to a translation in the $x_1$-variable, the unique one-dimensional transition layer given by Proposition~\ref{Exist_profile1D}.
\end{theorem}
\begin{proof}[Proof of Theorem~\ref{symmetry_wave}] 
First, we note that by Lemma \ref{lem_entropy_harmonic}, there exists an entropy $\Phi\in\mathcal{C}^3(\R^2, \R^2)$ associated to the potential $\frac 12w^2$ satisfying the saturation condition \eqref{satur_cond} for this potential $\frac 12w^2$, together with \eqref{ansatz_harmonic} (with $\mp=+$ and $\alpha\in\mathcal{C}^2(\R^2)$). The symmetry of global minimizers is proved by considering two cases:

\medskip

\nd {\it Case 1: $W=\frac 12w^2$ on $\R^2$.} Then, by Proposition \ref{pde_opt}, if $u$ is a global minimizer of $(\mathcal{P})$, it satisfies $2\Pi^+\nabla u=\Pi_0\nabla\Phi(u)$ a.e., i.e.
$$\partial_1u_2+\partial_2u_1=w(u)\in L^2(\Omega)\quad \mbox{and}\quad\partial_1u_1=\partial_2u_2=0\quad\text{a.e. in }\Omega .$$ 
In particular, $u_2$ only depends on $x_1$ and $u_1$ only depends on $x_2$. Thanks to Lemma~\ref{BC_unif}, since $H^1(\T)$ is embedded in $\mathcal{C}^0(\T)$, we know that $u_1(R_n,\cdot)$ converges uniformly to $a$ for a sequence $R_n\to\infty$ and thus, $u_1\equiv a$. This implies that $u$ is one-dimensional and the uniqueness property follows from Proposition~\ref{Exist_profile1D}.

\medskip

\nd {\it Case 2: $W\geq \frac 12w^2$ on $\R^2$.} We prove that if $u$ is a global minimizer in $(\mathcal P)$, 
i.e., $E(u)=c_W(u^-,u^+)$, then $u$ is also a global minimizer for the energy 
$v\mapsto \int_\Omega \frac 12 |\nabla v|^2+\frac 12w^2(v).$ Indeed,
\begin{align*}
\int_\Omega\frac 12 |\nabla u|^2+  \frac 12 w^2(u)&\le E(u)=c_W(u^-,u^+)
\le \mathrm{geod}^a_W(u^-,u^+)=\mathrm{geod}^a_{w^2/2}(u^-,u^+)=c_{\frac 12 w^2}(u^-,u^+),
\end{align*}
where the last two last equalities follow from $W=\frac12w^2$ on $[u^-,u^+]$ and Proposition \ref{minimal} 
(as $\Phi$ is an entropy satisfying the saturation condition \eqref{satur_cond} for the potential $\frac 12w^2$). 
Therefore, the conclusion follows by Case 1.
\end{proof}
\begin{proof}[Proof of Corollary~\ref{symmetry_ag}]
Since $W(z)=(1-|z|^2)^2$ satisfies the growth condition \eqref{growth_intro}, the conclusion follows from Theorem~\ref{symmetry_wave}. 
\end{proof}
When $w$ is harmonic, we have a similar rigidity result:
\begin{theorem}\label{symmetry_harmonic} 
Let $W:\R^2\to\R_+$ be a continuous potential and $u^\pm=(a,u^\pm_2)\in S_a$ be two wells of $W$ for some $a\in\R$. Assume that $W\geq \frac 12w^2$ on $\R^2$ and ($W=\frac 12w^2$ and $w>0$ on $(u^-,u^+)$) for some harmonic function $w\in\mathcal{C}^2(\R^2,\R)$. If $u$ is a global minimizer of $(\mathcal{P})$ such that either $u\in L^\infty$ or $|w|$ satisfies the growth condition \eqref{growth_intro} then $u$ is one-dimensional, i.e. $u(x)=g(x_1)$ a.e.\ with $g:\R\to\R^2_a$ being, up to a translation in the $x_1$-variable, the unique one-dimensional transition layer given by Proposition~\ref{Exist_profile1D}.
\end{theorem}
\begin{remark}\label{rotation_invariant}
The advantage of the Laplace operator over the wave operator consists in being rotation invariant. Consequently, if $w$ is harmonic, then Theorem~\ref{symmetry_harmonic} also applies in an infinite cylinder in any direction $\nu\in\S^1$ as explain at page~\pageref{pagina}, in particular, if $W$ is a multi-well potential that is positive on the segment relying two wells.
\end{remark}
\begin{proof}[Proof of Theorem~\ref{symmetry_harmonic}]
As in the previous proof, by Lemma \ref{lem_entropy_harmonic}, there exists an entropy $\Phi\in\mathcal{C}^3(\R^2, \R^2)$ associated to the potential $\frac 12w^2$ satisfying the saturation condition \eqref{satur_cond} for the potential $\frac 12w^2$, together with \eqref{ansatz_harmonic} (with $\mp=-$ and $\alpha\in\mathcal{C}^2(\R^2)$). By the same argument explained in Case 2 of the proof of Theorem \ref{symmetry_wave}, we can assume $W=\frac 12w^2$ on $\R^2$. 

By Proposition \ref{pde_opt}, if $u$ is a global minimizer of $(\mathcal{P})$, it satisfies $2\Pi^-\nabla u^T=\Pi_0\nabla\Phi(u)$ a.e., i.e. $u$ is a solution of the following first order quasilinear PDE system:
\begin{equation}\label{defectPDE}
\begin{cases}
\partial_1u_1+\partial_2u_2&=0\\
-\partial_2 u_1+\partial_1u_2&=w(u)
\end{cases}
\quad\text{a.e. in }\Omega .
\end{equation} 
Since either $u\in L^\infty$ or $|w|$ satisfies \eqref{growth_intro}, i.e. $|w|(z)\le C\exp(\beta |z|^2)$ for all $z\in\R^2$ and for some $C,\beta>0$, we have by the Moser-Trudinger inequality that $w(u)\in L^p_{loc}(\Omega)$ for every $p\in (1,+\infty)$. By \eqref{defectPDE}, we write $u=\nabla^\perp \varphi=(-\partial_2\varphi, \partial_1 \varphi)$ with $-\Delta \varphi=w(u)$ and we deduce by elliptic regularity that $u=\nabla^\perp \varphi\in W^{1,p}_{loc}(\Omega,\R^2)$ for every $p\in (1,+\infty)$. Thus, we have $u\in \mathcal{C}^{0,\theta}\subset L^\infty_{loc}$ with $\theta>0$. By the chain rule applied to the composition $w\circ u$ with $w\in\mathcal{C}^1$ ($w$ is harmonic, thus smooth) and $u\in L^\infty_{loc}\cap H^1_{loc}(\Omega)$, we can compute the derivative $\partial_2$ of the second equation in \eqref{defectPDE} in the distributional sense:
$$-\partial_{22}u_1+\partial_{12}u_2= \partial_{1}w(u)\; \partial_{2}u_1+\partial_2w(u)\; \partial_2u_2.$$
Since $\nabla\cdot u=0$, one has $\partial_2u_2=-\partial_1 u_1$ yielding $\partial_{12}u_2=-\partial_{11}u_1$ in the distribution sense, and thus,
$$-\partial_{22}u_1-\partial_{11}u_1= \partial_{1}w(u)\; \partial_{2}u_1-\partial_2w(u)\; \partial_1u_1.$$
Consequently, $u_1$ solves the following elliptic semi-linear equation
 \begin{equation}\label{equation_u1}
-\Delta u_1-\nabla^\perp w(u)\cdot \nabla u_1=0.
\end{equation}
In particular, since $w$ is smooth, we deduce that $u_1\in\mathcal{C}^{2,\theta}$ by a classical boot-strap argument for elliptic PDE's. Using the classical maximum principle and the boundary condition \eqref{BC}, we shall prove that $u_1$ is constant. Indeed, Lemma~\ref{BC_unif} yields two sequences $(R^\pm_n)_{n\ge 1}\to \pm \infty$ such that
$$u_1(R^\pm_n, \cdot)\to a\quad\text{uniformly on $\T$ when }n\to\infty.$$
Take $\varepsilon>0$ and $n$ large enough to have $|u_1(R_n^\pm,x_2)-a|\leq\varepsilon$ for all $x_2\in\T$. Applying the maximum principle to the elliptic equation \eqref{equation_u1} on the domain $[R_n^-,R_n^+]\times\T$, one gets
\[
|u_1(x)-a|\leq\varepsilon\quad\text{for all }x\in [R_n^-,R_n^+]\times\T.
\]
Since this can be done for arbitrary small values of $\varepsilon >0$ and large values of $n$, one has actually $u_1\equiv a$ and  $\nabla\cdot u=\partial_2u_2=0$. Thus, $u$ depends on $x_1$ and the uniqueness of $u$ (up to translation) follows from Proposition~\ref{Exist_profile1D}.
\end{proof}
}

\begin{example}
\label{exa:separ}
An elementary example of potential for which Theorem~\ref{symmetry_harmonic} applies is given by $W(z)=\frac 12(z_1z_2)^2$. In this case, the set $\{W=0\}$ is the union of the two axis $\{z_1=0\}$ and $\{z_2=0\}$. For two wells $u^+,u^-\neq 0$ which are not on the same axis, the transition axis $\nu:=(u^+-u^-)^\perp$ (which plays the role of $e_1$ in the preceding computations) can be any vector different than the two axis $e_1$ and $e_2$ (see the last paragraph in Section \ref{sym_intro}). Theorem~\ref{symmetry_harmonic} asserts that, with a periodicity condition with respect to the second variable in the basis $(\nu, \nu^\perp)$, that is $x\cdot \nu^\perp$, one has $1D$ symmetry of global minimizers of the energy for the transition between $u^-$ and $u^+$. For two wells lying on the same axis, e.g. $u^\pm=(0,u^\pm_2)\in\{z_1=0\}$, the minimization problem $(\mathcal P)$ has no solution. Indeed, if $u(x_1,x_2)=(0,\varphi(x_1))$ with $\varphi(\pm\infty)=u^\pm_2$, then the energy of $u$ writes
$$E(u)=\frac 12\int_\R |\varphi'(t)|^2\diff t,$$
so that the infimum over all admissible $\varphi$ is $0$. Of course, this infimum is not achieved if $u^-\neq u^+$ since any zero-energy configuration should be constant.

Note that $W(z)=\frac12(z_1z_2)^2$ can also be seen within the framework of Theorem~\ref{symmetry_wave} since $w(z)=z_1z_2$ solves the wave equation as $\partial_{11}(z_1z_2)=\partial_{22}(z_1z_2)=0$. However, as we noticed in Remark~\ref{rotation_invariant}, Theorem~\ref{symmetry_harmonic} applies for any rotation of $W$ contrary to Theorem~\ref{symmetry_wave}: for example, the rotation of angle $\frac\Pi 4$ of $w$ leads to the potential $\tilde{w}(z)= \frac 12 (z_1^2-z_2^2)$ that is still harmonic but not solution of the wave equation.
\end{example}

\begin{proof}[Proof of Theorem~\ref{thm:harm_wave}]
It is a direct consequence of Theorems \ref{symmetry_wave} and \ref{symmetry_harmonic}.
\end{proof}

\paragraph{The case of potentials satisfying the Tricomi equation. Proof of Theorem \ref{thm:tricomi}} \label{pag_tricomi}
The above results can be extended to potentials $W(z)=\frac 12w^2 (z)$ where $w\in\mathcal{C}^2(\R^2,\R)$ satisfies the Tricomi equation 
\eqref{tricomi_eq} for a continuous function $f:\R\to \R$ with $|f|\leq 1$ in $\R$. 
The idea is to construct a map $\Phi\in\mathcal{C}^1(\R^2,\R^2)$ and a scalar function $\alpha\in\mathcal{C}^1(\R^2,\R)$ such that
\begin{equation*}
\Pi_0\nabla\Phi(z)=\nabla\Phi (z)+\alpha(z)I_2=\left(
\begin{matrix} 
0& w(z)\\
f(z_1) w(z)&0
\end{matrix}
\right)
\quad\text{for all }z\in\R^2.
\end{equation*}
As before, by Poincar\'e's lemma, one checks that the existence of $\Phi$ and $\alpha$ are equivalent with the Tricomi equation in 
\eqref{tricomi_eq} (as a consequence, $-\nabla \alpha=(f(z_1)\partial_2 w, \partial_1 w)$ in $\R^2$). 

\medskip

\nd {\it Step 1. $\Phi$ is an entropy associated to the potential $W=\frac 12 w^2$}. Indeed, if $u\in \mathcal{C}^\infty(\Omega,\R^2)$,
$$
\nabla\cdot [\Phi(u)]+\alpha (u)\nabla\cdot u=[\nabla\Phi(u)+\alpha(u)\mathrm{Id}]:\nabla u^T{=}w(u) 
(f(u_1)\partial_2u_1+\partial_1u_2).
$$
In particular, if $\nabla\cdot u=0$, one deduces that $\|\nabla\cdot [\Phi(u)]\|_{L^1(\Omega)}\leq 2E(u)$ since $|f|\leq 1$; moreover,
\begin{align}
\nonumber
2\, \nabla\cdot [\Phi(u)]&=\big({f(u_1)\partial_2u_1+\partial_1u_2}\big)^2+w^2(u)-\big(w(u)- ({f(u_1)\partial_2u_1+\partial_1u_2})\big)^2\\
\nonumber
&=\big(|\nabla u_2|^2+f(u_1)^2 |\nabla u_1|^2 \big)+ w^2(u)-\big(w(u)- (f(u_1)\partial_2u_1+\partial_1u_2)\big)^2\\
\label{egalita20}
&\quad - \big(\partial_2u_2-f(u_1)\partial_1u_1\big)^2-2f(u_1)\big(\partial_1u_1\partial_2u_2-\partial_1u_2\partial_2u_1\big).
\end{align}
Denoting by $F$ the antiderivative of $f$ such that \(F(0)=0\) and using that $\nabla\cdot u=0$,  we compute as in the proof of Proposition \ref{antisymgradient}:
\begin{align*}
I_{R^-,R^+}:=&\Big|\int_{(R^-, R^+)\times\T} f(u_1)\big(\partial_1u_1\partial_2u_2-\partial_1u_2\partial_2u_1\big)\, \diff x\Big|\\
=&\Big|\int_{(R^-, R^+)\times\T}\Big(\partial_1(F(u_1))\,\partial_2u_2-\partial_1u_2\,\partial_2( F(u_1))\Big)\, \diff x\Big|\\
=&\Big|\int_{\T} \Bigl(F(u_1(R^+, x_2))\partial_1u_1(R^+, x_2)-F(u_1(R^-, x_2))\partial_1u_1(R^-, x_2) \Bigr) \diff x_2\Big|, \quad R^-<R^+.
\end{align*}
If $u\in \mathcal{C}^\infty\cap L^\infty \cap \dot{H}^1_{div}(\Omega,\R^2)$ and \(E(u)<+\infty\), then by Remark \ref{rem:BC} there exist two sequences $(R_n^+)_{n\ge 1}$ and $(R_n^-)_{n\ge 1}$ such that $R_n^\pm\to\pm\infty$ and 
$\|u_1(R_n^\pm,\cdot)\|_{L^2(\T^{d-1})}\|\partial_1 u_1(R_n^\pm,\cdot)\|_{L^2(\T^{d-1})}\to 0$ as $n\to +\infty$. Therefore, since $|F(u_1)|\le |u_1|$, we deduce $I_{R_n^-,R_n^+}\to 0$. By \eqref{egalita20}, we obtain 
\begin{equation}
\begin{split}
\label{entropie_tricomi_satur}
\int_\Omega \nabla\cdot [\Phi(u)]=E(u)- \frac 12 \int_\Omega (1-f(u_1)^2)|\nabla u_1|^2
- \frac 12 \int_\Omega\big(w(u)- (f(u_1)\partial_2u_1+\partial_1u_2)\big)^2\\
- \frac 12 \int_\Omega\big(\partial_2u_2-f(u_1)\partial_1u_1\big)^2
\end{split}
\end{equation}
and we conclude $\int_\Omega \nabla\cdot [\Phi(u)]\, dx\leq E(u)$, i.e. $\Phi$ is an entropy.

\medskip

\nd {\it Step 2. The saturation condition \eqref{satur_cond} for two wells $u^\pm=(a,u^\pm_2)\in S_a$ provided that $w\geq 0$ on $[u^-,u^+]$.} This follows by the computation \eqref{eq2018}.

\medskip

\nd {\it Step 3. Symmetry of a global minimizer $u$ of $(\mathcal P)$ provided that (either $u\in L^\infty$ or $|\nabla w|$ satisfies the growth condition \eqref{growth_intro}).}
By \eqref{entropie_tricomi_satur} and Lemma \ref{gauss-green}, for every $u\in \mathcal{C}^\infty\cap L^\infty \cap \dot{H}^1_{div}(\Omega,\R^2)$ with \(E(u)<+\infty\) and \(\overline{u}(\pm\infty)=u^\pm\), one has
\begin{equation}
\begin{split}
\label{tricomi_sat}
\Phi_1(u^+)-\Phi_1(u^-)=E(u)- \frac 12 \int_\Omega (1-f(u_1)^2)|\nabla u_1|^2
\qquad\qquad\qquad\qquad\qquad\qquad\qquad\\
 - \frac 12 \int_\Omega\big(w(u)- (f(u_1)\partial_2u_1+\partial_1u_2)\big)^2
- \frac 12 \int_\Omega\big(\partial_2u_2-f(u_1)\partial_1u_1\big)^2.
\end{split}
\end{equation}
Since each term of the RHS is controlled by the energy density $\frac 12(|\nabla u|^2+w^2(u))$ and since the integrands depend continuously on \(u\), we deduce by Lemma~\ref{lavrentiev} that the equality \eqref{tricomi_sat} still holds for all $u\in H^1_{div}(\Omega,\R^d)$ with $E(u)<+\infty$ and $\overline{u}(\pm\infty)=u^\pm$, without assuming that $u$ is smooth, but only that $u\in L^\infty$ or $W\in\mathcal{C}^2$ satisfies \eqref{growth_W}\footnote{Note that if $|\nabla w|$ satisfies the growth condition \eqref{growth_intro}, then $|w|$ also satisfies the growth condition \eqref{growth_intro}.}.
In particular, if $u$ is a global minimizer of $(\mathcal{P})$, then
\begin{multline*}
\mathrm{geod}_{W}^a(u^-,u^+)\overset{\text{Step 2}}{=}\Phi_1(u^+)-\Phi_1(u^-)
\overset{\eqref{tricomi_sat}}{\le} E(u)\overset{(\mathcal{P})}{=}c_{W}(u^-,u^+)\overset{\text{Prop.~\ref{minimal}}}{=}\mathrm{geod}_{W}^a(u^-,u^+).
\end{multline*}
Therefore, the above inequality, based on \eqref{tricomi_sat}, is an equality which means that the three integrals in \eqref{tricomi_sat} vanish, that is
\begin{align}
\label{lip2}
f(u_1(x))^2=1 \quad \textrm{or}\quad \nabla u_1(x)=0
&\quad \textrm{a.e. in } \Omega,\\
\label{leip}
f(u_1)\partial_2u_1+\partial_1u_2=w(u) \quad \textrm{and}\quad \partial_2u_2=f(u_1)\partial_1u_1 &\quad \textrm{a.e. in } \Omega.
\end{align}
Note that by Remark~\ref{rem_reg}, we have that $u\in W^{2,q}_{loc}(\Omega, \R^2)$ for every $q>1$; in particular, $u\in \mathcal{C}^1(\Omega,\R^2)$ so that \eqref{lip2} and \eqref{leip} hold for every $x\in \Omega$.
Let \(F:\R\to\R\) be an antiderivative of \(f\); we shall prove that under the above conditions, we have
\begin{equation}\label{eq_inter}
-\Delta [F(u_1)]+f(u_1)\nabla^\perp w(u) \cdot\nabla[F(u_1)]=0 \quad\text{distributionally in \(\Omega\).}
\end{equation}
We first note that each term in the equation \eqref{eq_inter} (and in the following computations leading to it) actually belongs to \(L_{loc}^q(\Omega)\) for every \(q>1\) since \(F\in\mathcal{C}^2\), \(w\in\mathcal{C}^2\) and \(u\in W^{2,q}_{loc}\). 
Next, by \eqref{lip2}, we have
\begin{equation}
\label{lip2cor}
\nabla u_1=(f(u_1))^2\,\nabla u_1=f(u_1)\nabla [F(u_1)]
\end{equation}
and from \eqref{leip}, we obtain
\[
\begin{cases}
\partial_1 [F(u_1)]=f(u_1)\partial_1u_1=\partial_2u_2
\\
\partial_2[F(u_1)]=f(u_1)\partial_2u_1=-\partial_1u_2+w(u).
\end{cases}
\]
Hence,
\[
\Delta [F(u_1)]=\partial_2[w(u)]=\partial_1w(u)\partial_2u_1+\partial_2w(u)\partial_2u_2
\]
and by the identity \(\partial_2u_2=-\partial_1u_1\) and \eqref{lip2cor}, we find
\[
\Delta [F(u_1)]=\nabla^\perp w(u)\cdot\nabla u_1=f(u_1)\nabla^\perp w(u)\cdot\nabla [F(u_1)]
\]
which is the desired equation \eqref{eq_inter}. Repeating the argument in the proof of Theorem \ref{symmetry_harmonic}, we deduce from \eqref{eq_inter} that $F(u_1)\equiv F(a)$ in $\Omega$ yielding $u_1\equiv a$ in $\Omega$ by \eqref{lip2cor}; by the divergence constraint, $u_2$ depends only on $x_1$.

\begin{proof}[Proof of Theorem~\ref{thm:tricomi}]
It is a direct consequence of the above arguments.
\end{proof}

\begin{remark} Theorem \ref{thm:tricomi} leads to a large class of potentials $W(z)=\frac 12w^2 (z)$ for which the symmetry result holds true (see e.g. Remark \ref{rem:tric} where $f$ is constant in the Tricomi equation \eqref{tricomi_eq}). An example with a nonconstant function $f$ in \eqref{tricomi_eq} is given by
$$w(z_1,z_2)=(1+\sin^2\frac{z_1}2)\cos z_2 \, \textrm{ and } \, f(z_1)=\frac{\cos z_1}{2(1+\sin^2\frac{z_1}2)}, \, |f|\leq 1,$$ 
for which Theorem \ref{thm:tricomi} applies for any two wells $u^-=(z_1,\pi (k-\frac12))$ and 
$u^+=(z_1, \pi (k+\frac12))$ with $k\in \Z$ and $z_1\in \R$.
\end{remark}

\subsection{One-dimensional symmetry in higher dimension\label{higherdimension}. Proof of Theorems~\ref{thm:rigid_sym} and \ref{thm:finite_wells}}
We start by investigating the existence of entropies in any dimension $d\geq 2$ by the three sufficient conditions ($\mathcal{E}_{strg}$), ($\mathcal{E}_{sym}$) or ($\mathcal{E}_{asym}$) in \eqref{strong_punct}, \eqref{ent_sym} or \eqref{ent_asym}. Obviously, the question of existence of an entropy depends on the potential; in fact, our aim is rather to find potentials $W$ with pairs of zeros $(u^-,u^+)$ for which one has existence of an entropy $\Phi$ satisfying the saturation condition. 
We shall see in particular that the condition ($\mathcal{E}_{asym}$) for entropies $\Phi$ with antisymmetric Jacobians (analogue to the case of harmonic potentials $W$ in $2D$) is too restrictive in dimension $d\geq 3$, i.e., only trivial entropies can be found in this case (see Proposition \ref{rigidity_antisym}). The criterium ($\mathcal{E}_{sym}$) (where the entropy $\Phi$ has symmetric Jacobian), analogue to potentials that are solutions to the wave equation in $2D$, provides nontrivial entropies, in particular, corresponding to an extension of the Ginzburg-Landau potential in dimension $d\geq 3$ (see Theorem~\ref{rigidity_sym}). We are also able to handle a nontrivial class of potentials with a finite number of wells (see Theorem~\ref{1Dgeometry}) by use of the first criterium ($\mathcal{E}_{strg}$) in dimension $d\geq 2$.

\paragraph{Strong punctual condition ($\mathcal{E}_{strg}$).}
We look for $\mathcal{C}^1$ maps $\Phi:\R^d\to\R^d$ satisfying the saturation condition \eqref{satur_cond} and the punctual estimate \eqref{strong_punct}, i.e., 
$|\Pi_0\nabla\Phi|^2\leq 2W$ in $\R^d$. 
We are able to construct such an entropy in the situation when the following holds: 
\begin{equation}
\label{strong}
\mathrm{geod}_{W}^a(u^-,u^+)=\mathrm{geod}_{W}(u^-,u^+), \quad \forall u^{\pm}\in S_a,
\end{equation}
where $\mathrm{geod}_W$ was defined in \eqref{geod_W_gen} (it corresponds to the geodesic (pseudo-)distance between $u^-$ and $u^+$ in $\R^d$ -- and not $\R^d_a$ as in the definition of $\mathrm{geod}_W^a$ in \eqref{geod_H1} -- endowed with the (pseudo-)metric $2Wg_0$).

\begin{theorem}\label{strongsym}
If $W:\R^d\to\R_+$ is a continuous potential satisfying the growth condition \eqref{growth_W}, if $a\in\R$, $u^-,u^+\in S_a$ and if $\mathrm{geod}_{W}^a(u^-,u^+)=\mathrm{geod}_{W}(u^-,u^+)$, then any global minimizer $u$ in $(\mathcal P)$ is one-dimensional, i.e. $u=g(x_1)$ with $g:\R\to\R^d_a$.
\end{theorem}
\begin{proof}
We first set \(\varphi (z):=\mathrm{geod}_{W}(u^-,z)\) for every \(z\in\R^d\). It is easy to see that $\varphi\in\mathrm{Lip}_{loc}(\R^d,\R)$ and $|\nabla\varphi|\le\sqrt{2W}$ a.e. Since \(\varphi\) does not need to belong to \(\mathcal{C}^1\) class \footnote{If $\varphi\in \mathcal{C}^1$, then the corresponding entropy is $\Phi=\varphi e_1$ which 
satisfies \eqref{strong_punct} and the saturation condition \eqref{satur_cond}.}, we introduce a standard mollifying kernel \(\rho^\ell= {\ell^d}\rho({\cdot}{\ell})\), with \(\rho\in\mathcal{C}^\infty_c(\R^d,\R_+)\) such that \(\int_{\R^d}\rho=1\), and we set
\(\varphi^\ell:=\rho^\ell\ast \varphi\) for each \(\ell\in\N^*\). We observe that, by Jensen's inequality,
\begin{equation}
\label{molEst}
|\nabla\varphi^\ell|^2= |\rho^\ell\ast\nabla\varphi|^2\le\rho^\ell\ast |\nabla\varphi|^2\le 2\rho^\ell\ast W\quad\text{in } \R^d.
\end{equation}
Now, let \(u\in\dot{H}^1_{div}(\Omega,\R^d)\) satisfy \(E(u)<+\infty\), \(\overline{u}(\pm\infty)=u^\pm\) and take the sequence of smooth approximations \((u_k)_{k\in\N^*}\subset\mathcal{C}^\infty\cap L^\infty\cap\dot{H}^1_{div}(\Omega,\R^d)\) provided by Lemma \ref{lavrentiev}. In particular we impose that \(u_k(\pm t)=u^\pm\) for large values of \(t\ge 0\), i.e., \(t\ge T_k\in\R_+\). We now set \(\omega_k=[-T_k,T_k]\times\T^{d-1}\) and we compute
\begin{multline*}
\varphi^\ell(u^+)-\varphi^\ell(u^-)=\int_{\omega_k}\partial_1(\varphi^\ell(u_k)) =\int_{\omega_k}\nabla\varphi^\ell(u_k)\cdot\partial_1 u_k\\
\le \int_{\omega_k}\frac 12|\nabla u_k|^2+\frac 12|\nabla\varphi^\ell(u_k)|^2-\frac 12\sum_{j=2}^d |\partial_j u_k|^2
\le \int_{\omega_k}\frac 12|\nabla u_k|^2+(\rho^\ell\ast W)(u_k)-\frac 12\sum_{j=2}^d|\partial_j u_k|^2.
\end{multline*}
We now observe that when \(\ell\to\infty\), \(\varphi^\ell(u^+)-\varphi^\ell(u^-)\to \varphi(u^+)-\varphi(u^-)=\mathrm{geod}_W(u^-,u^+)\) and \(\int_{\omega_k}(\rho^\ell\ast W)(u_k)\to \int_{\omega_k}W(u_k)\) since \((\rho^\ell\ast W)_\ell\) tends to \(W\) uniformly on compact sets and \(u_k\in L^\infty\) for each \(k\). Thus, passing to the limit \(\ell\to\infty\) in the preceding estimates and rearranging the terms, we obtain\footnote{Note that \(\nabla u_k\) is compactly supported in \(\omega_k\) so that integrating on \(\Omega\) or \(\omega_k\) is the same.}
\[
\mathrm{geod}_W(u^-,u^+)+\frac 12\int_{\Omega}\sum_{j=2}^d|\partial_j u_k|^2\le E(u_k),
\]
which in the limit \(k\to\infty\) yields
\[
\mathrm{geod}_W(u^-,u^+)+\frac 12\int_{\Omega}\sum_{j=2}^d|\partial_j u|^2\le E(u).
\]
Since \(\mathrm{geod}_W(u^-,u^+)=\mathrm{geod}_{W}^a(u^-,u^+)\) is the infimum of the energy over 1D admissible maps (see Proposition \ref{inf_1D}), in particular, 
$\mathrm{geod}_W(u^-,u^+)\geq E(u)$ for any global minimizer $u$ in $(\mathcal{P})$, 
we immediately deduce that $u$ only depends on $x_1$.
\end{proof}

We now investigate the existence of potentials $W$ with a finite number of wells for which Theorem~\ref{strongsym} applies. The simplest way to guarantee \eqref{strong} is to set $W=\frac 12w^2$ for some $w$ such that the line segment $[x,y]$ minimizes the $\mathcal{L}_w$-length between any two points 
$x, y\in \{w=0\}$. Here, we denoted for every continuous function $w:\R^d\to \R_+$ and for all Lipschitz curve $\gamma:[-1,1]\to \R^d$ (and not restricted to $\R^d_a$), the length:
\[
\mathcal{L}_w(\gamma)=\int_{-1}^1 w(\gamma(t)) |\dot{\gamma}(t)|\diff t.
\]
The existence of appropriate weight functions $w$ is given by:
\begin{theorem}\label{1Dgeometry}
Let $X=\{x_0,\dots,x_d\}$ be an affine basis of $\R^d$ and let $\delta$ be a pseudo-metric over $X$, that is $\delta\in\Delta (X)$, where $\Delta (X)$ is defined by 
\footnote{\label{foot11} Recall that a pseudo-metric $\delta$ can vanish at a point $(x,y)$ for some $x\neq y$.}
\[
\Delta(X)=\{\delta\in\R_+^{X\times X}\;:\; \forall x,y,z\in X,\, \delta(x,x)=0,\, \delta(x,y)=\delta(y,x),\, \delta(x,y)\leq \delta(x,z)+\delta(z,y)\}.
\]
Then there exists a Lipschitz bounded function $w:\R^d\to\R_+$ such that
\begin{enumerate}
\item
for all $z\in X$, $w(z)=0$, 
\item
$w(z)=\sqrt{2}$ if $|z|$ is large enough,
\item
for all $x,y\in X$ with $\delta(x,y)>0$ and $t\in (0,1)$, $w(t y+(1-t)x)>0$,
\item
for all $x,y\in X$,
\(
\mathcal{L}_w(\overline{xy})=\mathrm{geod}_{w^2/2}(x,y)=\delta(x,y),
\)
where $\overline{xy}$ stands for the line segment between $x$ and $y$, parametrized by $\overline{xy}(t)=t y+(1-t)x$ for all $t\in [0,1]$.
\end{enumerate}
Moreover, if $\delta$ is a metric, i.e. $\delta(x,y)>0$ for all $x,y\in X$ with $x\neq y$, then $w$ can be chosen in such a way that $w>0$ on $\R^d\setminus X$.

Setting $W=\frac 12w^2$, for every two wells $u^\pm\in X$, any $\nu\in\mathbb{S}^{d-1}$ with $\nu\cdot (u^+-u^-)=0$ and any $R\in SO(d)$ such that $R\nu=e_1$, if $u\in \dot{H}_{div}^1(\Omega_R,\R^d)$ is a global minimizer of $E_R$ in \eqref{minimalNu} over divergence-free configurations satisfying the boundary condition \eqref{constraintNu}, then $u$ is one-dimensional, i.e., $u=g(x\cdot \nu)$ where $g\in \dot{H}^1(\R,\R^d)$ with $g(\pm\infty)=u^\pm$ and $E_R(u)=\delta(u^-, u^+)$.
\end{theorem}

\begin{remark}
The assumption that $X$ is an affine basis in $\R^d$ cannot be removed in Theorem~\ref{1Dgeometry}. Indeed, let $X\subset\R^2$ be the set of vertices of a square endowed with the metric $\delta (x,y):=2$ if $x,y\in X$ lie on the same edge of the square and $\delta(x,y):=1$ if $x,y\in X$ lie on the same diagonal of this square (see Figure \ref{square}). 
\begin{figure}[h]
\centering
\begin{tikzpicture}[scale=2]
\draw (0,0) node {$\bullet$} -- (0.5,0) node[below] {$2$} -- (1,0) node {$\bullet$} node[below right] {$x_1$} -- (1,0.5) node[right] {$2$} -- (1,1) node {$\bullet$} node[above right] {$x_2$} -- (0.5,1) node[above] {$2$} -- (0,1) node {$\bullet$} -- (0,0.5) node[left] {$2$} -- cycle;

\draw (0,0) -- (0.8,0.8) node[below] {$0.4$} -- (1,1);
\draw (0,1) -- (0.8,0.2) node[above] {$0.2$} -- (1,0);

\draw (0.2,0.2) node[right] {$0.6$};
\draw (0.2,0.8) node[right] {$0.8$};

\draw (0.5,0.5) node {$\bullet$} node[left] {$x_0$};
\end{tikzpicture}
 \caption{A complete graph of the four vertices of a square and the associated $\delta$-distances\label{square}}
\end{figure}
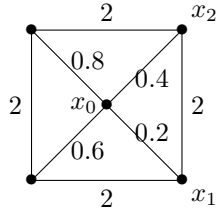
Assume by contradiction that $\delta=\mathrm{geod}_W$ for $W=\frac 12 w^2$ with some weight function $w$ (defined on $\R^2$) such that the line segments $[x,y]$ are minimal for $\mathrm{geod}_W$ for every two vertices $x$ and $y$ of the square. Since the length (in the metric $\mathrm{geod}_W$ over $\R^2$) of the two diagonals is $1$, there exist two vertices $x_1,x_2$ on the same edge such that $\mathrm{geod}_W(x_0,x_i)\le 1/2$ for $i=1,2$, where $x_0$ is the intersection of the two diagonals. We thus have by the triangle inequality $2=\mathrm{geod}_W(x_1,x_2)\le \mathrm{geod}_W(x_0,x_1)+\mathrm{geod}_W(x_0,x_2)\le 1$ which is a contradiction.
\end{remark}
The proof of Theorem~\ref{1Dgeometry} relies on two tools: the decomposition of $\delta$ in terms of extremal  pseudo-metrics (see Lemma~\ref{extremalmetric} in the appendix) and the existence of a calibration $\varphi$ for the line segments $\overline{xy}$, with $x,y\in X$, when the pseudo-metric $\delta$ is extremal and $w=|\nabla\varphi|$ (see Lemma~\ref{extremalcalibration} in the appendix). Here, we just explain how Lemma~\ref{extremalmetric} and Lemma~\ref{extremalcalibration} imply Theorem~\ref{1Dgeometry}.
\begin{proof}[Proof of Theorem~\ref{1Dgeometry}]
We divide the proof into several steps.

\medskip
\noindent\textsc{Step 1: The case of an extremal pseudo-metric $\delta$, i.e.,  $\delta=\delta_Y$ for some $Y\subset X$ with $Y\neq \emptyset$ and $Y\neq X$, where}
\[
\delta_Y(x,y):=
\begin{cases}
0&\text{if }(x,y\in Y)\text{ or }(x,y\notin Y),\\
1&\text{otherwise.}
\end{cases}
\]
Let $ \varphi_Y\in\mathcal{C}_c^\infty(\R^d,\R)$ be a scalar function satisfying all the properties claimed in Lemma~\ref{extremalcalibration} and define the Lipschitz  compactly supported function $w_Y:\R^d\to \R_+$ by 
\[
w_Y(z)=|\nabla\varphi_Y (z)|\quad\text{for all }z\in\R^d.
\] 
We claim that the differential form $\omega=\diff \varphi_Y$ is a calibration for the line segment $\overline{xy}$ for every $x,y\in X$, in the following sense:
\begin{itemize}
\item
for any Lipschitz curve $\gamma:[0,1]\to \R^d$ with $\gamma(0)=x$ and $\gamma(1)=y$, one has
\[
\varphi_Y(y)-\varphi_Y(x)=\int_\gamma\omega=\int_0^1 \nabla\varphi_Y(\gamma(t))\cdot \dot{\gamma}(t)\diff t\leq \int_0^1 w_Y(\gamma(t)) |\dot{\gamma}(t)|\diff t=\mathcal{L}_{w_Y}(\gamma),
\]
\item
the preceding inequality is an equality when $\gamma=\overline{xy}$ with ($x,y\in Y$), ($x,y\in X\setminus Y$) or ($x\in Y$ and $y\in X\setminus Y$), i.e.
\[
\varphi_Y(y)-\varphi_Y(x)=\int_0^1 (\nabla\varphi_Y(ty+(1-t)x)\,,\, y-x)\diff t=\mathcal{L}_{w_Y}(\overline{xy}).
\]
\end{itemize}
This comes from the fact that $\nabla\varphi_Y(ty+(1-t)x)$ and $y-x$ are positively collinear if $x\in Y$ and $y\in X\setminus Y$ and $w_Y=|\nabla\varphi_Y|$; as a consequence, the segment $\overline{xy}$ minimizes the $\mathcal{L}_{w_Y}$-length between any two points $x,y\in X$ and by Lemma ~\ref{extremalcalibration}, 
$\mathcal{L}_{w_Y}(\overline{xy})=|\varphi_Y(y)-\varphi_Y(x)|=\delta(x,y)$ for every $x,y\in X$, i.e. $w=w_Y$ satisfies Point 4 in Theorem~\ref{1Dgeometry}. Points 1 and 3 in Theorem~\ref{1Dgeometry} are a consequence of the properties of $\varphi=\varphi_Y$ in Lemma~\ref{extremalcalibration}.

\medskip
\noindent\textsc{Step 2: The case of a general pseudo-metric $\delta$.} By Lemma~\ref{extremalmetric}, $\delta$ writes 
\[
\delta=\sum_{\emptyset\neq Y\subsetneq X}\lambda_Y \delta_Y
\] 
for some parameters $\lambda_Y\geq 0$.
We set
\[
w=\sum_{\emptyset\neq Y\subsetneq X}\lambda_Yw_Y,
\]
with the $w_Y:\R^d\to\R_+$ defined in Step 1. It is easy to check Points 1, 3 and 4 in Theorem~\ref{1Dgeometry}. 
For instance, Point 4 comes from the fact that for every Lipschitz curve $\gamma:[0,1]\to \R^d$ connecting $x$ to $y$, one has
\[
\mathcal{L}_w(\overline{xy})=\sum_{\emptyset\neq Y\subsetneq X}\lambda_Y {\cal L}_{w_Y}(\overline{xy})=\delta(x,y)\leq \sum_{\emptyset\neq Y\subsetneq X}\lambda_Y {\cal L}_{w_Y}(\gamma)=\mathcal{L}_w(\gamma).
\]
\noindent\textsc{Step 3: Reaching Point 2 and improvement to the case when $\delta$ is a metric.} If $w_0$ is the function given by Step 2, we set $w=w_0+w_1$, where $w_1$ is any Lipschitz function such that
\(
w_1=0\text{ on }G:=\cup_{x,y\in X}[x,y]\text{, } w_1>0\text{ on }\R^d\setminus G\text{ and such that }w_1\equiv\sqrt{2}
\)
if $|z|$ is large enough. The line segments between any two points $x,y\in X$ are still optimal with $\mathcal{L}_w$ instead of $\mathcal{L}_{w_0}$ because of the inequality $\mathcal{L}_w(\gamma)\geq {\cal L}_{w_0}(\gamma)$ for all curves $\gamma$ and the equality $\mathcal{L}_w(\overline{xy})=\mathcal{L}_{w_0}(\overline{xy})$. Hence, the function $w$ satisfies Point 4 in Theorem~\ref{1Dgeometry}; the other points are easy to check.

\medskip
\nd \textsc{Step 4}. The symmetry of global minimizers of $E_R$ is a direct consequence of Theorem \ref{strongsym} and the analysis presented in Section \ref{sym_intro} at the paragraph ``Change of variables under rotation".
\end{proof}

\begin{proof}[Proof of Theorem \ref{thm:finite_wells}] It is a direct consequence of Theorem~\ref{1Dgeometry}.
\end{proof}

\paragraph{Entropies with antisymmetric Jacobians ($\mathcal{E}_{asym}$).} If in dimension $2$, we have constructed entropies $\Phi$ satisfying ($\mathcal{E}_{asym}$) that are holomorphic (see Lemma~\ref{lem_entropy_harmonic}), we will show that in dimension $d\geq 3$, the antisymmetry of $\Pi_0\nabla\Phi$ imposed in the criterium ($\mathcal{E}_{asym}$) is very rigid for maps $\Phi$:
\begin{proposition}\label{rigidity_antisym}
Let $d\geq 3$ and $\Phi=(\Phi^1,\dots ,\Phi^d) :\R^d\to \R^d$ be a locally Lipschitz map such that
\begin{equation}
\label{nablaantisym}
\Pi_0\nabla\Phi(z)\text{ is antisymmetric for a.e. }z\in\R^d.
\end{equation}
Then there exist $c=(c_1,\dots , c_d)\in\R^d$ and a linear antisymmetric application $L=(L^1,\dots,L^d)\in\mathcal{L}(\R^d,\R^d)$ such that for all $i\in \{1,\dots ,d\}$ and $z\in\R^d$,
\begin{equation}
\label{quadratic_entropy}
\Phi^i (z)=\Phi^i(0)+L^i z+\sum_{j=1}^d\left\{ c_j z_jz_i - c_i \frac{|z_j|^2}{2}\right\}.
\end{equation}
\end{proposition} 
\begin{proof} Up to regularizing $\Phi$ by convolution with a smooth mollifying kernel (thus preserving the algebraic constraint), one can assume that $\Phi$ is smooth. Indeed, for some mollifying kernel $(\rho_\varepsilon)_{\varepsilon>0}$, assume that $\Phi_\varepsilon:=\rho_\varepsilon\ast \Phi$ writes in the preceding form: $\Phi_\varepsilon=\Phi_\varepsilon(0)+L_\varepsilon+q_{\varepsilon}$, where $L_\varepsilon$ is the linear part and $q_{\varepsilon}$ is the quadratic part which depends on the parameter $c_\varepsilon\in\R^d$. Since $(\Phi_\varepsilon)_\varepsilon$ is locally bounded in $W^{1,\infty}$, we know that $(\Phi_\varepsilon(0))_\eps$ and $(L_\varepsilon=\nabla\Phi_\varepsilon(0))_\eps$ are bounded. Thus, $q_\varepsilon=\Phi_\varepsilon-\Phi_\varepsilon(0)-L_\varepsilon$ is also bounded in the space of quadratic forms, which means that the parameter $c_\varepsilon$ is bounded. Thus, there is a subsequence $\varepsilon_i\to 0$ such that $c_{\varepsilon_i}\to c\in\R^d$ and $L_{\varepsilon_i}\to L\in\mathcal{L}(\R^d,\R^d)$ as $i\to \infty$. In the limit, one gets the identity $\Phi=\Phi(0)+L+q$ where $q$ is the quadratic form given by the last term in the RHS of \eqref{quadratic_entropy}.

Moreover, up to replacing $\Phi $ by $\Phi -\Psi$ with $\Psi (z)=\Phi(0)+\nabla\Phi(0)z$, one can assume that $\Phi(0)=0$ and $\nabla\Phi (0)=0$. For the sake of simplicity, we shall write $f_i=\partial_i f $ for the partial derivative w.r.t. $z_i$ of some scalar or vector function $f$ defined on $\R^d$. In particular, writing $\Phi=(\Phi^1,\dots , \Phi^d)$, we have the notation
$$\Phi^i_j=\partial_j\Phi^i \quad \text{for all }i,j\in \{1,\dots,d\}.$$
Now, the algebraic constraint \eqref{nablaantisym} rewrites
\begin{equation*}
\label{cond}
\begin{cases}
\Phi^1_1=\dots=\Phi^d_d=:\alpha\in L^\infty_{loc},\\
\Phi_j^i=-\Phi^j_i\quad\text{for all }i\neq j.
\end{cases}
\end{equation*}
In particular, if $i,j,k\in\{1,\dots ,d\}$ are three distinct indices, then by the Schwarz theorem,
$$\partial_j\Phi_{k}^i=-\partial_j\Phi^k_{i}=\Phi^j_{ki}=-\Phi^i_{jk}\quad \text{and thus,}\quad \Phi_{jk}^i=0.$$
In particular, $\Phi^i_j$ only depends on $z_i$ and $z_j$. Therefore, for the purpose of notation, we afford to write
$$\Phi^i_j (z)=\Phi^i_j (z_i,z_j)\quad\text{for }i\neq j.$$
Then, for every $i$ and $j$ such that $i\neq j$, one has
$$\Phi^i_{jj}=-\Phi^j_{ij}=-\alpha_i \quad\text{and}\quad \Phi^i_{ji}=\alpha_j\ .$$
In particular, $\alpha_i$ only depends on $z_i$ and $z_j$ for all $j\neq i$. Since $d\geq 3$, this means that $\alpha_i$ depends on $z_i$ only:
$$\alpha_i (z)=\alpha_i (z_i) .$$ 
Now, for every $i$ and $j$ with $i\neq j$, one has
$$-\alpha_{ii}=(\Phi^i_{jj})_i=\Phi^i_{ijj}=(\Phi^i_{ij})_j=\alpha_{jj}\ .$$
In particular, $-\alpha_{ii}=\alpha_{jj}$ for all $i\neq j$ which implies for $k\notin\{i,j\}$, $-\alpha_{ii}=\alpha_{kk}=-\alpha_{jj}=\alpha_{ii}$, that is $\alpha_{ii}=0$. As $\alpha_i$ depends only on $z_i$, we deduce that $\alpha_i$ is constant for all $i\in\{1,\dots ,d\}$, i.e.,
$$\alpha_i\equiv :c_i\in\R .$$
Since $\Phi^i_j (z)=\Phi^i_j (z_i,z_j)$ with $(\Phi^i_j)_i=\alpha_j=c_j$ and $(\Phi^i_{j})_j=-\alpha_i=-c_i$, and since $\nabla\Phi (0)=0$, one has
$$\begin{cases}
\Phi^i_j (z)=c_jz_i-c_iz_j \ \text{ when }i\neq j\ ,&\\
\Phi^1_1 (z)=\dots=\Phi^d_d(z)=\alpha (z)=\sum_{i}c_iz_i\ ,
\end{cases}$$
and the proposition follows.
\end{proof}

We will show that in dimension $d\geq 3$, the rigidity \eqref{nablaantisym} imposed on entropies within the criterium ($\mathcal{E}_{asym}$) cannot be compatible with the saturation condition \eqref{satur_cond} for two distinct wells $u^\pm\in S_a$, $a\in\R$ if the geodesic cost $\mathrm{geod}_W^a(u^-, u^+)>0$.
\begin{corollary}\label{cor_pde_opt}
Let $W:\R^d\to\R_+$ be a continuous potential and $u^\pm\in S_a$, with $a\in\R$. Assume that there exists an entropy $ \Phi\in\mathcal{C}^1(\R^d,\R^d)$ satisfying ($\mathcal{E}_{asym}$) and the saturation condition \eqref{satur_cond}. If $u\in\dot{H}_{div}^1(\Omega,\R^d)$ is a global minimizer in ($\mathcal P$) such that either ($u\in L^\infty(\Omega,\R^d)$ and $W\in\mathcal{C}^2(\R^d,\R_+)$) or $W$ satisfies the growth condition \eqref{growth_W} then the $\T^{d-1}$-average $\overline{u}$ of $u$ is constant; in particular, $u^-=u^+$.
\end{corollary}
\begin{proof}
By Proposition~\ref{rigidity_antisym}, there exist an antisymmetric matrix $A\in\R^{d\times d}$ and $c\in\R^d$ such that $\Pi_0\nabla\Phi (z)=z\otimes c-c\otimes z+A$ for a.e. $z\in\R^d$ and we deduce by Proposition~\ref{pde_opt} that
\[
2\Pi^-\nabla u^{T}=u\otimes c-c\otimes u+A\quad\text{a.e. in $\Omega$.}
\]
By integrating over $x'\in\T^{d-1}$, we obtain the system $2\Pi^-\nabla \overline{u}^T=\overline{u}\otimes c-c\otimes \overline{u}+A$ which rewrites
\begin{equation}
\label{edpA}
\begin{cases}
\frac{\diff}{\diff t}\varphi(t)=ac'-c_1\varphi(t)+A'_1\\
c'\otimes\varphi - \varphi\otimes c'=A'
\end{cases}
\quad\text{for a.e. $t\in\R$,}
\end{equation}
where $\varphi:\R\to\R^{d-1}$, $c'\in\R^{d-1}$ and $A'\in\R^{(d-1)\times (d-1)}$ are determined by $\overline{u}=(a,\varphi)\in\R^d_a$, $c=(c_1,c')\in\R^d$, $A'=(A_{ij})_{i,j\ge 2}$, and $A'_1$ is the first row vector of $A'$.

If $c'=0$, then the only bounded solutions of the ODE $\frac{\diff}{\diff t}\varphi(t)=-c_1\varphi(t)+A'_1$ in $\R$ are the constant solutions; thus, we deduce that $\overline{u}=(a,\varphi)$ is constant.

If $c'\neq 0$, by multiplying the second equation of \eqref{edpA} by $\frac{c'}{|c'|^2}$, we obtain
\begin{equation}
\label{gamma}
\varphi(t)= \Big(\gamma(t) -\frac{A'}{|c'|}\Big)\frac{c'}{|c'|},\quad\text{where }\gamma(t):=\varphi(t)\cdot \frac{c'}{|c'|}\quad \text{for a.e. }t\in\R.
\end{equation}
Moreover, the first equation of \eqref{edpA} yields
\[
\frac{d}{dt}\gamma=-c_1\gamma+b,\quad \text{with}\quad b\in\R.
\]
Again, since the only bounded solution of this ODE in $\R$ are constant, we deduce that $\gamma$ is constant. Then $\varphi$ and $\overline{u}=(a,\varphi)$ are constant as well by \eqref{gamma}.
\end{proof}

\paragraph{Entropies with symmetric Jacobians ($\mathcal{E}_{sym}$).}
This criterium turns out to be more useful than ($\mathcal{E}_{asym}$) in dimension $d\geq 3$, although very restrictive. By Proposition~\ref{diagonal_entropy}, if there exists a map $ \Phi\in{\mathcal C}^1(\R^d,\R^d)$ satisfying the following conditions:
\begin{itemize}
\item
$\nabla\Phi(z)$ is symmetric and satisfies \eqref{33} for all \ $z\in\R^d$,
\item
$|\Pi_0\nabla\Phi(z)|^2\leq 4W(z)$ for all \ $z\in\R^d$,
\item
$\Phi$ satisfies the saturation condition \eqref{satur_cond}, i.e., $\Phi_1(u^+)-\Phi_1(u^-)=\mathrm{geod}_{W}^a(u^-,u^+)$,
\end{itemize}
then one has one-dimensional symmetry of global minimizers of ($\mathcal P$) provided some growth condition on $W$. One can reformulate this result in the following way, where the saturation condition \eqref{satur_cond} is replaced by \eqref{el_nd}.
\begin{proposition}\label{rigidity_sym}
Let $ \Phi\in{\mathcal C}^1(\R^d,\R^d)$ be such that $\nabla\Phi$ is symmetric satisfying \eqref{33} in $\R^d$ and consider the potential $W$ given in 
\eqref{calibW}
and two wells $u^\pm\in\R^d_a\cap\{W=0\}$. If $v\in \dot{H}^1_{ div}(\Omega,\R^d)$ solves the system
\begin{equation}\label{el_nd}
2\,\Pi^+\nabla v=\Pi_0\nabla\Phi(v),
\end{equation}
and if $v(x)=v(x_1)$, $v_1\equiv a$, $v(\pm\infty)=u^\pm$ and $W$ satisfies the growth condition \eqref{growth_W}, then $v$ is a global minimizer of ($\mathcal P$) and any other global minimizer $u$ depends on $x_1$ only.
\end{proposition}
\begin{proof} We first observe that \eqref{calibW} rewrites $ |\Pi_0\nabla\Phi|^2=4W$; hence, by \eqref{calibW} and Proposition \ref{criterium_weak}, we know that $\Phi$ is an entropy. We now show that if $v$ is a $1D$ solution of \eqref{el_nd}, then $\Phi$ satisfies the saturation condition \eqref{satur_cond}. 
 Indeed, note that \eqref{el_nd} reads
\begin{equation*}\label{matrix_el}
2\Pi^+\nabla v=
\left(\begin{matrix}
0&\dot{v}_2&\dots &\dot{v}_d\\
\dot{v}_2&0&\dots&0\\
\vdots&\vdots&&\vdots\\
\dot{v}_d&0&\dots&0
\end{matrix}\right)
=
\Pi_0\nabla\Phi(v)
=
\left(\begin{matrix}
0&&(\partial_j\Phi_i(v))_{j>i}\\
&\ddots&\\
(\partial_j\Phi_i(v))_{j<i}&&0
\end{matrix}\right).
\end{equation*}
Then one has $|\dot{v}|^2=\frac 12 |2\Pi^+\nabla v|^2=\frac 12 |\Pi_0\nabla\Phi(v)|^2\stackrel{\eqref{calibW}}{=}2W(v)$ and since $ \dot{v}=(0,\partial_2\Phi_1(v),\dots,\partial_d\Phi_1(v))$, we compute
\begin{equation*}
\Phi_1(u^+)-\Phi_1(u^-)=\int_\R \nabla\Phi_1(v(t))\cdot \dot{v}(t)\diff t=\int_\R |\dot{v}(t)|^2\diff t=\int_\R \sqrt{2W(v(t))}|\dot{v}(t)|\diff t\stackrel{\eqref{geod_H1}}{\geq} \mathrm{geod}_{W}^a(u^-,u^+).
\end{equation*}
Moreover, the reverse inequality also holds. Indeed, by \eqref{geod_H1}, we can choose a sequence of curves $(\gamma_k)_{k\ge 1}$ in $\mathrm{Lip}([-1,1],\R^d_a)$ such that $\gamma_k(\pm 1)=u^\pm$  and $(L_{W}(\gamma_k))_{k\ge 1}\to \mathrm{geod}_{W}^a(u^-,u^+)$ so that
\[
\Phi_1(u^+)-\Phi_1(u^-)=\int_{-1}^1 \nabla\Phi_1(\gamma_k(t))\cdot \dot{\gamma_k}(t)\diff t\stackrel{\eqref{calibW}}{\leq} \int_{-1}^1 \sqrt{2W(\gamma_k)}|\dot{\gamma_k}|\to \mathrm{geod}_{W}^a(u^-,u^+). 
\]
Thus, the saturation condition $\Phi_1(u^+)-\Phi_1(u^-)=\mathrm{geod}_{W}^a(u^-,u^+)$ follows and, as a by-product, one gets optimality of $v$ since $E(v)=L_W(v)=\Phi_1(u^+)-\Phi_1(u^-)$ is minimal by Proposition \ref{minimal}.
The one-dimensional symmetry of other minimizers $u$ is a consequence of Proposition~\ref{diagonal_entropy}.
\end{proof}

\begin{proof}[Proof of Theorem \ref{thm:rigid_sym}] It is a direct consequence of Proposition~\ref{diagonal_entropy}.
\end{proof}

\medskip

\nd {\it Strategy for constructing entropies}. We now investigate whether \eqref{calibW} provides nontrivial potentials $W$ for which one has one-dimensional symmetry of global minimizers in ($\mathcal P$). We thus look for maps $\Phi\in\mathcal{C}^1(\R^d,\R^d)$ such that for all $z\in\R^d$, $\nabla\Phi(z)$ is symmetric. If so, by the Poincar\'e Lemma, there exists $\Psi\in \mathcal{C}^2(\R^d)$ such that
\[
\Phi(z)=\nabla\Psi (z)\quad\text{for all }z\in\R^d.
\]
In addition, we require  that \eqref{33} holds for $\Phi$, which amounts to imposing 
\eqref{wavePsi} on $\Psi$.
By analogy with the wave equation in $\R^2$, solutions of these equations can be written
$$\Psi(z)=\sum_{\sigma\in\{\pm1\}^d}f_\sigma (\sigma\cdot z) ,$$
where $(f_\sigma)_\sigma$ is a family of scalar functions defined over $\R$ (this form of $\Psi$ follows by an induction argument over the dimension $d$). However, this formula is not so easy to manipulate and we rather use an induction method: the entropy ${\Phi}$ in $\R^d$ will be constructed as an extension of the entropy $\overline{\Phi}$ defined on $\R^{d-1}\sim \{z_d=0\}\cap\R^d$.
More precisely, assume that the map $\overline{\Phi}=\nabla\overline{\Psi}$ (with $\overline{\Psi}\in\mathcal{C}^2(\R^{d-1})$) is an entropy in $\R^{d-1}$ leading to the potential $\overline{W}:=\frac 14|\Pi_0\nabla\overline{\Phi}|^2$. We now look for an entropy $\Phi:\R^d\to\R^d$, of the form $\Phi=\nabla\Psi$ with $\Psi:\R^d\to\R$, such that $\Psi(z_1,\dots,z_{d-1},0)=\overline{\Psi}(z_1,\dots,z_{d-1})$. The function $\Psi$ defined by
\begin{equation}
\label{extensionPsi}
\Psi(z_1,\dots,z_d)=\frac 12 \big(\overline{\Psi}(z_1,\dots,z_{d-2},z_{d-1}+z_d)+\overline{\Psi}(z_1,\dots,z_{d-2},z_{d-1}-z_d)\big)
\end{equation}
is an extension of $\overline{\Psi}$ which solves \eqref{wavePsi} in $\R^d$ provided that $\overline{\Psi}$ solves the same equation \eqref{wavePsi} in dimension $d-1$.

\medskip

\nd {\it Ginzburg-Landau type potential in dimension $d\geq 3$}. We shall build examples of entropies in every dimension $d\ge 3$ by use of the preceding induction method. Let us initialize the induction in dimension $d=2$ with the Ginzburg-Landau potential $W_{d=2}(\cdot):=\frac 12(1-|\cdot|^2)^2$ for which we have the Aviles-Giga entropy
\[
\Phi^{d=2}:=\nabla\Psi_{d=2},\quad\text{where}\quad\Psi_{d=2}(z_1,z_2)=-z_1z_2\left(\frac{z_1^2+z_2^2}{3}-1\right)
\]
and let $(\Psi_d)_{d\ge 2}$ be the unique sequence of scalar functions given inductively by  \eqref{extensionPsi}, where $\Psi=\Psi_d:\R^d\to\R$ and $\overline{\Psi}=\Psi_{d-1}:\R^{d-1}\to\R$ for every $d\ge 3$. Then, an easy computation yields
$$\Psi_d (z)=-z_1z_2\left(\frac{z_1^2+z_2^2}{3}+|z''|^2-1\right)\quad \text{for every } z=(z_1,z_2, z'')\in\R^d, \, z''=(z_3, \dots, z_d).$$
The map $\Phi^d:=\nabla\Psi_d:\R^d\to\R^d$ is an entropy for the following potential\footnote{Note that the growth condition \eqref{growth_W} is valid for $W_d$ only for dimensions $d\leq 4$. }
$$W_d(z):=\frac 14|\Pi_0\nabla\Phi^d(z)|^2=\frac 12(|z|^2-1)^2+2|z''|^2(z_1^2+z_2^2)\quad \text{for every }z\in\R^d.$$

\medskip

\nd {\it Symmetry of global minimizers in the case of the potential $W_d$}. We will follow Proposition \ref{rigidity_sym}. Let us give a detailed study in dimension $d=3$ (the same argument works for $d=4$). The gradient of the entropy $\Phi^3=\nabla \Psi_3$ writes
\begin{equation*}\label{entropy3d}
\nabla\Phi^3 (z)=\nabla^2\Psi_3 (z)=
\left(\begin{matrix}
-2z_1z_2&1-|z|^2&-2z_2z_3\\
1-|z|^2&-2z_1z_2&-2z_1z_3\\
-2z_2z_3&-2z_1z_3&-2z_1z_2
\end{matrix}\right) .
\end{equation*}
We have
\[
\{W_3=0\}=\{z\in \mathbb{S}^2\;:\;z_3=0\quad\text{or}\quad z_1=z_2=0\}=\mathbb{S}^1\cup\{\pm e_3\},
\]
where $\mathbb{S}^2$ is the unit sphere in $\R^3$, $\mathbb{S}^1=\mathbb{S}^2\cap \{z_3=0\}$ and $(e_1,e_2,e_3)$ is the canonical basis of $\R^3$. 

\medskip

\nd {\it Case 1: wells $u^\pm$ in $\mathbb{S}^1$}.
Let $u^\pm$ be two wells in $\mathbb{S}^1$ such that $(u^+-u^-)\cdot e_1=0$:
$$u^\pm = (a,\pm b, 0),$$
where $b>0$ and $a^2+b^2=1$ (see Figure \ref{W=0}). Note that since $W_3$ is invariant by rotation around the axis $\R e_3$, it is not restrictive to take $\nu=e_1$ when considering two wells such that $u^+-u^-$ is orthogonal to $\nu\in\mathbb{S}^1$.
\begin{figure}
\centering
\begin{tikzpicture}[scale=0.7,>=triangle 45,scale=2.5]

   \begin{scope}[canvas is zy plane at x=0]
     \draw  (0,-1) -- (0,1);
   \end{scope}

   \begin{scope}[canvas is zx plane at y=0]
     \draw[thick] (0,0) circle (1cm);
     \draw[dashed] (-1,0) -- (1,0) (0,-1) -- (0,1);
     \draw (0.70710678118,-0.70710678118)--(0.70710678118,0.70710678118);
   \end{scope}
   
  \begin{scope}[canvas is xy plane at z=0]
  \draw (0,1)--(1,0)--(0,-1)--(-1,0)--cycle;
   \end{scope}
 
 \draw (0,1,0) node[above] {$e_3$} node { $\bullet$}; 
 \draw (0,-1,0)node[below] {$-e_3$} node {\large $\bullet$}; 
 \draw (1,0,0) node[right] {$e_2$} node { $\bullet$}; 
  \draw (-1,0,0) node[left] {$-e_2$} node { $\bullet$}; 

  \draw (-0.70710678118,0,0.70710678118) node {$\bullet$} node[below left] {$u^-$}; 
  \draw (0.70710678118,0,0.70710678118) node {$\bullet$} node [below right] {$u^+$}; 
 
 \draw (0,0,1.2) node[below right]  {$e_1$};
 \end{tikzpicture}
\caption{Phase space with the set $\{W_3=0\}=\mathbb{S}^1\cup\{\pm e_3\}$\label{W=0}}
\end{figure}
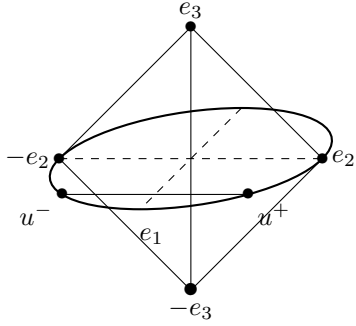
By Proposition \ref{rigidity_sym}, the symmetry of global minimizers in $(\mathcal{P})$ follows from the existence of a $1D$ solution of the system \eqref{el_nd}. Note that a one-dimensional transition $v=(a,v_2(x_1),v_3(x_1))$, with $(v_2,v_3)(\pm\infty)=(\pm b,0)$, satisfies \eqref{el_nd} if and only if
\begin{equation}\label{ode}
(\dot v_2(x_1), \dot v_3(x_1))=\left(b^2-v_2^2(x_1)-v_3^2(x_1)\; ,\; -2v_2(x_1)v_3(x_1)\right)\quad\text{and}\quad a\,v_3(x_1)=0 .
\end{equation}
A solution of this ODE such that $(v_2,v_3)(\pm\infty)=(\pm b,0)$ is given by
\[
(v_2(x_1),v_3(x_1))=(b\, \mathrm{tanh}(b\, x_1),0);
\]
if $a\neq 0$ this is the only solution of \eqref{ode} up to translation (since \eqref{ode} then yields $v_3\equiv 0$), while if $a=0$, i.e. $b=1$, there are also solutions with non vanishing $v_3$ (consider the unique solution of \eqref{ode} such that $(v_2,v_3)(0)=(0,h)$ with $h\in (-1,1)$; it is easy to see that this solution is defined on $\R$ and satisfies $(v_2,v_3)(\pm\infty)=(\pm b,0)$).
Thus, this argument proves both the one-dimensional symmetry of global minimizers in $(\mathcal{P})$ (via Proposition \ref{rigidity_sym}) and  the uniqueness of optimal $1D$ transition layers (up to a translation) when $a\neq 0$.
\medskip

\nd {\it Case 2: wells $u^\pm$ in the set $\{\pm e_2,\pm e_3\}$}. In this case, $a=0$. By symmetry, it is enough to consider transitions from $u^-=e_3$ to $u^+=e_2$ and from $u^-=e_3$ to $u^+=-e_3$. In both cases, \eqref{ode} reads
$$( \dot v_2,\dot v_3)= (1-v_2^2-v_3^2 , -2v_2v_3).$$
For the transition between $e_3$ and $e_2$, it is convenient to use the change of variable $u =(u_1,u_2):=(v_2+v_3,v_2-v_3)$ so that the preceding ODE is equivalent to the decoupled system
$$(\dot u_1, \dot u_2)= (1-u_1^2 \;,\; 1-u_2^2),\quad u (\pm\infty)=(1,\pm 1),$$
whose only solution lies on a straight line (see figure \ref{geodesics3D}), and is given by
$$u_1(x_1)=v_2(x_1)+v_3(x_1)\equiv 1\quad \text{and}\quad u_2(x_1)=v_2(x_1)-v_3(x_1)=\tanh (x_1).$$
As before, this argument proves both the one-dimensional symmetry and uniqueness (up to a translation) of global minimizers in $(\mathcal{P})$.
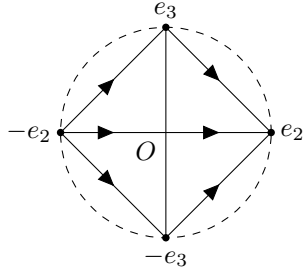
\begin{figure}
\centering
\begin{tikzpicture}[scale=0.7,>=triangle 45,x=2.0cm,y=2.0cm]
\draw (-1,0)--(1,0);
\draw (0,-1)--(0,1);

\draw (1,0)--(0,1)--(-1,0)--(0,-1)--(1,0);

\draw [dashed] (0,0) circle (1);

\draw node[below left] {$O$};
\fill (1,0) circle (2pt) node[right] {$e_2$};
\fill (-1,0) circle (2pt) node[left] {$-e_2$};
\fill (0,1) circle (2pt) node[above] {$e_3$};
\fill (0,-1) circle (2pt) node[below]{$-e_3$};

\draw[->] (.5,0)--+(.01,0);
\draw[->] (-.5,0)--+(.01,0);
\draw[->] (.5,.5)--+(.01,-.01);
\draw[->] (-.5,.5)--+(.01,.01);
\draw[->] (-.5,-.5)--+(.01,-.01);
\draw[->] (.5,-.5)--+(.01,.01);
\end{tikzpicture}
\caption{Geodesics in the plane $\{z_1=0\}$
\label{geodesics3D}}
\end{figure}
\label{fail_satur}

For the transition between $e_3$ and $-e_3$, we remark that the $x_2$-axis $\{z_1=z_3=0\}$ is the reunion of five solutions of \eqref{ode}: two stationary solutions $v\equiv \pm e_2$, one supported on $\{z_2<-1,\, z_1=z_3=0\}$, one on $\{-1<z_2<1,\, z_1=z_3=0\}$, and the other on $\{1<z_2,\, z_1=z_3=0\}$. In particular, by the Cauchy-Lipschitz Theorem, no solution can meet the line $\{z_1=z_3=0\}$ and there is no solution connecting $-e_3$ to $e_3$; hence, there is no global minimizer in $(\mathcal{P})$. Also note that the entropy $\Phi=\Phi^3$ satisfies $\Phi_1(e_3)-\Phi_1(-e_3)=0$ (because $\Phi_1$ is even in $z_3$) and $\mathrm{geod}^a_{W}(-e_3,+e_3)>0$ (because $\pm e_3$ are isolated zeros of $W$); therefore, the saturation condition imposed in Theorem~\ref{thm:rigid_sym} is not always satisfied.

\appendix
\section{Characterization of extremal pseudo-metrics}
Given a finite set $X$ with at least two elements, we define the set $\Delta(X)$ of pseudo-metrics on $X$ by
\[
\Delta(X)=\{\delta\in\R_+^{X\times X}\;:\; \forall x,y,z\in X,\, \delta(x,x)=0,\, \delta(x,y)=\delta(y,x),\, \delta(x,y)\leq \delta(x,z)+\delta(z,y)\}
\]
(see footnote at page \pageref{foot11}) and the set $\Delta_1(X)$ of normalized pseudo-metrics on $X$ by
\[
\Delta_1(X):=\left\{\delta\in\Delta(X)\;:\; {\sigma}(\delta)=1\right\},\quad\text{where }{\sigma}(\delta)=\sum_{x,y\in X}\delta(x,y).
\]
We look for those normalized pseudo-metrics $\delta$ which are extremal in $\Delta_1(X)$ in the following sense
\[
\forall t\in (0,1),\,\forall \delta_1,\delta_2\in\Delta_1(X),\, \big(\delta=t\delta_1+(1-t)\delta_2\Longrightarrow \delta_1=\delta_2\big).
\]

\begin{lemma}\label{extremalmetric}
A pseudo-metric $\delta$ is extremal in the compact convex set $\Delta_1(X)$ if and only if it is of the form $\delta=\frac1{\sigma(\delta_Y)}\delta_Y$ for some $Y\subset X$ (with $Y\neq\emptyset$ and $Y\neq X$), where $\delta_Y$ is defined for all $x,y\in X$ by
\[
\delta_Y(x,y)=
\begin{cases}
0&\text{if }(x,y\in Y)\text{ or }(x,y\notin Y),\\
1&\text{otherwise.}
\end{cases}
\]
In particular, any pseudo-metric $\delta\in \Delta(X)$ writes 
\(
\delta=\sum_{\emptyset \neq Y\subsetneq X}\lambda_Y \delta_Y,\text{ with } (\lambda_Y)_Y\subset\R_+.
\)
\end{lemma}
\begin{proof} The second part, i.e. the decomposition of any pseudo-metric in terms of extremal pseudo-metrics, is a consequence of the first part of the lemma and the Krein-Milman theorem: $\Delta_1(X)$ is the convex enveloppe of its extremal points. It remains to prove the characterization of extremal pseudo-metrics in $\Delta_1(X)$. 

\medskip
\noindent\textsc{Step 1: from pseudo-metrics to metrics.} For every $\delta\in\Delta(X)$, let $X/\delta$ be the set of equivalence classes in $X$ endowed with the equivalence relation $x\sim y$ defined by ($x\sim y$ iff $\delta(x,y)=0$). $X/\delta$ is endowed with the {\bf metric} $\overline{\delta}$ defined by 
$\overline{\delta}(\xi,\upsilon):=\delta(x,y)$ whenever $x\in\xi$ and $y\in\upsilon$ for every $\xi,\upsilon\in X/\delta$. Note that $\overline{\delta}$ is well defined, thanks to the triangle inequality on $\delta$, and that $X/\delta$ has at least two points when $\delta\not\equiv 0$. Moreover, the first part of the lemma is equivalent to ($\delta\in\Delta_1(X)$ is extremal iff $X/\delta$ has exactly two points). We use the following fact:
\begin{claim}
\label{cl1}
A pseudo-metric $\delta\in\Delta_1(X)$ is extremal if and only if the normalized metric 
\[
\hat{\overline{\delta}}:=\frac1{\sigma(\overline{\delta})}\overline{\delta}\in\Delta_1(X/\delta)
\]
is extremal.
\end{claim}
\begin{proof}[Proof of Claim \ref{cl1}] 
Indeed, first assume that $\delta$ is extremal in $\Delta_1(X)$. It is clear that any pseudo-metric $\alpha\in \Delta_1(X/\delta)$ induces a pseudo-metric $\hat \alpha_\ast\in\Delta_1(X)$ defined by 
$$\hat \alpha_\ast(x,y)=\frac1{\sigma(\alpha_\ast)} \alpha_\ast(x,y), \quad \textrm{ with } \quad \alpha_\ast(x,y)=\alpha(\overline{x},\overline{y}),$$ 
where $\overline{x}$ and $\overline{y}$ stand for the equivalence classes of $x$ and $y$ respectively in $X/\delta$. Assume that $\hat{\overline{\delta}}= t\alpha+(1-t)\beta$ with $\alpha,\beta\in \Delta_1(X/\delta)$ and $t\in (0,1)$. Then, for all $x,y\in X$, one has
\[
\delta(x,y)=\overline{\delta}(\overline{x},\overline{y})=t\sigma(\overline{\delta})\alpha(\overline{x},\overline{y})+(1-t)\sigma(\overline{\delta})\beta(\overline{x},\overline{y})=t\sigma(\overline{\delta})\alpha_*({x},{y})+(1-t)\sigma(\overline{\delta})\beta_*({x},{y});
\]
in particular, this yields \(1=\sigma(\delta)=t\sigma(\overline{\delta})\sigma(\alpha_*)+(1-t)\sigma(\overline{\delta})\sigma(\beta_*)\) and so
\[
\delta(x,y)=s\hat \alpha_\ast(x,y)+(1-s)\hat \beta_\ast(x,y),\quad\text{with }s=t \sigma(\overline{\delta}) \sigma(\alpha_\ast)\in (0,1).
\]
Since $\delta$ is extremal in $\Delta_1(X)$, one has $\hat \alpha_\ast=\hat \beta_\ast$ yielding $\sigma(\alpha_\ast)=\sigma(\beta_\ast)$ (because $\alpha,\beta\in \Delta_1(X/\delta)$) and finally, $\alpha=\beta$.

Conversely, assume that $\hat{\overline{\delta}}$ is extremal in $\Delta_1(X/\delta)$ and that $\delta=t\delta_1+(1-t)\delta_2$ with $\delta_1,\delta_2\in\Delta_1(X)$ and $t\in (0,1)$. For $i\in\{1,2\}$, $\delta_i$ induces a pseudo-metric $\widehat{\overline{\delta_i}}\in\Delta_1(X/\delta)$ defined by $\widehat{\overline{\delta_i}}=\frac1{\sigma(\overline{\delta_i})}\overline{\delta_i}$, where $\overline{\delta_i}(\overline{x},\overline{y})=\delta_i(x,y)$
and $\bar x$, $\bar y$ are the equivalence classes in $X/\delta$ of $x$ and $y$, respectively. It is clear that $\overline{\delta_1}$ and 
$\overline{\delta_2}$ are well defined since $\delta(x,y)=0$ implies that $\delta_1(x,y)=\delta_2(x,y)=0$. Moreover, the extremal pseudo-metric $\hat{\overline{\delta}}$ decomposes into 
$\hat{\overline{\delta}}=s\widehat{\overline{\delta_1}}+(1-s)\widehat{\overline{\delta_2}}$ with $s=\frac{t\sigma(\overline{\delta_1})}{\sigma(\overline{\delta})}=
1-(1-t)\frac{t\sigma(\overline{\delta_2})}{\sigma(\overline{\delta})}$, which by the same argument as above implies $\overline{\delta_1}=\overline{\delta_2}$ and so $\delta_1=\delta_2$.
\end{proof}

\medskip

\noindent\textsc{Step 2: case where $\delta$ is a metric.} By the preceding claim, it is enough to prove that a \textbf{metric} $\delta\in\Delta_1(X)$ (and not only a pseudo-metric) is extremal if and only if $X$ has $2$ points. For the first implication, if $X$ has two points, then any metric $\delta$ is extremal since $\Delta_1(X)$ is reduced to a single point. Conversely, we have:
\begin{claim}
\label{cl2}
Assume that $X=\{x_1,\dots,x_n\}$ has $n\geq 3$ distinct points and that $\delta\in\Delta_1(X)$ is a metric over $X$, then $\delta$ is not extremal.
\end{claim}
\begin{proof}[Proof of Claim \ref{cl2}] First, it is standard to see that the metric space $(X,\delta)$ is isometrically embedded in $\R^{n-1}$ endowed with the euclidean distance. In other words, there exists a subset $Y=\{y_1,\dots,y_n\}\subset \R^{n-1}$ of $n$ points (thought as a polytope) such that $\delta(x_i,x_j)=|y_i-y_j|$ for all $i,j\in\{1,\dots,n\}$. Let $D$ be the straight line $(y_1,y_2)$. Up to reorder the points $x_1,\dots,x_n$ and $y_1, \dots, y_n$, one may assume that 
\[
Y\cap D=\{y_1,\dots,y_m\}\quad\text{with }2\leq m\leq n, 
\]
and that $(y_1,\dots,y_m)$ is an increasing sequence in $D$ ordered by the relation ($x\leq y$ iff $(y-x)\cdot (y_2-y_1)\geq 0$). Given a parameter $h\in\R$ (not necessarily positive) with $|h|$ being small, we now build a small perturbation $\delta^h$ of the metric $\delta$ on $X$ such that $\delta^h(x_i,x_i):=0$ and for all $i,j\in\{1,\dots,n\}$ with $i<j$,
\[
\delta^h(x_i,x_j)=\delta^h(x_j,x_i):=
\begin{cases}
\delta(x_i,x_j)+h&\text{if }i=1\text{ and }2\leq j\leq m,\\
\delta(x_i,x_j)&\text{otherwise.}
\end{cases}
\]
Let us justify that $\delta^h$ is a metric, at least for small values of $|h|$. The idea is that $\delta^h$ corresponds to the euclidian metric in $\R^{n-1}$ by moving the point $y_1$ on the line $D$, keeping $y_2, \dots, y_m$ fixed and moving eventually the other points $y_{m+1}, \dots, y_n$; calling $y_1', \dots, y_n'$ these new points, then $\delta^h$ is a metric on $X$ iff such a (modified) polytope $y'_1, \dots, y'_n$ exists.
The only nontrivial fact is the triangle inequality. Consider a triangle $(x_i,x_j,x_k)$, with $i<j<k$. If $i\ge 2$ or ($i=1$ and $j>m$), the triangle inequality of $\delta$ in $(x_i,x_j,x_k)$ is trivial since $\delta$ is a metric. Otherwise, one has to show that for all $j\in\{2,\dots,m\}$ and $k>j$ (with $i=1$), one has
\begin{equation}
\label{toCheck}\tag{Tr}
|\delta^h(x_1,x_k)-\delta^h(x_k,x_j)| \leq \delta^h(x_1,x_j)\leq\delta^h(x_1,x_k)+\delta^h(x_k,x_j).
\end{equation}
Let us divide the proof of these inequalities according to whether $k>m$ or not, and considering the initial polytope $y_1, \dots, y_n$ corresponding to the metric 
$\delta$: 
\begin{itemize}[leftmargin=*]
\item
If $k>m$, \eqref{toCheck} is equivalent to 
\[
||y_1-y_k|-|y_k-y_j|| \leq |y_1-y_j|+h\leq |y_1-y_k|+|y_k-y_j|. 
\]
Since the triangle $(y_1,y_k,y_j)$ is not flat (as $y_1,y_j\in D$ but $y_k\notin D$), these inequalities are true for small values of $h$.
\item
If $k\leq m$, one has $1<j<k\leq m$ and, since $(y_1,\dots,y_m)\subset D$ is ordered in a monotonous way on the line $D$, one has $|y_k-y_1|=|y_k-y_j|+|y_j-y_1|$. Thus, \eqref{toCheck} is equivalent to the following trivial inequalities for $|h|\le |y_1-y_2|$:
\[
|y_j-y_1|+h\leq |y_1-y_j|+h\leq |y_j-y_1|+h+2|y_k-y_j|.
\]
\end{itemize}
Let us prove that $\delta$ is not extremal. Fix $h>0$ small enough so that $\delta^h$ and $\delta^{-h}$ are two metrics over $X$ and $2(m-1)h<\sigma(\delta)=1$. Since
\[
\delta=\frac 12(\delta^h+\delta^{-h})=\lambda_h\, \frac{\delta^h}{\sigma(\delta^h)}+\lambda_{-h}\, \frac{\delta^{-h}}{\sigma(\delta^{-h})}\quad\text{and}\quad \lambda_{-h}+\lambda_h=1,
\]
where $\lambda_{\pm h}:=\frac 12\sigma(\delta^{\pm h})=\frac 12\sigma(\delta)\pm (m-1)h>0$ and $\sigma(\delta)=1$, we conclude that $\delta$ is not extremal in $\Delta_1(X)$.
\end{proof}
\nd This proves completely Lemma \ref{extremalmetric}.
\end{proof}

\section{Calibration of extremal pseudo-metrics}
\begin{lemma}\label{extremalcalibration}
Let $X=\{x_0,\dots,x_d\}$ be an affine basis of $\R^d$. If $\delta=\delta_Y$ for some $Y\subset X$ then there exists a smooth compactly supported function $\varphi\in\mathcal{C}_c^\infty(\R^d,\R)$ such that:
\begin{enumerate}
\item
for all $x,y\in X$, $|\varphi(x)-\varphi(y)|=\delta(x,y)$,
\item
for all $x,y\in X$ with $(x,y\in Y)$ or $(x,y\notin Y)$, and for all $t\in [0,1]$, $\nabla\varphi(ty+(1-t)x)=0$,
\item
for all $x\in Y$, $y\in X\setminus Y$, and $t\in (0,1)$, $\nabla\varphi(ty+(1-t)x)$ and $y-x$ are positively collinear, i.e. $\nabla\varphi(ty+(1-t)x)=\lambda (y-x)$ with $\lambda=\lambda(x,y,t)>0$.
\end{enumerate}
\end{lemma}
\begin{proof}
When $\delta=0$, i.e. $Y=\emptyset$ or $Y=X$, one can take $\varphi=0$. We now assume that $Y\notin\{\emptyset,X\}$ and we reorder the affine basis $(x_0,\dots,x_d)$ in such a way that 
\be
\label{xnote}
Y=\{x_0,\dots,x_m\}\quad\text{with}\quad 0\leq m\leq d-1.
\ee
We shall construct the calibration $\varphi$ step by step. We first need to pick a nonnegative function $g\in\mathcal{C}^\infty(\R,\R_+)$ having the following properties:
\begin{itemize}[leftmargin=*]
\item
for all $t\leq 0$, $g(t)=0$ and for all $t\geq 1$, $g(t)=1$,
\item 
for all $t\in (0,1)$, $g'(t)>0$,
\item
for all $t\in\R$, $g(t)+g(1-t)=1$.
\end{itemize}
{\sc Step 1: smooth transitions $g^\lambda_{ij}$ between $g^\lambda_{ij}(x_i)=0$ and $g^\lambda_{ij}(x_j)=1$.} For every $i,j\in\{0,\dots,d\}$, $i\neq j$, $\lambda\in (0,1)$, and $z\in\R^d$, we set
\[
g^\lambda_{ij}(z)=
\begin{cases}
g\left(\lambda^{-1}p_{ij}(z)\right)g \left(1-\frac{ |z-x_i|-p_{ij}(z)}{\lambda_0 |z-x_i|}\right)&\text{if }z\neq x_i,\\
0&\text{if }z=x_i,
\end{cases}
\]
where $\lambda_0\in (0,1)$ will be fixed later and 
\[
p_{ij}(z)=(z-x_i)\cdot \frac{x_j-x_i}{|x_j-x_i|}.
\] 
If $\lambda<|x_i-x_j|$ for all $i,j$ with $i\neq j$ then the function $g^\lambda_{ij}$ performs a transition between $g^\lambda_{ij}(x_i)=0$ and $g^\lambda_{ij}(x_j)=1$ along the segment $[x_i,x_j]$. Moreover, $g^\lambda_{ij}$ is smooth in $\R^d$ and supported in the cone 
\[
\mathcal{C}^{\lambda_0}_{ij}:=\big\{z\in\R^d\;:\; p_{ij}(z)\ge (1-\lambda_0) |z-x_i|\big \}\supset \big\{tx_j+(1-t)x_i\;:\; t\ge 0\big \}.
\]
{\sc Step 2: partition of unity.} Let us pick $\lambda_0\in (0,1)$ small enough so that the balls $\overline{B}(x_i,\lambda_0)$ with $i\in\{0,\dots,d\}$ are disjoint and two distinct sets in $\{\mathcal{C}_{ij}^{\lambda_0}\cap \mathcal{C}_{ji}^{\lambda_0}\;:\; i<j\}$ can only meet at a point in $X=\{x_0,\dots,x_d\}$ (see Figure \ref{theXij}); define the family of functions $(\xi_{ij})_{i\leq j}\subset\mathcal{C}^\infty(\R^d,\R_+)$ by
\[
\xi_{ij}(z)=
\begin{cases}
g^{\lambda_0}_{ij}(z)g^{\lambda_0}_{ji}(z)&\text{if }0\le i<j\le d\\
g(1-\lambda_0^{-1}|z-x_i|)&\text{if }0\le i=j\le d
\end{cases}
\quad \text{for all }z\in\R^d.
\]
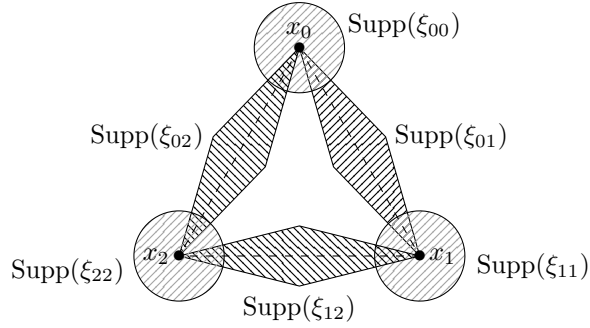
\begin{figure}
\centering
\begin{tikzpicture}[scale=1,>=triangle 45,scale=0.4]
\begin{scope}[scale=1,shift={(0,0,0)}, rotate around={30:(0,4,0)}]
   \begin{scope}[canvas is zy plane at x=0]
   \end{scope}
   
  \begin{scope}[canvas is xy plane at z=0]
  \draw[pattern=north east lines] (0,4) --(1,0) node[right] {$\mathrm{Supp}(\xi_{01})$} --(0,-4) --(-1,0)--cycle;
  \draw[pattern=north east lines, pattern color=gray!70] (0,-4) node[right] {$x_1$} node {$\bullet$} circle (1.5cm);
   \draw (1.2,-5.1) node[right] {$\mathrm{Supp}(\xi_{11})$};
  \draw[pattern=north east lines, pattern color=gray!70] (0,4) node[above] {$x_0$} node {$\bullet$} circle (1.5cm); 
  \draw (1.5,4) node[right] {$\mathrm{Supp}(\xi_{00})$};
  \draw[dashed] (0,4) -- (0,-4);
   \end{scope}
   \end{scope}
   \begin{scope}[scale=1,shift={(0,0,0)}, rotate around={-30:(0,4,0)}]
   \begin{scope}[canvas is zy plane at x=0]
   \end{scope}
   
  \begin{scope}[canvas is xy plane at z=0]
  \draw[pattern=north west lines] (0,4)--(1,0)--(0,-4)--(-1,0) node[left] {$\mathrm{Supp}(\xi_{02})$}--cycle;
    \draw[dashed] (0,4) -- (0,-4);
   \end{scope}
   \end{scope}
   \begin{scope}[scale=1,shift={(0,0,0)}, rotate around={-30:(0,4,0)}, rotate around={-60:(0,-4,0)}]
   \begin{scope}[canvas is zy plane at x=0]
   \end{scope}
   
  \begin{scope}[canvas is xy plane at z=0]
  \draw[pattern=north west lines] (0,4) --(1,0) node[below] {$\mathrm{Supp}(\xi_{12})$}--(0,-4) --(-1,0)--cycle;
    \draw[pattern=north east lines, pattern color=gray!70] (0,-4) node[left] {$x_2$} node {$\bullet$} circle (1.5cm); 
      \draw (0.5,-5.5) node[left] {$\mathrm{Supp}(\xi_{22})$};
    \draw[dashed] (0,4) -- (0,-4);
   \end{scope}
   \end{scope}
 \end{tikzpicture}
 \caption{The supports of the $\xi_{ij}$ when $d=2$\label{theXij}}
 \end{figure}
If $i<j$, then $\mathrm{Supp}(\xi_{ij})=\mathcal{C}_{ij}^{\lambda_0}\cap \mathcal{C}_{ji}^{\lambda_0}$ while $\mathrm{Supp}(\xi_{ii})=\overline{B}(x_i,\lambda_0)$. In particular, for every $z\notin X$, $\xi_{ij}(z)$ vanishes except at most for two choices of indices: either $\xi_{ij}$ alone, or ($\xi_{ij}$ and $\xi_{ii}$ with $i<j$) or ($\xi_{ij}$ and $\xi_{jj}$ with $i<j$). Moreover, one has
\[
\sum_{i\leq j}\xi_{ij}=1\quad\text{in }\cup_{i,j}[x_i,x_j].
\]
Indeed, for all $\ell\in\{0,\dots,d\}$, $\sum_{i\leq j}\xi_{ij}(x_\ell)=\xi_{\ell\ell}(x_\ell)=1$ and for all $z$ in the open segment $(x_i,x_j)$ with $i<j$, $\sum_{i\leq j}\xi_{ij}(z)=\xi_{ii}(z)+\xi_{jj}(z)+\xi_{ij}(z)$; if $\mathrm{dist}(z,\{x_i,x_j\})>\lambda_0$ then $\xi_{ii}(z)=\xi_{jj}(z)=0$ and $\xi_{ij}(z)=1$; otherwise, if for instance $p_{ij}(z)=|z-x_i|\leq \lambda_0$, then $\xi_{ij}(z)=g(\lambda_0^{-1}p_{ij}(z))$, $\xi_{jj}(z)=0$,
\(
\xi_{ii}(z)=g(1-\lambda_0^{-1}p_{ij}(z))=1-g(\lambda_0^{-1}p_{ij}(z)),
\) 
and thus $\sum_{i\le j}\xi_{ij}(z)=\xi_{ii}(z)+\xi_{ij}(z)=1$.

\medskip
\noindent{\sc Step 3: construction of the calibration $\varphi$.} Recalling the notation \eqref{xnote}, we set for all $z\in\R^d$:
\[
\varphi(z)=\sum_{0\le i\le m<j\le d}\Big(\big(\xi_{ij}(z)+\xi_{ii}(z)\big)\, g^{r_{ij}}_{ij}(z)-\xi_{jj}(z)\, g^{r_{ji}}_{ji}(z)\Big)+\sum_{m<i,j\le d} \xi_{ij}(z),
\]
where $r_{ij}:=|x_i-x_j|$. We observe that each term in both sums indexed by $i,j\in\{0,\dots,d\}$ is supported in $\mathrm{Supp}(\xi_{ij})$ (in particular, $\varphi$ is compactly supported) and we deduce that for all $z\in (x_i,x_j)$ with $i,j\in\{0,\dots,d\}$ and $i\neq j$,
\[
\begin{cases}
\varphi(z)=0&\text{if }i,j\leq m,\\
\varphi(z)=\xi_{ii}+\xi_{ij}+\xi_{jj}=1&\text{if }i,j>m,\\
\varphi(z)=\big(\xi_{ij}+\xi_{ii}\big)\, g^{r_{ij}}_{ij}-\xi_{jj}\, g^{r_{ji}}_{ji}+\xi_{jj}=g^{r_{ij}}_{ij}&\text{if }i\leq m<j,
\end{cases}
\]
since $g^{r_{ji}}_{ji}=1-g^{r_{ij}}_{ij}$ and $\xi_{ij}+\xi_{ii}+\xi_{jj}=1$ on $(x_i,x_j)$. In particular, $\varphi(x_i)=0$ for all $i\in\{0,\dots,m\}$ and $\varphi(x_j)=1$ for all $j\in\{m+1,\dots,d\}$ which proves the conclusion $1.$  in Lemma \ref{extremalcalibration}. The other properties required on $\nabla\varphi$, i.e. the conclusions 2. and 3. in Lemma \ref{extremalcalibration}, are a consequence of the fact that for all $z\in (x_i,x_j)$, with $i,j\in \{0,\dots,d\}$ and $i\neq j$:
\begin{itemize}[leftmargin=*]
\item
the derivatives of $\varphi$ in any direction orthogonal to $x_j-x_i$ vanish, i.e. $\nabla\varphi(z)$ and $x_j-x_i$ are collinear;
\item
if $i,j\leq m$ or $i,j>m$, $\varphi$ is constant in $(x_i,x_j)$, and so $\nabla\varphi(z)=0$;
\item
if $i\leq m<j$, $\varphi(\cdot)=g^{r_{ij}}_{ij}(\cdot)=g(r_{ij}^{-1}p_{ij}(\cdot))$ on $(x_i,x_j)$ and so $\nabla\varphi(z)=\lambda_{ij}(z) (x_j-x_i)$ with $\lambda_{ij}(z):=r_{ij}^{-1}g'(r_{ij}^{-1}p_{ij}(z))\frac{x_j-x_i}{|x_j-x_i|}>0$,
\end{itemize}
where we used that $g$ is increasing in $(0,1)$.
The proof is now complete.

\end{proof}

\paragraph{Acknowledgment.} The authors thank Beno\^it Merlet for very helpful discussions. R.I. acknowledges partial support by the ANR project ANR-14-CE25-0009-01.

\end{document}